\documentclass[12pt]{article}
\usepackage{latexsym, amsmath, amsfonts, bbm, amsthm, amssymb}
\usepackage[colorlinks=true]{hyperref}
\usepackage{times}
\usepackage{a4wide}
\usepackage{graphicx, tocloft}
\usepackage{mathrsfs}
\usepackage[dvipsnames]{xcolor}
\usepackage{url}
\usepackage{titletoc}

\newcommand{\EE}{\mathbb E}
\newcommand{\e}{\mathrm{e}}
\newcommand{\cov}{\mathrm{Cov}}
\newcommand{\var}{\mathrm{Var}}
\newcommand{\tr}{\mathrm{Tr}}
\newcommand{\Tr}{\mathrm{Tr}}
\newcommand{\ent}{\mathrm{ent}}
\newcommand{\RR}{\mathbb R}

\newcommand{\CC}{\mathbb C}
\newcommand{\PP}{\mathbb P}

\newcommand{\eps}{\varepsilon}
\newcommand{\vphi}{\varphi}
\newcommand{\cM}{\mathcal M}
\renewcommand{\tilde}{\widetilde} 
\newcommand{\cE}{\mathcal E}
\newcommand{\cF}{\mathcal F}
\newcommand{\cI}{\mathcal I}
\newcommand{\id}{\mathrm{Id}}

\newtheorem{theorem}{Theorem} %[section]
\newtheorem{definition}[theorem]{Definition}

\newtheorem{question}[theorem]{Question}
\newtheorem{lemma}[theorem]{Lemma}
\newtheorem{conjecture}[theorem]{Conjecture}

\newtheorem{proposition}[theorem]{Proposition}
\newtheorem{corollary}[theorem]{Corollary}
\newtheorem{fact}[theorem]{Fact}
\theoremstyle{definition}
\newtheorem{remark}[theorem]{Remark}

\newtheorem{example}[theorem]{Example}
\newtheorem{exercise}{Exercise}

\begin{document}

\title{Isoperimetric inequalities in high-dimensional \\ convex sets}
\author{Lecture notes by Bo'az Klartag and Joseph Lehec\footnote{Sections 1,3,4,5,9 correspond to lectures by B.K. while Sections 2,6,7,8 correspond to lectures by J.L.}}
%\\Mathematics Subject Classification: 45-02 (60-02)}

\date{Institut Henri Poincar\'e (IHP), Paris, May 21--24, 2024}
\maketitle
\bibliographystyle{plain}

\begin{abstract} 
These are lecture notes  
focusing on recent progress towards Bourgain's slicing problem and the isoperimetric conjecture proposed by Kannan, Lovasz and Simonovits (KLS).
\end{abstract}

\titlecontents{section}
[3em]{}{\contentslabel{2.5em}}{\hspace*{-2.5em}}{\titlerule*[1pc]{.}\contentspage}
\setcounter{tocdepth}{1}
\tableofcontents

\bigskip $ $ \bigskip

{\it Acknowledgement.}  We would like to thank the Institut Henri Poincar\'e for hosting the series of lectures. These notes are a direct outcome of those lectures. We are grateful to Esther Bou Dagher, Andreas Malliaris, and especially Cyril Roberto for the impeccable organization of this event. We also thank the 
anonymous referees for their careful reading of the manuscript
and accurate comments. 

\vfill
\pagebreak

\section{The Poincar\'e inequality} 
\label{sec_poincare}

Even if we were not hosted by an institution that honors Poincar\'e, a good starting point for these lectures would be the mathematical
inequality that carries his name. It was published by Poincar\'e in 1892--1896 in the case where the dimension is $2$ or $3$, and the measure $\mu$ is the uniform probability measure on a convex body $K$.

\medskip 
Recall that an absolutely continuous  measure $\mu$ in $\RR^n$ is {\it log-concave} if its density $\rho$ is log-concave, namely 
\begin{equation}  \rho(\lambda x + (1-\lambda) y) \geq \rho(x)^{\lambda} \rho(y)^{1-\lambda} \qquad \qquad (x,y \in \RR^n, 0 < \lambda < 1). \label{eq_455} \end{equation}
In general, a Borel measure $\mu$ in $\RR^n$ is log-concave 
if it is supported in an affine
subspace and has a log-concave density in this subspace.
The uniform probability measure on a convex body is log-concave, as well as all Gaussian measures. 

\begin{theorem}[``The Poincar\'e inequality'']
Let $K \subseteq \RR^n$ be a convex body, let $\mu$ be a log-concave probability measure on $K$. Then 
for any $C^1$-smooth function $f: K \rightarrow \RR$ with $\int_K f d \mu = 0$,
\begin{equation}  \int_K f^2 d \mu \leq C_P(\mu) \cdot \int_K |\nabla f|^2 \, d \mu \label{eq_1029} \end{equation}
where $C_P(\mu) \leq C_n \cdot Diam^2(K)$, and $C_n > 0$ depends only on the dimension $n$.  \label{thm_p}
\end{theorem}

Here $Diam(K) = \sup_{x,y \in K} |x-y|$ is the diameter of $K$ and $| \cdot |$ is the standard Euclidean norm in $\RR^n$.  Intuitively, the inequality says that if $f$ does not vary too wildly locally, i.e. controlled gradient, then it does not vary too much globally, i.e. bounded variance. 

\medskip 
For a historical account of the Poincar\'e inequality, see Allaire \cite{A}.
The Poincar\'e constant $C_P(\mu)$ of the probability measure $\mu$ is defined as the smallest number for which (\ref{eq_1029}) is valid
for all $C^1$-smooth functions $f$ with $\int f d \mu = 0$. 

\medskip 
The quantity $1 / C_P(\mu)$ is often referred to as the {\it spectral gap} of $\mu$, for reasons to be explained.  In 1960, Payne and Weinberger \cite{PW} found that for any $n$, the best possible value of the supposedly-dimensional constant $C_n$ is in fact
$$ C_n = \frac{1}{\pi^2}, $$
which does not depend on the dimension. We proceed with  an adaptation of the original proof by Poincar\'e, a proof which does not yield the optimal (in)dependence on the dimension, yet it suffices for some purposes.

\begin{proof}[Proof of Theorem \ref{thm_p}] Passing to a subspace if necessary, we may assume that the probability measure $\mu$ is absolutely continuous with 
a log-concave density $\rho: \RR^n \rightarrow [0, \infty)$, which vanishes outside $K$.
We express the variance as a double integral and use the fundamental theorem of calculus:
\begin{align*} 
\int_K f^2 d \mu & = \frac{1}{2} \int_{K} \int_K |f(y) - f(x)|^2 \, \mu(dx)\mu(dy) \\ & 
= \frac{1}{2} \int_{K} \int_K \left| \int_0^1 \nabla f((1-t) x + t y) \cdot (y-x) dt \right|^2 \, \mu(dx)\mu(dy)
\\ & \leq \frac{Diam^2(K)}{2} \int_{K} \int_K \int_0^1 \left|  \nabla f((1-t) x + t y) \right|^2 \rho(x) \rho(y) \, dt dx dy,
\end{align*}
where we used the inequality $|y-x| \leq Diam(K)$. Let us show that for any $0 \leq t \leq 1$,
\begin{equation}
 \int_{\RR^n} \int_{\RR^n} \left|  \nabla f((1-t) x + t y) \right|^2 \rho(x) \rho(y) \,  dx d y \leq C_{n,t} \int_{\RR^n} |\nabla f|^2  \, d \mu, 
 \label{eq_1552} \end{equation}
We integrate over $\RR^n$ now, but recall that the density $\rho$ vanishes outside $K$, so this does not make a difference.  
Our goal is to replace the product $\rho(x) \rho(y)$ in (\ref{eq_1552}) by some expression involving $\rho( (1-t) x + t y)$ and then apply a linear change of variables. Log-concavity will be handy here. We split the argument into two cases. 
If $t \approx 1/2$, then we will use the inequality
$$ \min \{ \rho(x), \rho(y) \} \leq \rho( (1-t) x + t y ) $$
that follows from the definition (\ref{eq_455}) of log-concavity. It implies that 
$$ \rho(x) \rho(y) \leq \rho( (1-t) x + t y ) \cdot \max \{ \rho(x), \rho(y) \} \leq \rho( (1-t) x + t y ) \cdot [\rho(x) + \rho(y)]. $$
Thus the integral in (\ref{eq_1552}) is at most
\begin{align*} 
 & \int_{\RR^n} \int_{\RR^n} \left|  \nabla f((1-t) x + t y) \right|^2 \rho( (1-t) x + t y ) \cdot [\rho(x) + \rho(y)] dx d y \phantom{aaaaaa} ``u = (1-t) x + t y''
 \\ & \int_{\RR^n} \int_{\RR^n} |\nabla f(u)|^2 \rho(u) \rho(x) \frac{du}{t^n} dx + \int_{\RR^n} \int_{\RR^n} |\nabla f(u)|^2 \rho(u) \rho(y) \frac{du}{(1-t)^n} dy
 \\ & = \left[ \frac{1}{t^n} + \frac{1}{(1-t)^n} \right] \int_{\RR^n} |\nabla f|^2 d \mu.
 \end{align*}
In the case where $t$ is not too close to $1/2$ we will use the inequality
 $$ \rho(x) \rho(y) \leq \rho( (1-t) x + t y ) \rho( t x + (1-t) y ) $$
 and change variables linearly via
 $$ u = (1-t) x + t y, \qquad v = t x + (1-t) y. $$
 Since $du_j \wedge dv_j = [(1-t)^2 - t^2] dx_j \wedge dy_j$ for $j=1,\ldots,n$, the integral in (\ref{eq_1552}) is bounded by 
 \begin{align*} & \int_{\RR^n} \int_{\RR^n} \left|  \nabla f((1-t) x + t y) \right|^2 \rho( (1-t) x + t y ) \rho( t x + (1-t) y ) dx d y 
 \\ & = \int_{\RR^n} \int_{\RR^n} | \nabla f(u)|^2 \rho(u) \rho(v) \frac{du dv}{|t^2 - (1-t)^2|^n} = \frac{1}{|1 - 2t|^n} \int_{\RR^n} |\nabla f|^2 d \mu. 
 \end{align*}
Thus the Poincar\'e inequality follows with 
$$ C_n  \leq \frac{1}{2} \int_0^1 \min \left \{ \frac{1}{t^n} + \frac{1}{(1-t)^n}, \frac{1}{|1-2t|^n} \right \} dt \leq C\cdot  \frac{3^n}{n}, $$
for some universal constant $C > 0$, where we separately consider the contribution of the intervals $[0,1/3],[1/3,2/3],[2/3,1]$ to the integral. 
\end{proof}

Throughout these lectures, we write $C, c, \tilde{C}, \tilde{c}, \bar{C}$ etc. to denote various positive universal constants whose value may change from one line to the next.
Consider the case where $\mu$
is the uniform probability measure on a domain $K \subseteq \RR^n$. Its Poincar\'e constant, sometimes denoted also by $C_P(K)$,
measures the {\it conductance} of $K$. It is large when $K$ has a bottleneck. 

\medskip 
Intuitively, it seems that convexity assumptions 
rule out many types of bottlenecks, possibly in high dimensions as well. 
Can we describe the Poincar\'e constant in terms of simple geometric characteristics of $K \subseteq \RR^n$, under convexity assumptions? 

\begin{conjecture}[Kannan-Lov\'asz-Simonovits \cite{KLS}] For any log-concave probability measure $\mu$ on $\RR^n$,
\begin{equation} \|\cov(\mu) \|_{op} \leq C_P(\mu) \leq C \cdot \|\cov(\mu) \|_{op} \label{eq_1651} \end{equation}
where $C > 0$ is a universal constant.
\end{conjecture}

Here $\| A \|_{op}$ is the operator norm of the symmetric matrix $A \in \RR^{n \times n}$, i.e., its maximal eigenvalue in absolute value,  and $\cov(\mu) \in \RR^{n \times n}$ is the inertia matrix 
or the {\it covariance matrix} of $\mu$. The $i,j$ entry of the matrix $\cov(\mu)$ is 
$$ \int_{\RR^n} x_i x_j \, \mu(dx) - \int_{\RR^n} x_i \, \mu(dx) \int_{\RR^n} x_j \, \mu(dx). $$
The covariance matrix is a symmetric, positive semi-definite matrix. If $X$ is a random vector with law $\mu$ and density $\rho$, we write $C_P(X) = C_P(\mu) = C_P(\rho)$ and $\cov(X) =\cov(\mu) =\cov(\rho)$. With this notation, the Poincar\'e inequality states that for any $C^1$-smooth function $f$, 
$$ \var(f(X)) \leq C_P(X) \cdot \EE |\nabla f(X)|^2. $$

\medskip Originally the conjecture by Kannan, Lov\'asz and 
Simonovits \cite{KLS} was formulated in terms of a Cheeger inequality 
rather than a Poincar\'e inequality, but the two formulations turn out to be 
equivalent. We shall return to this in the next section. For various perspectives 
on the KLS conjecture, we refer the reader 
to the monographs by Artstein-Avidan, Giannopoulos and Milman \cite{AAGM}
and by Brazitikos, Giannopoulos, Valettas and Vritsiou \cite{BGVV}, as well as to the survey papers by Ball \cite{Blegacy} and by Lee and Vempala \cite{LV_survey}.

\medskip 
We note that the left-hand side inequality in (\ref{eq_1651}) is a trivial fact: for any linear functional $f_{\theta}(x) = x \cdot \theta$ 
with $\theta \in S^{n-1} = \{ x \in \RR^n \, ; \, |x| = 1 \}$,
$$\cov(X) \theta \cdot \theta =\var( f_{\theta}(X) ) \leq C_P(X) \cdot \EE |\nabla f_{\theta}(X)|^2 = C_P(X), $$
and (\ref{eq_1651}) follows by taking the supremum over all $\theta \in S^{n-1}$.  Thus the KLS conjecture suggests that in the log-concave case, the Poincar\'e inequality
is saturated by linear functions, up to a universal constant.
\begin{exercise}[Tensorization] \label{exo_tens}
For $\mu,\nu$ probability measures on $\RR^n$ and $\RR^m$ respectively,  
\[ 
C_P ( \mu\otimes \nu ) = \max ( C_P (\mu), C_P (\nu) ) . 
\]
\end{exercise} 
Here are examples of log-concave measures for which we can compute 
the Poincar\'e constant. 

\begin{enumerate}
\item Consider the one-dimensional case, where $X$ is a random variable that is distributed uniformly in some interval of length $L$. Then,
$$\var(X) = \frac{L^2}{12} \qquad \text{and} \qquad  C_P(X) = \frac{L^2}{\pi^2}, $$
with the extremal function for the Poincar\'e inequality on $[0,\pi]$ being $f(x) = \cos x$.
\item Consider the case where $X$ is distributed uniformly in $K = [0,1]^n$. In this case, 
$$ Diam(K) = \sqrt{n} $$
while by the tensorization property of the Poincar\'e constant
%(see e.g.~\cite[section 4.3.]{BGL})
(see the exercise above)
$$ C_P(X) = \frac{1}{\pi^2} $$
and
$$\cov(X) = \frac{1}{12} \cdot \id. $$
We thus see that the diameter bound for the Poincar\'e constant is rather weak in high dimensions, even with the optimal, dimension-independent constant.
\item Suppose that $X$ is distributed uniformly in a Euclidean ball. The Euclidean unit ball $B^n = \{ x \in \RR^n \, ; \, |x| \leq 1 \}$ has volume
$$ \frac{\pi^{n/2}}{\Gamma(1 + n/2)} = \left( \frac{ \sqrt{2 \pi e} + o(1) }{\sqrt{n}} \right)^n, $$ which is a rather small number 
in high dimensions. In order to normalize the volume (or the covariance, or the Poincar\'e constant), we had better look at the random vector 
$X$ that is distributed uniformly in a Euclidean ball $K = \sqrt{n} \cdot B^n$. In this case,
$$ Diam(K) = 2 \sqrt{n}, \qquad\cov(X) = \frac{n}{n+2} \cdot \id. $$
The Poincar\'e constant of $X$ may be described using Bessel functions,
 and it has the order of magnitude of a universal constant, in accordance with the KLS conjecture. The Szeg\"o-Weinberger inequality \cite{SW1,SW2} states that among all uniform distributions on domains in $\RR^n$ of fixed volume, the Poincar\'e constant is minimized for a Euclidean ball.
\item Next we discuss the case where $X$  is a standard Gaussian random vector in $\RR^n$. Here,
$$\cov(X) = \id \qquad \text{and} \qquad C_P(X) = 1 . $$
Thus the Poincar\'e inequality in the Gaussian case is precisely saturated by linear functions. 
Furthermore, by
considering Hermite polynomials one can show the following: In the Gaussian case, a function nearly saturates the Poincar\'e inequality if and only if it is nearly a low-degree polynomial. 
Indeed, in one direction, if $f$ is a polynomial of degree at most $d$ in $n$ real variables then we can reverse the Poincar\'e inequality as follows:
$$ \EE |\nabla f(X)|^2 \leq d \cdot\var( f(X) ). $$
In the other direction, if $f$ is a smooth function with
$$ \EE |\nabla f(X)|^2 \leq R \cdot\var( f(X) ) $$
then the function $f$ may be approximated by a polynomial of bounded degree: For any $d \geq 0$ there exists a polynomial $P$ of degree at most $d$ such that 
$$ \EE |(f - P)(X)|^2 \leq \frac{R}{d+1} \cdot\var( f(X) ). $$
In fact, this polynomial $P$ is obtained by truncating the Hermite expansion of $f$.

\item Let us work in $\CC^n$ and consider the probability measure $\mu$ on $\CC^n$ with density 
$$ \prod_{j=1}^n \frac{e^{-|z_j|}}{2 \pi}. $$
The measure $\mu$ is a log-concave probability measure on $\CC^n$. Its covariance matrix is $$\cov(\mu) = 3 \cdot \id $$ and its Poincar\'e constant has the order of 
magnitude of a universal constant, in accordance with the KLS conjecture. 

\medskip The density of $\mu$ decays expoentially at infinity. Exponentially, but not faster; any log-concave probability density decays 
exponentially at infinity, yet the Gaussian density decays even faster. This reflects on spectral properties. In the exponential case
there are functions that nearly saturate the Poincar\'e inequality, and they do not necessarily resemble low-degree polynomials. For instance:

{\bf Claim:} For any holomorphic function $f: \CC^n \rightarrow \CC$ with $f \in L^2(\mu)$ and $\int f d \mu = 0$ (or equivalently, with $f(0) = 0$),
the Rayleigh quotient satisfies
\begin{equation}  \frac{1}{3} \leq \frac{\int_{\CC^n} |\nabla f|^2 \, d \mu}{\int_{\CC^n} |f|^2 \, d \mu} \leq \frac{1}{2}.  \label{eq_525} \end{equation}
Here is a proof for $n = 1$, which can be easily generalized for any dimension. It suffices to check the validity of (\ref{eq_525}) for monomials $z^k$, because of orthogonality relations. If $f(z) = z^k$ with $k \geq 1$ then,
$$ \| f \|_{L^2(\mu)}^2 = (2k+1)! $$
while
$$ \| f' \|_{L^2(\mu)}^2 = k^2 (2k-1)! $$
The ratio between the two is always between 4 and 6. We remark that by considering the real part of $f$, we see that (\ref{eq_525}) holds true for any pluri-harmonic function $f$, and in particular, when $n=1$ the relation (\ref{eq_525}) holds true for any harmonic function $f: \RR^2 \rightarrow \RR$ (thanks to A. Eskenazis for suggesting to add this remark).

%\medskip {\bf Exercise 1:} The function $f(z) = \max \{ |z_1|,\ldots, |z_n| \}$ is $1$-Lipschitz, yet it cannot be approximated by a polynomial of bounded degree in $L^2(\mu)$.
\end{enumerate}

\begin{exercise}[Subbaditivity] \label{exo_sub}
For two independent random vectors $X$ and $Y$ in $\RR^n$,
\[ 
C_P(X+Y) \leq C_P(X)+C_P(Y) . 
\]
\end{exercise}  

\subsection{Applications}

Poincar\'e's original motivation for his inequality was related to analysis of partial differential equations such as the {\it heat equation}. The motivation of Kannan, Lov\'asz and Simonovits in the 1990s came from algorithms based on Markov chains (MCMC)  for sampling and for estimating the volume of a high-dimensional convex body. Such tasks appear in linear programming. 
Another motivation for this research direction, that was put forth by Ball in the early 2000s and later jointly with Nguyen \cite{BN}, was the relation to Bourgain's slicing problem
discussed below. There are models in probability and statistical physics for which log-concavity and Poincar\'e inequalities are relevant. 
Let us describe here another application, related to the {\it Central Limit Theorem for Convex Sets} \cite{K_clt} from 2006. 

\medskip A random vector $X$
 in $\RR^n$ is {\it isotropic} or {\it normalized} if $\EE X = 0$ and 
$$\cov(X) = \id. $$
Any random vector with finite second moments can be made isotropic by applying an affine-linear transformation. The relation between Gaussian approximation and the Poincar\'e constant stems from the following:
\begin{enumerate}
	\item[(i)] The Poincar\'e inequality with $f(x) = |x|$ yields $\displaystyle\var(|X|) \leq C_P(X) $.
	Thus most of the mass of an isotropic random vector $X$ is contained in spherical shell 
	$$ \left \{ x \in \RR^n \, ; \, \sqrt{n} - 3 \sqrt{C_P(X)} \leq |x| \leq \sqrt{n} + 3  \sqrt{C_P(X)} \right \}, $$
	whose width has the order of magnitude of the square root of the Poincar\'e constant.
	\item[(ii)] Gaussian approximation principle (Sudakov \cite{sudakov}, Diaconis-Freedman \cite{DF}): When most of the mass of the isotropic random vector
	$X$ is contained  in a thin spherical shell, we have {\it approximately Gaussian marginals}.
\end{enumerate}

\medskip 
The following theorem is the current state of the art on Gaussian approximation 
under Poincar\'e inequality. We write $\sigma_{n-1}$ for the uniform probability measure on the unit sphere $S^{n-1}$. 

\begin{theorem}[Bobkov, Chistyakov, G\"otze {\cite[Proposition 17.5.1]{BCG}}]
Let $X$ be an isotropic random vector in $\RR^n$. Then there exists a subset $\Theta \subseteq S^{n-1}$
with $\sigma_{n-1}(\Theta) \geq 9/10$ such that any $\theta \in \Theta$,
$$ \sup_{t \in \RR} \left| \PP( X \cdot \theta \leq t ) \, - \, \frac{1}{\sqrt{2 \pi}} \int_{-\infty}^t e^{-s^2/2} ds \right|
\leq \frac{C \log n}{n} \cdot C_P(X), $$
where $C > 0$ is a universal constant.
\label{thm1}
\end{theorem}

We do not know whether the logarithmic factor in Theorem \ref{thm1} is necessary. It is currently known that $C_P(X) \leq C \cdot \log n$ for an isotropic, log-concave random vector $X$ in $\RR^n$, see \cite{K_root}. 
Consequently Theorem \ref{thm1} yields good error estimates in the Central Limit Theorem for Convex sets, and more generally for log-concave measures.

\medskip 
If all we know about the Poincar\'e constant is the diameter bound, then even in the case of the cube we would be off by a factor of $n$, and we would not obtain any non-trivial bound for the Central Limit Theorem for Convex sets. Thus in high dimensions it is necessary to refine the diameter bound, as suggested in the KLS conjecture.

\medskip What techniques can we use to this end, techniques that go beyond change of variables, Fubini theorem, and the Cauchy-Schwartz inequality
used above? High-dimensional convex geometry is a playground for various geometric and analytic ideas that transcend the field of convexity. Any list of approaches that have proven useful 
to convexity must include convex localization,
optimal transport, curvature and the Bochner formula, semigroup tools, geometric measure theory, stochastic localization and complex analysis. In these lectures we explore only
some of these directions. 

\subsection{1D log-concave distributions}

Before going on to study methods for high dimensions, let us briefly discuss the one-dimensional case. What do log-concave densities look like in one dimension? 

\begin{proposition}[``How to think on 1D log-concave random variables''] 
Let $X \in \RR$ be a log-concave random variable with density $\rho$ which is isotropic. Then for any $x \in \RR$,
$$ c' \mathbbm 1_{\{|x| \leq c''\}} \leq \rho(x) \leq C e^{-c |x|} $$
where $c', c'', c, C > 0$ are universal constants.
\label{thm_1526}
\end{proposition}

\begin{exercise} 
Prove this proposition. \\
\textit{Hint:} for the upper bound, if $\rho(b) < \rho(a) / 2$ for some $a < b$, then $\rho$ decays exponentially and in fact $\rho(x) \leq \rho(b) 2^{-x / (b-a)}$ for all $x > b$. As for the lower bound, it's enough to show that $\rho(x) > c'$ for some $x > c''$ and for some $x < -c''$.
\end{exercise}  

\begin{corollary}[``reverse H\"older inequalities''] For any isotropic, log-concave, real-valued random variable $X$ and any $p > -1$,
\begin{equation}  c \cdot \min \{ p+1, 1 \} \leq \| X \|_p = (\EE |X|^p)^{1/p} \leq C (|p| + 1), \label{eq_1748} 
	\end{equation}
where $c, C > 0$ are universal constants. \label{cor_1749} 
\end{corollary}

The case $p = 0$ in (\ref{eq_1748}) is interpreted by continuity, i.e., 
$$ \| X \|_0 = \exp(\EE \log |X|). $$
This is not a norm, yet a nice feature is its multiplicativity: for any random variables $X$ and $Y$, possibly dependent,
$$ \| X Y \|_0 = \| X \|_0 \| Y \|_0. $$

\begin{proof}[Proof of Corollary \ref{cor_1749}] Begin with the inequality on the right-hand side. By the monotonicity of $p \mapsto \| X \|_p$, it is enough to look at $p > 0$. In this case, 
$$ \| X \|_p^p = \int_{-\infty}^{\infty} |t|^p \rho(t) dt \leq C \int_{-\infty}^{\infty} |t|^p e^{- c |t|} dt = \frac{2C}{c^{p+1}} \Gamma(p+1) \leq (\tilde{C} p)^p. $$
where we used the fact that for integer $p$, we have $\Gamma(p+1) = p! \leq p^p$. For the lower bound, by monotonicity it suffices to look at $p < 0$. Setting $q = -p \in (0,1)$ we have
$$ \EE \frac{1}{|X|^q} \leq C \int_{-\infty}^{\infty} \frac{1}{|t|^q} e^{- c |t|} dt \leq \frac{C'}{1-q} $$
and hence
\[ 
\| X \|_p = \left( \EE \frac{1}{|X|^q} \right)^{-1/q} \geq \left( C' (1-q) \right)^{1/q}
\geq \tilde{C} (1 - q). \qedhere 
\]
\end{proof}

We proceed to discuss the isoperimetric profile of a log-concave distribution in one dimension. Bobkov \cite{bobkov_extremal} shows that for a probability density $\rho$ on the real line,
\begin{equation} \rho \textrm{ is log-concave} \qquad \Longleftrightarrow \qquad \rho \circ \Phi^{-1}: [0,1] \rightarrow (0, \infty) \ \textrm{is concave} \label{eq_1828} \end{equation}
where $\Phi(x) = \int_{-\infty}^x \rho(t) dt$ and $\Phi^{-1}(y) = \inf \{ x \in \RR \, ; \, \Phi(x) \geq y \}$. Once stated, (\ref{eq_1828}) is not difficult to prove.  It follows from (\ref{eq_1828}) that the function
$$ I(x) = \min \left \{ \rho \circ \Phi^{-1}, \rho \circ (1 - \Phi)^{-1} \right \} $$
is concave. Write $\mu$ for the measure whose density is $\rho$, and note that 
$$ I(x) = \min \{ \rho( \partial H ) \, ; \, H \ \text{is a ray with} \ \mu(H) = x \} $$
Since the boundary $\partial H$ is a singleton as $H$ is a ray, in this case we abbreviate $\rho(\partial H) = \rho(a)$ if $\partial H = \{ a \}$. The following Proposition 
by Bobkov implies that the concave function $I$ is the {\it isoperimetric profile} of the probability density $\rho$.

\medskip We prefer to discuss isoperimetry through $\eps$-neighborhoods. 
For $\eps > 0$ and a subset $A \subseteq \RR$ we write $A_{\eps} = \{ x \in \RR\, ; \, \inf_{y \in A} |x-y| < \eps \}$ for its $\eps$-neighborhood. 
We remark that analogously to (\ref{eq_1828}), the log-concavity of $\rho$ implies that the function $x \mapsto \Phi(\Phi^{-1}(x) + \eps)$ is
concave. This shows that the function 
$$ I_{\eps}(x) = \min \{ \mu(H_{\eps}) \, ; \, H \ \text{is a ray with} \ \mu(H) = x \} $$
is a concave function of $x \in [0,1]$.

\begin{proposition}[Bobkov \cite{bobkov_extremal}] Let $\mu$ be a log-concave probability measure on $\RR$ with density $\rho$. Fix $0 < p < 1, \eps > 0$. Then among all Borel subsets $A \subseteq \RR$ with $\mu(A) = p$, the infimum of $\mu( A_{\eps})$ is attained for a half line.
	\label{prop_1637}
\end{proposition}

\begin{proof}[Sketch of Proof] It suffices to show that half lines are better than finite unions of intervals. How can we deal with a subset $A$ that is a finite union of intervals? Using the following claim. For $a \in \RR$ with $\mu([a, \infty)) > p$ consider the unique interval $J(a) = (a,b)$ such that $\mu(J(a)) = p$. The claim is that the function 
	$$ a \mapsto \mu( J(a)_{\eps} ) $$
	is unimodal, thanks to log-concavity (i.e., the function is increasing and then decreasing). Again, once stated this is not too difficult to prove. 
	Given this claim, one may fix all intervals in $A$ but one, and then move the remaining one around and expand and shrink it so as to preserve the total $\mu$-measure. 
	It follows that gluing this interval to one of the sides cannot increase the $\mu$-measure of the $\eps$-neighborhood.
\end{proof}
Combining this with Proposition~\ref{thm_1526} one gets the following 
Cheeger type isoperimetry for 1D log-concave measures. 
\begin{corollary} Let $\mu$ be an isotropic, log-concave probability measure on $\RR$ and let $\eps, p \in (0,1)$. 
	Then for any Borel set $S \subseteq \RR$ with $\mu(S) = p$,
	$$ \mu(S_{\eps} \setminus S) \geq c \cdot \eps \cdot \min \{p, 1-p \} $$
	where $c > 0$ is a universal constant. 	
	\label{cor_1844}
\end{corollary}

\begin{exercise}
Fill in the details in the proofs of Proposition \ref{prop_1637} and Corollary \ref{cor_1844}.
\end{exercise} 

\pagebreak 

\section{Related functional inequalities} 

\subsection{Cheeger's inequality} 
\label{sec_cheeger}
Let $\mu$ be a probability measure on $\RR^n$, or more 
generally on some metric space $(X,d)$ equipped with its 
Borel $\sigma$-field. The isoperimetric problem for 
$\mu$ asks the following questions: Among sets of given measure, which 
sets have minimal perimeter? There are several possible notions of 
perimeter. For our purposes, the most convenient one is 
the exterior Minkowski content, defined as follows: for every 
measurable subset $A$ of the ambient space we let  
\[ 
\mu_+ ( A ) = \liminf_{\eps \to 0} \frac{ \mu ( A_\eps \backslash A ) } \eps . 
\]
where $A_\eps$ is the $\eps$-neighborhood of $A$, namely 
the set of points whose distance to $A$ is at most $\eps$. 
Proposition~\ref{prop_1637}, at the end of the previous section, 
shows in particular that for 1D log-concave measures, half-lines solve 
the isoperimetric problem. In higher dimension though,
the exact answer to the isoperimetric problem is only known in 
a handful of very specific cases. For instance, for the Haar measure 
on the sphere equipped with the geodesic distance, 
spherical caps (i.e. geodesic balls) are the solution. 
This is usually attributed to P. L\'evy (1922). The answer is also known on Gauss space, and this time affine half-spaces solve the isoperimetric problem. This was proved in 1975 by Sudakov and Tsirelson~\cite{ST}, and independently 
by Borell~\cite{borell-gauss}. In general solving exactly the isoperimetric 
problem is hopeless and we content ourselves 
with a more modest task, such as finding lower bounds on the perimeter 
of a set $A$ in terms of its measure. When this lower bound is linear, 
we say that $\mu$ satisfies Cheeger's inequality. 
\begin{definition}
We say that $\mu$ satisfies Cheeger's inequality if there is a 
constant $C$ such that 
\begin{equation}\label{eq_cheeger} 
\min ( \mu ( A ), 1-\mu (A) ) \leq C \mu_+ (A) , 
\end{equation} 
for every measurable set $A$. 
The smallest $C$ such that this holds true is called 
the Cheeger constant, and we denote it $\psi_\mu$ below.  
\end{definition} 
For instance, Corollary~\ref{cor_1844} from the previous 
section shows that the Cheeger constant of an isotropic 
log-concave measure in 1D is bounded above by a universal constant. 
\begin{remark} 
It is more common to put the constant in the left-hand side 
of the inequality~\eqref{eq_cheeger} rather than in the right-hand side. 
So our Cheeger constant is the reciprocal 
of the \emph{usual} Cheeger constant. 
\end{remark} 

Cheeger's inequality can be seen as an $L^1$-Poincar\'e 
inequality. 
\begin{lemma} \label{lem_cheeger}
Inequality~\eqref{eq_cheeger} is equivalent to the following:
\begin{equation}
\label{eq_cheeger2} 
\min_{c\in \RR} \int_X \vert f -c \vert \, d\mu \leq C
\int_X \vert\nabla f \vert \, d\mu  , 
\end{equation}
for every Lipschitz function $f$. 
\end{lemma} 
\begin{remark}
In the right-hand side the quantity $\vert \nabla f(x)\vert$ 
should be interpreted as the local Lipschitz constant of $f$, 
namely 
\[ 
\vert \nabla f (x) \vert = \limsup_{y\to x} \frac{ \vert f(x) -f(y) \vert }{ d(x,y) }  . 
\]  
This only make sense in a metric space with no isolated points. 
Actually we will only investigate the case $X = \RR^n$ equipped 
with its usual Euclidean metric from now on.  
\end{remark} 
\begin{remark} 
It is well known that the infimum in the left-hand side is attained at any median 
for $f$, i.e. any real $c$ such that both $\mu ( f \leq c )$ and 
$\mu ( f \geq c )$ are at least $1/2$. 
\end{remark}
\begin{proof} 
We only give a proof sketch, and refer to Bobkov and Houdr\'e~\cite{BH} (for instance) for more details. The derivation of~\eqref{eq_cheeger2} from \eqref{eq_cheeger} relies on the co-area formula: for any Lipschitz $f$ we have 
\[ 
\int_X \vert \nabla f \vert \, d\mu \geq \int_\RR \mu^+ ( f > t ) \, dt . 
\]
In most cases this inequality is actually an equality, but we only need this inequality, which admits a soft proof, again see~\cite{BH}. 
Applying Cheeger's inequality to the right-hand side then yields~\eqref{eq_cheeger2}. For the converse implication, given a set $A$, we apply~\eqref{eq_cheeger2} to some suitable Lipschitz approximation of the indicator function of $A$. A bit more precisely, we pick $\eps_n \to 0$ such that 
\[ 
\lim \frac{ \mu ( A_{\eps_n}\backslash A)}{\eps_n} \to \mu_+ (A), 
\] 
we pick another positive sequence $(\delta_n)$ tending to $0$ 
(for instance $\delta_n = 1/n$) and we observe that the sequence $(f_n)$ given by 
\[ 
f_n = \left(1 - \frac 1{(1-\delta_n)\eps_n} \cdot d(x,A_{\delta_n \eps_n}) \right)_+
\] 
satifies $0\leq f_n \leq 1$ for every $n$, $f_n \to \mathbbm 1_{\overline A}$ pointwise, and $\limsup \int \vert \nabla f_n \vert \, d\mu \leq \mu_+ (A)$. Applying~\eqref{eq_cheeger2} to $f_n$ and letting $n$ tend to $+\infty$ 
yields~\eqref{eq_cheeger} after some computation.  
\end{proof} 
From this version of Cheeger's inequality it is relatively straightfoward 
to see that Cheeger's inequality is stronger than the Poincar\'e inequality. 
Recall from Section \ref{sec_poincare} that we say that $\mu$ satisfies Poincaré
if there is a constant $C$ such that 
\[ 
\var_\mu ( f  ) \leq C \int_{\RR^n} \vert \nabla f \vert^2 \, d\mu 
\]
for every Lipschitz function $f$.  
Also we let $C_P (\mu)$ be the best constant $C$ such that 
this holds true.  
\begin{proposition}[Cheeger 1970]
Let $\mu$ be a probability measure on $\RR^n$ satisfying the Cheeger inequality. 
Then $\mu$ satisfies Poincar\'e, and we have 
\[ 
C_P (\mu) \leq 4 \psi_\mu^2 . 
\]
\end{proposition} 
\begin{remark} Maybe it is unfortunate but our convention for the Cheeger constant and Poincar\'e constant do not have the same homogeneity. The Cheeger constant of a probability measure on $\RR^n$ is $1$-homogeneous, if we scale 
$\mu$ by a factor $\lambda$ then the Cheeger constant is multiplied by $\lambda$. 
One the other hand the Poincar\'e constant is $2$-homogeneous. 
\end{remark} 
\begin{proof} 
Assume that $f$ is Lipschitz and bounded, and has its median at $0$. Applying~\eqref{eq_cheeger2} to $f_+^2$
we get 
\[
\int_{\RR^n} f_+^2 \, d\mu 
 \leq \psi_\mu \int_{\RR^n} \vert\nabla f_+^2 \vert \, d \mu 
 = 2 \psi_\mu \int_{\RR^n} f_+ \vert \nabla f_+ \vert \, d\mu . 
\] 
The Cauchy-Schwarz inequality then yields 
\[ 
\int_{\RR^n} f_+^2 \, d\mu 
\leq 4\psi_\mu^2  \int_{\RR^n} \vert\nabla f_+ \vert^2  \, d \mu 
= 4 \psi_\mu^2   \int_{\RR^n} \vert\nabla f \vert^2 \mathbbm 1_{\{ f >0\}}  \, d \mu .
\] 
We can do the same with $f_-$ and adding up the two inequalities 
yields the result. 
\end{proof} 
The converse inequality is not true in general, one can cook up 
examples on the line. However it turns out that if we restrict to log-concave 
measures then the converse is true. This is a result of Buser~\cite{B} from 1982, 
to which we will come back later on in this section. 

\subsection{Semigroup tools}
\label{sec_semigroup}
Let $\mu$ be a probability measure on $\RR^n$. We do not need log-concavity
for now but let us assume that $\mu$ is supported on the whole space and has a smooth density $\rho$.
Letting $V = -\log \rho$ be the potential of $\mu$, the Laplace 
operator associated to $\mu$ is the differential 
operator given by 
\[ 
L_\mu = \Delta -  \nabla V \cdot \nabla , 
\] 
initially defined on the space of compactly supported 
smooth functions. 
For such functions, an integration 
by parts gives 
\[ 
\int_{\RR^n} (L_\mu f) g \, d\mu = - 
\int_{\RR^n}  \nabla f \cdot \nabla g \, d\mu. 
\]   
This shows in particular that $L_\mu$ is symmetric and that $-L_\mu$ 
is a monotone (unbounded) operator on $L^2 (\mu)$. Moreover this operator 
is known to be essentially self-adjoint, in the sense that its minimal
extension is self-adjoint. By a slight abuse of notation we still call $L_\mu$ 
this extension. A bit more explicitly, we call $\mathcal D$
the space of functions $f\in L^2(\mu)$ for which there exists 
a sequence $(f_n)$ of smooth compactly supported functions 
such that $f_n \to f$ and $(L_\mu f_n)$ converges.
The limit of $L_\mu f_n$ does not depend on the choice 
of the converging sequence $(f_n)$ 
(this is an immediate consequence of the symmetry of $L_\mu$) 
and we set $L_\mu f = \lim L_\mu f_n$. 
The fact that this new $L_\mu$ is self adjoint
is not quite immediate, not every monotone operator is essentially self 
adjoint. This has to do with elliptic regularity, we refer 
to~\cite[Corollary 3.2.2]{BGL} for the details. 
From the integration by parts above we can see that if $(f_n)$ and $(L_\mu f_n)$
converge then also $\nabla f_n$ converges. 
This means that the domain $\mathcal D$ 
contains $H^1 (\mu)$ and that the integration by parts 
$\langle L_\mu f ,g \rangle = - \langle \nabla f , \nabla g\rangle$ 
remains valid for every $f,g$ in the domain. 
Here the inner product is the one from $L^2 ( \mu)$, and when we apply it to tensors it has to be interpreted coordinate wise. Being self-adjoint 
and monotone (negative) the operator $L_\mu$ admits a spectral decomposition 
\begin{equation}\label{eq_spectral}
L_\mu = - \int_0^\infty \lambda \, d E_\lambda . 
\end{equation}
The semigroup associated to $L_\mu$ is 
then defined as 
\[ 
P_t = \e^{tL_\mu} = \int_0^\infty \e^{-t\lambda} \, d E_\lambda . 
\] 
For fixed $t$ the operator $P_t$ is a self-adjoint bounded operator in $L^2 (\mu)$
and we have the semigroup property $P_t \circ P_s = P_{t+s}$. If $f$ is 
a fixed function of $L^2 (\mu)$ the function $F(t,x) = P_t f ( x)$ is 
the solution to the parabolic equation 
\[ 
\begin{cases}
F(0,\cdot ) = f  \\
\partial_t F = L_\mu F ,
\end{cases} 
\] 
at least in a weak sense. \\
We now move on to the probabilistic representation of the semigroup $(P_t)$.
Consider the diffusion $(X_t)$ given by 
\begin{equation}\label{eq_SDE}
d X_t = \sqrt 2 \cdot d W_t - \nabla V ( X_t) \, dt , 
\end{equation}
where $(W_t)$ is standard Brownian motion. 
Then $(X_t)$ is a Markov process, and $(P_t)$ is the
corresponding semigroup. Namely for every test function $f$ 
we have 
\[ 
P_t f (x) = \EE_x f(X_t) 
\] 
where the subscript $x$ next to the expectation denotes 
the starting point of $(X_t)$.
This allows to prove inequalities for the semigroup 
$(P_t)$ using probabilistic techniques. The next result is 
considered folklore, see e.g.~\cite[section 9.9.]{BGL} for 
some historical perspectives. 
\begin{lemma} If $\mu$ is log-concave then Lipschitz functions 
are preserved along the semigroup, and moreover 
$\Vert P_t f\Vert_{\rm Lip} \leq \Vert f\Vert_{\rm Lip}$ for 
every $f$ and every $t>0$. 
\end{lemma}

\begin{proof} 
Let $x,y\in \RR^n$, and let $(X^x_t)$ 
and $(X^y_t)$ be two solutions of the SDE~\eqref{eq_SDE} using the 
same Brownian motion, but starting at two different points $x$ and $y$. 
This is called parallel coupling. Then the process $(X^x_t-Y^x_t)$ 
is an absolutely continuous function of $t$ (the Brownian part cancels 
out). Moreover, thanks to the convexity of $V$, 
\[ 
\frac d{dt} \vert X^x_t-X^y_t\vert^2  
= - 2 (X^x_t-X^y_t)\cdot (\nabla V(X^x_t) -\nabla V(X^y_t)) \leq 0 . 
\] 
So the distance $\vert X^x_t-X^y_t\vert$ is almost surely decreasing. 
Therefore its expectation is also decreasing, and in particular 
\[ 
\EE \vert X^x_t - X^y_t \vert \leq \vert x-y\vert . 
\]   
Now suppose $f$ is a Lipschitz function. 
Then from the previous inequality we get 
\[ 
\vert P_t f(x) - P_t f(y) \vert
= \vert \EE f(X^x_t) - \EE f(X^y_t) \vert
\leq \EE \vert f(X^x_t) -f(X^y_t) \vert 
\leq \Vert f\Vert_{\rm Lip} \cdot \vert x-y\vert , 
\]  
which is the result. 
\end{proof} 
The next result seems to be due to Varopoulos~\cite{varo}. 
\begin{proposition}\label{prop_reg} Suppose $\mu$ is log-concave. 
Then for every bounded function $f$ and every $t>0$ 
the function $P_t f$ is Lipschitz and
moreover 
\[ 
\Vert P_t f \Vert_{\rm Lip} \leq \frac 1{\sqrt t} \cdot \Vert f\Vert_\infty . 
\]  
\end{proposition} 

\begin{proof} 
Again we use a coupling argument, see~\cite{ledouxSG} for an alternate argument 
using only analytic tools. Suppose that 
$f$ is a bounded function. Fix $x,y\in\RR^n$, and let 
$(X^x_t)$ and $(X^y_t)$ be two processes solving the SDE~\eqref{eq_SDE} 
initiated at $x$ and $y$ respectively. Then 
\begin{equation}
\label{eq_mirror} 
\vert P_t f(x) - P_t f(y) \vert
\leq \EE \vert f(X^x_t) -f(X^y_t) \vert \leq 2 \Vert f \Vert_\infty \cdot 
\PP ( X^x_t \neq X^y_t ). 
\end{equation} 
It remains to choose a coupling for which the right-hand side is 
small. Parallel coupling is awful here, as it actually prevents $X^x_t$ and $X^y_t$ 
from meeting. Instead, we choose the Brownian increment for $X^y_t$ to be the reflection of that of $X^x_t$ with respect to the hyperplane $(X^x_t-X^y_t)^\perp$.
If $(W_t)$ is the Browian motion for $X_t^x$, the equation for 
$X_t^y$ is thus 
\[ 
d X^y_t = \sqrt 2 \cdot \left( \id - 2 v_t^{\otimes 2} \right) dW_t 
- \nabla V ( X^y_t ) \, dt  
\] 
where $(v_t)$ is the unit vector $(X^x_t-X^y_t) / \vert X^x_t -X^y_t \vert$.
Actually we do so until the first time (denoted $\tau$) when the two processes
meet. After time $\tau$ we just set $X_t^y = X_t^x$. 
We will not justify properly here why this is well defined, 
but this coupling technique, usually referred to as \emph{mirror coupling}, 
is a relatively standard tool,
see for instance~\cite{rogers}. 
It\^o's formula shows that 
up to the coupling time $\tau$ the equation for the distance between the two processes is 
\[ 
d \vert X^x_t - X^y_t \vert = 2 \sqrt 2 v_t \cdot d W_t - v_t \cdot ( \nabla V (X^x_t ) - \nabla V (X_t^y) ) \, dt . 
\] 
It\^o's term vanishes because the Brownian increment takes place in a direction 
where the Hessian matrix of  
the norm vanishes. Once again, in the log-concave 
case the second term from the right hand side is negative. Notice also 
that $B_t:= \int_0^t v_s\cdot dW_s$ is 
a standard (one dimensional) Brownian motion. Therefore up to the 
coupling time $\tau$ we have 
\[ 
\vert X^x_t - X^y_t \vert \leq \vert x-y\vert + 2 \sqrt 2 B_t , 
\]
where $(B_t)$ is some standard one dimensional Brownian motion. Therefore
\[ 
\PP ( X^x_t \neq X^y_t ) = \PP ( \tau >t  ) \leq 
\PP \left( \forall s\leq t \colon  B_s >- \frac{\vert x-y\vert}{2 \sqrt 2} \right) . 
\] 
By the reflection principle for the Brownian motion 
\[ 
\PP \left( \exists s\leq t \colon  B_s \leq - \frac{\vert x-y\vert}{2 \sqrt 2} \right) 
=  2\cdot  \PP \left( B_t \leq  - \frac{\vert x-y\vert}{2\sqrt 2} \right)
= \PP \left( \vert g \vert \geq   \frac{\vert x-y\vert}{2\sqrt{2t}} \right) 
\] 
where $g$ is a standard Gaussian variable. Hence the inequality 
\[ 
\PP ( X^x_t \neq X^y_t ) \leq \Psi \left( \frac{\vert x-y\vert}{2\sqrt{2t}} \right) ,
\] 
where $\Psi(r) = (2/\pi)^{1/2} \int_0^r \e^{-u^2/2} \, du$ 
is the distribution function of $\vert g\vert$.
Recalling~\eqref{eq_mirror} and taking the supremum over $x,y$ gives
\[ 
\Vert P_t f\Vert_{\rm Lip} 
\leq \frac 1{\sqrt{2t}} \cdot \sup_{a>0} \left\{ \frac{ \Psi(a) }a \right\} \cdot 
\Vert f\Vert_\infty . 
\] 
The expression inside the sup is decreasing, so the sup 
equals the limit as $a$ tends to $0$, which is $(2/\pi)^{1/2}$.  
We thus get the desired inequality (even with a better 
constant than announced). 
\end{proof} 
The next corollary is taken from Ledoux~\cite{ledoux-buser}.
\begin{corollary} \label{cor_ledoux} 
If $\mu$ is log-concave, then for every locally Lipschitz function $f$ 
we have
\[ 
\Vert f- P_t f \Vert_{L^1(\mu)} \leq  2 \sqrt t \cdot \Vert \vert \nabla f \vert \Vert_{L^1(\mu)} . 
\] 
Also for every measurable set $A$ we have 
\[ 
 \mu (A) (1-\mu (A)) = \var_\mu ( \mathbbm 1_A)  \leq \sqrt{2t} \cdot \mu^+ (A ) + \var_\mu (P_t \mathbbm 1_A ) .
\] 
\end{corollary} 
\begin{proof} Let $f$ be a Lipschitz function and 
$g$ be a smooth bounded function. Using the fact that the semigroup is self adjoint, 
and the integration by part formula, we get  
\[ 
\langle f - P_t f  ,  g  \rangle
= \langle  f , g - P_t g  \rangle \\
= - \int_0^t \langle f , LP_s g \rangle \, dt 
 = \int_0^t \langle \nabla f , \nabla P_s g \rangle \, ds . 
 \] 
By the previous proposition,
\[ 
\langle \nabla f , \nabla P_s g \rangle \leq  
\Vert \vert \nabla f \vert \Vert_{L^1(\mu)} \cdot 
 \Vert P_s  g \Vert_{\rm Lip} \leq \frac 1{\sqrt s}  \Vert \vert \nabla f \vert \Vert_{L^1(\mu)}
 \Vert g\Vert_\infty .  
\] 
Integrating between $0$ and $t$ and plugging back in the previous display we get
\[ 
\langle f - P_t f  ,  g  \rangle
\leq 2 \sqrt t \cdot   \Vert \vert \nabla f \vert \Vert_{L^1(\mu)}
 \Vert g\Vert_\infty ,   
\]
which is the result. 
For the second inequality, applying the 
first one to a suitable Lipschitz approximation of the indicator 
function of $A$, as in the proof of Lemma~\ref{lem_cheeger}, 
we get 
\[ 
\Vert \mathbbm 1_A - P_t \mathbbm 1_A \Vert_1 \leq 2 \sqrt t \cdot \mu^+ ( A). 
\]  
Moreover, using reversibility, it is not hard to see that 
\[ 
\Vert \mathbbm 1_A - P_t \mathbbm 1_A \Vert_1 
= 2 \left( \var_\mu ( \mathbbm 1_A ) - \var_\mu ( P_{t/2} \mathbbm 1_A ) \right) .
\] 
Hence the result.  
\end{proof} 

\subsection{A result of E. Milman} 
\label{sec_emil}
We said earlier that the inequality $C_P (\mu) \leq C \psi_\mu^2$ can
be reversed in the log-concave case. Actually we will prove a much 
stronger statement, which is due to E. Milman. 
\begin{definition} 
If $\mu$ is a probability measure on $\RR^n$, the function 
\[ 
I_\mu \colon 
r\in [0,1] \mapsto \inf \{ \mu_+ ( S ) \colon \mu (S) = r \} . 
\] 
is called the isoperimetric profile of $\mu$. 
\end{definition} 
With this definition Cheeger's inequality can be rewritten as
\[ 
\psi_\mu \cdot I_\mu (r) \geq \min (r,1-r) . 
\] 
The following is a deep result from geometric measure theory. 
\begin{theorem} The 
isoperimetric profile of a log-concave measure is concave. 
\end{theorem} 
We will use this as a blackbox, we refer to the 
appendix of \cite{Emil} for an historical 
account and the relevant references. Another good reference 
for this is Bayle's Ph.D. thesis~\cite{bayle} (if you read 
french).  
This has important implications for us. 
Indeed, since the isoperimetric profile is non negative, 
its concavity implies that 
\[ 
I_\mu (t) \geq 2 \cdot I_\mu (1/2) \min (t,1-t) .  
\]  
In particular the Cheeger constant of $\mu$ satisfies
\begin{equation}
\label{eq_cheeger3}
\psi_\mu \leq \frac 1{ 2 \cdot I_\mu (1/2) } . 
\end{equation}
Therefore, for a log-concave measure, 
in order to prove Cheeger's inequality, it is enough to look 
at the perimeter of sets of measure $1/2$. 
Combining this information with the results from 
the previous section we arrive at the following. 
\begin{theorem}\label{thm_ledouxmilman} 
If $\mu$ is log-concave, then there exists a $1$-Lischitz function $f$ 
satisfying 
\[ 
\Vert f\Vert_\infty^2 \approx \var_\mu (f) \approx \psi_\mu^2 . 
\] 
\end{theorem} 
Here the symbol $\approx$ means that the ratio between the 
two quantities is comprised between two positive universal constants. 
Theorem~\ref{thm_ledouxmilman} is essentially due to E. Milman~\cite{Emil}. The 
proof we give is very much inspired by Ledoux's proof of Buser's inequality~\cite{ledoux-buser}. 
\begin{proof} 
By~\eqref{eq_cheeger3} if $A$ is a set of measure $1/2$ that has near minimal 
surface, say up to a factor $2$, then 
\begin{equation}\label{eq_ledoux} 
\mu_+ ( A ) \leq \frac 1{ \psi_\mu } . 
\end{equation} 
Let $t>0$. By Corollary~\ref{cor_ledoux}, and since $\mu(A)=1/2$, 
\[ 
\frac 14 \leq \sqrt{2t} \cdot \mu_+ (A) +  \var_\mu (P_t\mathbbm 1_A) 
\leq \frac{ \sqrt{2t} }{\psi_\mu} +  \var_\mu (P_t\mathbbm 1_A)  .
\] 
If $t$ is a sufficiently small multiple of $\psi_\mu^{2}$ we thus get 
$\var_\mu ( P_t \mathbbm 1_A  ) \geq \frac 18$ (say). 
On the other hand, by Proposition~\ref{prop_reg}, 
\[ 
\Vert P_t \mathbbm 1_A \Vert_{\rm Lip} \leq \frac 1{\sqrt{t}} \leq \frac {C}{\psi_\mu},  
\] 
for some constant $C$. 
Putting everything together we see that 
the function $f = (\psi_\mu / C) \cdot P_t \mathbbm 1_A$ 
is $1$-Lipschitz and satisfies 
\[ 
\psi_\mu^2 \lesssim \var_\mu (f) \leq \Vert f\Vert_\infty^2 \lesssim 
\psi_\mu^2  \qedhere.  
\] 
\end{proof} 
Note that since $f$ is $1$-Lipschitz, the Poincar\'e 
inequality yields $\var_\mu (f) \leq C_P (\mu)$. The 
result above thus implies that 
\[ 
\psi_\mu^2 \lesssim C_P ( \mu ) . 
\]
In other words, the Cheeger inequality can be reversed in the 
log-concave case.  
Moreover, the theorem actually yields a lot more.
It implies that it is enough to bound the variance of Lipschitz functions 
to get Poincar\'e (or Cheeger). More precisely, we get the following. 
\begin{corollary}[Buser~\cite{B}, Ledoux~\cite{ledoux-buser}, E. Milman~\cite{Emil}]\label{cor_Emil}
For any log-concave measure $\mu$,  
\[ 
\psi_\mu^2 \approx C_P (\mu) \approx  \sup \left\{  \var_\mu (f)  \colon  \Vert f \Vert_{\rm Lip} \leq 1  \right\} .
\] 
\end{corollary} 
Constants are mostly regarded as irrelevant in theses notes but let us mention 
that for the left-most equality, 
the optimal constants are actually known. Indeed De Ponti and Mondino \cite{DM} 
proved that 
$$ \frac{1}{\pi} \psi_{\mu}^2 \leq C_P(\mu) \leq 4 \psi_{\mu}^2. $$

In section~\ref{sec_needle_milman} 
we give another proof of this corollary based on $L_1$ transportation 
that avoids the concavity of the isoperimetric profile blackbox. 

Let us also point out that the corollary 
does not quite use the full strength of Theorem~\ref{thm_ledouxmilman}, it does 
not use the information about the $L^\infty$ norm of $f$. So we actually
have stronger form of the corollary. 
Namely, in the log-concave case, to get Cheeger, or Poincar\'e,  
it is enough to bound the variance of a bounded Lipschitz function whose 
Lipschitz constant is $1$, and whose $L^\infty$-norm is of the same order as its standard deviation.

\subsection{Concentration of measure}\label{sec_conc}

\begin{definition} Let $(X,d,\mu)$ be a metric measure space. 
The concentration function of $\mu$ is defined by 
\[ 
\alpha_\mu  \colon r\mapsto \sup \left\{ 1-\mu (S_r) \colon \mu (S) \geq 1/2 \right\} 
\]
where $S_r$ is the $r$-neighborhood of the set $S$. 
\end{definition} 
As for isoperimetry, we can only compute the exact value of the concentration function in some very specific models such as the uniform measure on the sphere or the Gaussian measure. In general we are happy with a good upper 
bound for $\alpha_\mu$. The most interesting types of upper bounds for us are the case of Gaussian concentration and of exponential concentration. 
\begin{definition} We say that $\mu$ satisfies Gaussian concentration 
if there is a constant $C_G$ such that   
\[ 
\alpha_\mu (r) \leq 2 \cdot \exp \left( - \frac {r^2 }{C_G} \right) , \quad \forall r \geq 0. 
\] 
We say that $\mu$ satisfies exponential concentration 
if there exists a constant $C_{\text{exp}}$ such that  
\[ 
\alpha_\mu (r) \leq 
2 \cdot \exp \left( - \frac {r}{C_{\text{exp}}} \right) , \quad \forall r\geq 0  .
\] 
Moreover the smallest constants $C_G,C_{\text{exp}}$ such that the above inequalities
hold true are called the Gaussian concentration constant and the
exponential concentration constant, respectively. 
\end{definition} 
% sucht  
We are interested here in concentration properties of log-concave 
measures on $\RR^n$. Gaussian concentration cannot be true in general (think of $\mu$ being the exponential measure) but there is no obstruction to having exponential concentration with a dimension free constant for isotropic log-concave measures, and this is in fact equivalent to the KLS conjecture from the 
previous section. 
Indeed, it is well-known that the Poincar\'e 
inequality yields exponential concentration, and more precisely 
that for any probability measure $\mu$ on $\RR^n$ satisfying 
the Poincar\'e inequality we have 
\[ 
\alpha_\mu(r) \leq 2 \cdot \exp \left( - \frac{r}{ L \cdot \sqrt{C_P (\mu)}} \right) , \quad \forall r \geq 0 ,  
\]
where $L$ is a universal constant. 
We will skip the derivation of this from Poincar\'e 
here, but this is not very hard, see for instance~\cite[section 4.4.2]{BGL}. 

Once again, in the log-concave case this implication can be reversed. 
Indeed, by E. Milman's theorem (Corollary~\ref{cor_Emil}) from the previous 
subsection the Poincar\'e constant is a largest variance of a $1$-Lipschitz function 
(up to a constant). If $f$ is $1$-Lipshitz, by definition 
of the concentration function we have  
\[ 
\mu ( f - m \geq r ) \leq \alpha_\mu (r) ,  
\]
for every $r>0$, and where $m$ is a median 
for $f$. From this we obtain easily 
\[ 
\var_\mu (f) \leq 4 \int_0^\infty r \cdot \alpha_\mu (r) \, dr .
\] 
Therefore, in the log-concave case 
\begin{equation}
\label{eq_Emilbis}
C_P (\mu) \lesssim \int_0^\infty r \cdot \alpha_\mu(r) \, dr .  
\end{equation}
This implies in particular that the Poincar\'e constant of $\mu$ 
and the exponential concentration constant squared 
are actually of the same order. 

\subsection{Log-Sobolev and Talagrand} \label{sec_ls}
We have seen earlier that Poincar\'e is weaker than
Cheeger in general but equivalent to it within 
the class of log-concave measures. We shall see now that 
log-concavity also allows to reverse the hierarchy between 
the log-Sobolev inequality 
and the transportation inequality. 
A probability measure $\mu$ 
on $\RR^n$ is said to satisfy the logarithmic 
Sobolev inequality if there exists a constant $C>0$ such that 
\[ 
D(\nu \mid \mu ) \leq \frac C2  I ( \nu \mid \nu ) 
\] 
for every probability measure $\nu$, where $D(\nu\mid \mu)$ 
and $I(\nu \mid \mu)$ denote the relative entropy and 
Fisher information, respectively:
\[ 
D (\nu \mid \mu) = \int_{\RR^n} \log ( \frac{d\nu}{\, d\mu} ) \, d\nu 
\quad \text{and}\quad
I (\nu \mid \mu) = \int_{\RR^n} \vert \nabla \log ( \frac{d\nu}{\, d\mu})\vert^2 \, d\nu . 
\]  
The best constant $C$ is called the log-Sobolev constant, denoted $C_{LS}(\mu)$ below. The factor $1/2$ is just a matter of convention. 
With this convention the log-Sobolev constant of the standard Gaussian 
$1$. This is a stronger inequality than Poincar\'e. More precisely 
we have $C_{P} (\mu) \leq C_{LS} (\mu)$ for any $\mu$. This is easily 
seen by applying log-Sobolev to a probability measure 
whose density with respect to $\mu$ is $1 + \eps f$ and 
letting $\eps$ tend to $0$. Not every log-concave measure satisfy
log-Sobolev, simply because log-Sobolev implies sub-Gaussian tails, so for instance the exponential measure (on $\RR$) does not 
statisfy log-Sobolev. A bit more precisely, log-Sobolev 
implies Gaussian concentration: if $\mu$ 
satisfies log-Sobolev then for any set $S$ we have
\[ 
\mu (S) (1-\mu ( S_r)) \leq \exp \left( - c \cdot \frac {r^2 }{C_{LS} (\mu) } \right) . 
\] 
We will come back to that later on. 
% 
%Given a Lipschitz function $f$, 
%applying log-Sobolev to $\e^{\lambda f}$ yields 
%an integro-differential inequality for the function 
%\[ 
%\lambda \mapsto \int_{\RR^n} \e^{\lambda f} \, d\mu . 
%\] 
%Solving this differential inequality yields the 
%following. If $\mu$ satisfies log-Sobolev, then for 
%any $1$-Lipschitz function $f$ we have 
%\begin{equation}\label{eq_BG}
%\log\left( \int_{\RR^n} \e^{\lambda f} \, d\mu \right) 
%\leq \int_{\RR^n} f \, d\mu 
%+  C_{LS}(\mu) \cdot \frac{ \lambda^2 }2 .   
%\end{equation} 
%%Applying this to $f = d(x,S)$ yields Gaussian concentration. 
%This can be reformulated in terms of 
%a transportation inequality. 

Recall that if $\mu,\nu$ 
are probability measures on $\RR^n$, the quadratic transportation 
cost from $\mu$ to $\nu$ is defined as 
\[ 
T_2 (\nu ,\nu) = \inf \left\{ 
\int_{\RR^n \times \RR^n} \vert x-y\vert^2 \, d\pi \right\}, 
\] 
where the infimum is taken over every coupling $\pi$ 
of $\mu$ and $\nu$, namely every probability measure 
on the product space whose marginals are $\mu$ and $\nu$. 
In the next section we will speak about the Monge transport cost, 
which is the $L^1$ version of this. 
%sly a larger quantity than the Monge cost squared. 
%
\begin{proposition}[Otto and Villani~\cite{OV}]
If $\mu$ satisfies log-Sobolev then for every probability measure 
$\nu$ we have 
\[ 
T_2 ( \nu , \mu) \leq 2 C_{LS} (\mu) \cdot D(\nu \mid \mu) .  
\] 
\end{proposition}
This transportation/entropy inequality is sometimes called Talagrand's inequality, as it was first established by Talagrand for the Gaussian measure, see~\cite{Talag}. 
Again in the log-concave case 
the implication log-Sobolev/Talagrand can be reversed. 
Indeed, we have the following, also due to Otto and Villani. 
\begin{proposition}[Otto and Villani~\cite{OV}]\label{prop_HWI}
If $\mu$ is log-concave then for every probability measure $\nu$
on $\RR^n$ we have
\[ 
D(\nu \mid \mu) \leq \sqrt{ T_2 (\nu, \mu) \cdot I (\nu \mid \mu ) }. 
\]   
\end{proposition} 
This is only a particular case of the Otto-Villani 
result, there is also a version for semi-log-concave measures, 
namely measures for which we have a possibly negative 
lower bound on the Hessian of the potential. 
This inequality goes by the name 
HWI. The reason for this name is 
not apparent from our choice of notations, but relative entropy
is often denoted $H$, and the transport 
cost $T_2$ can also be denoted $W_2$ or rather $W_2^2$ 
(for Wasserstein). 
From the HWI inequality we see that the implication 
between log-Sobolev and Talagrand can be reversed for log-concave 
measures: if we happen to know 
\[ 
T_2 (\nu, \mu) \leq C_2 D(\nu\mid \mu) 
\] 
for $\mu$ log-concave, then 
we get log-Sobolev for $\mu$ and $C_{LS}(\mu) \leq 2 C_2$.
We will not spell out the proofs of the Otto-Villani 
results here and we refer to~\cite{OV} (see also~\cite{BoGL}). 

We have seen above that the equivalence between Cheeger 
and Poincar\'e can be considerably reinforced.
This is also the case here, and this is yet again a 
result of E. Milman. 
\begin{theorem}[E. Milman~\cite{Emil2}]\label{thm_milman_gauss} 
For a log-concave probability measure we have equivalence between 
Gaussian concentration and the log-Sobolev inequality, 
and moreover the log-Sobolev constant and the Gaussian concentration constant are within a universal factor of each other. 
\end{theorem} 
\begin{proof} There are several proofs of this result in 
the literature, see~\cite{Emil2,ledoux_semi}. 
The proof sketch that we give 
here is taken from Gozlan, Roberto, Samson~\cite{GRS}. 
We said earlier that log-Sobolev implies Gaussian 
concentration, but a bit more is true: the weaker Talagrand inequality
also implies Gaussian concentration. Let us explain why briefly. 
By some convex duality principle, $T_2$ can be also 
expressed as a supremum, namely  
\[ 
T_2 (\mu ,\nu) = \sup_f  \left\{
\int_{\RR^n} Q_{1/2} f \, d\mu - \int_{\RR^n} f \, d\nu \right\} 
\] 
where the $Q_t f$ is the infimum convolution of $f$ with 
some multiple of the distance squared:
\[ 
Q_t f (x) = \inf_{y\in \RR^n} \left\{ f(y) + \frac1{2t} \vert x-y\vert^2 \right\} . 
\] 
It can also be shown that $(Q_t)$ is a semigroup of operators,  
namely we have $Q_sQ_t = Q_{s+t}$. 
Lastly there is also some duality between the log-Laplace transform and the relative entropy: 
\[ 
\log \int_{\RR^n} \e^f \, d\mu = \sup_{\nu} \left\{ \int_{\RR^n} f \, d\nu 
- D(\nu \mid \mu) \right\} ,   
\]   
where the supremum is taken over every probability measure $\nu$.  
Using all this, it is pretty easy to see that Talagrand's inequality 
\[
T_2 (\nu, \mu) \leq 2 C_T \cdot D(\nu\mid \mu) , \quad \forall \nu
\]
is equivalent to 
\[ 
\int_{\RR^n} \exp ( Q_{C_T} f ) \, d\mu 
\leq \exp  \left( \int_{\RR^n} f \, d\mu \right) , \quad \forall f. 
\] 
Applying this to both $Q_{C_T} f$ and $-Q_{C_T} f$, using the fact that $(Q_t)$ is a semigroup, and multiplying the two inequalities 
together we get 
\[ 
\int_{\RR^n} \exp ( Q_{C_T} ( - Q_{C_T} f ) \, d\mu \cdot 
\int_{\RR^n} \exp ( Q_{2C_T} f ) \, d\mu  \leq 1 . 
\] 
But clearly $-f \leq Q_{C_T} (-Q_{C_T} f)$, so we obtain 
\[ 
\int_{\RR^n} \exp ( -f ) \, d\mu \cdot 
\int_{\RR^n} \exp ( Q_{2C_T} f ) \, d\mu  \leq 1 . 
\]   
Applying to $f = -\log \mathbbm 1_A$ we get 
\[ 
\int_{\RR^n} \exp \left( \frac{d(x,A)^2 }{ 2C_T } \right) \, dx 
\leq \frac 1{\mu (A)}  , 
\]
for every set $A$. 
By Markov inequality this implies 
\[ 
\alpha_\mu ( r) \leq 2 \cdot \exp\left( - \frac{r^2}{2 C_T} \right) .  
\] 
So Talagrand implies Gaussian concentration, and moreover 
the Gaussian concentration constant is at most the constant in 
Talagrand, up to a factor $2$. Now we
want to reverse this, so we assume
\[ 
\alpha_\mu (r) \leq 2 \e^{ - r^2 / C_G } .  
\]  
It is easily seen to imply 
\[ 
\int_{\RR^n} \exp ( Q_{2C_G} f )  \, d\mu \lesssim  \exp ( m_f ) . 
\] 
for every $f$, and where $m_f$ is a median for $f$.
Again the notation $\lesssim$ means up to a universal factor. 
Again, applying this 
$-Q_{2 C_G}f$ and $Q_{2C_G}f$ and multiplying the two inequalities together we get 
\[ 
\int_{\RR^n} \e^{-f} \, d\mu \cdot 
\int_{\RR^n} \exp ( Q_{4C_G} f )  \, d\mu \lesssim  1 ,  
\]  
hence by Jensen's inequality
\[ 
\int_{\RR^n} \exp ( Q_{4C_G} f )  \, d\mu \lesssim  \exp \left( \int_{\RR^n} f \, d\mu \right). 
\] 
In other words we get the dual version of Talagrand, 
but with some prefactor. In terms of transport and 
entropy this gives 
\[ 
T_2 (\nu, \mu) \lesssim C_G ( D (\nu \mid \mu ) + 1 ) . 
\] 
So we have an additional additive constant in the 
right-hand side of Talagrand. We have not used 
log-concavity yet, this would be true for any measure
satisfying Gaussian concentration.  
Now assuming log-concavity, we can plug this into the 
HWI inequality (Proposition~\ref{prop_HWI}). We get 
\[ 
D(\nu\mid \mu ) \lesssim C_G \cdot I(\nu \mid \mu ) + 1 .  
\] 
Again, we get some weak form of log-Sobolev 
with an additional constant term in the right-hand side.
This is sometimes called non-tight log-Sobolev inequality. 
To get rid of that constant, observe first that we clearly have
from the first theorem of E. Milman 
(see equation~\eqref{eq_Emilbis}) 
\[ 
C_P (\mu) \lesssim C_G. 
\]  
Moreover, non-tight log-Sobolev can be reformulated as 
\[ 
\ent_\mu( f^2 ) \lesssim C_G \int_{\RR^n} \vert\nabla f\vert^2 \, d\mu 
 + \int_{\RR^n} f^2 \, d\mu ,  
\] 
where the entropy of a non negative function $f$ is defined as 
\[ 
\ent_\mu (f) = \int_{\RR^n} f\log f \,  \, d\mu 
- \left( \int_{\RR^n} f \, d\mu \right) 
\log \left( \int_{\RR^n} f \, d\mu \right) . 
\] 
Now there is a nice inequality by Rothaus~\cite{rothaus} which states that 
for any $f\colon \RR^n \to \RR$ and any constant $c$ we have 
\[ 
\ent_\mu ((f+c)^2 ) \leq \ent_\mu (f^2) + 2 \int_{\RR^n} f^2 \, d\mu .   
\] 
Using this inequality it is easy to see that 
our non tight version of log-Sobolev and the 
bound that we have on $C_P (\mu)$ altogether
imply  
\[ 
\ent_\mu (f) \lesssim C_G \int_{\RR^n} \vert \nabla f\vert^2  \, d\mu,  \]
which is a reformulation of the desired log-Sobolev inequality.  
\end{proof} 
\pagebreak

\section{Optimal transport theory with the Monge cost}

Let $\mu_1$ and $\mu_2$ be two measures in $\RR^n$, say compactly-supported and absolutely continuous, with the same total mass, i.e., $\mu_1(\RR^n) = \mu_2(\RR^n)$. We would like to push-forward the measure $\mu_1$ to the measure $\mu_2$ in the most efficient way, that minimizes the average distance that points have to travel.
That is, we look at the optimization problem
$$ \inf_{S_*(\mu) = \nu} \int_{\RR^n} |S x - x| \, \mu_1(dx). $$
This is the problem of Optimal Transport with the Monge cost
or the $L^1$ cost, considered by Monge in 1781. 
See Cayley's review of Monge's work \cite{cayley} from 1883. 
For a more recent survey on Monge's problem, see for instance~\cite{BKsurvey}.  
Here is a heuristics from Monge's paper that explains why this problem induces a partition into segments.

\medskip
\noindent
{\bf Monge heuristic:} For the optimal transport map $T$,
the segments $(x, T(x)) \ (x \in Supp(\mu_1))$ do not intersect, unless they overlap.

\begin{proof}[Explanation.] Suppose that open segments $(x,Tx)$ and $(y,Ty)$ are not parallel and intersect at a point $z$. The triangle inequality then shows that 
$\vert x - Ty \vert + \vert y - Tx \vert < \vert x - Tx \vert + \vert y - Ty \vert$,
which contradicts the fact that the map $T$ is optimal. 
\end{proof}

This is related to the following elementary riddle: given $50$ red points and $50$ blue points in the plane, in general position, find a matching so that the corresponding segments do not intersect.

\medskip
Since the above argument relies only on the triangle inequality, you would expect that the optimal transport problem would induce a partition into geodesics also for Riemannian manifolds, or Finslerian manifolds, or measure metric spaces of some type -- basically wherever the triangle inequality holds true (under some regularity assumptions). 
%This is in contrast to the hyperplane bisection method, that applies only in highly symmetric spaces such as $\RR^n, S^n$ or $\HH^n$.

\subsection{Linear programming relaxation and the dual problem}

In Monge's problem we minimize over all maps $S$ that push-forward $\mu_1$ to $\mu_2$. There is a relaxation of this problem, that looks at all possible {\it couplings}, or transport plans, of the two distributions. That is, instead of mapping a point $x$ to a single point $Tx$, we are allowed to spread the mass across a region. Thus we look at all  measures $\gamma$ on $\RR^n \times \RR^n$ with
$$ (\pi_1)_* \gamma = \mu_1 \qquad \text{and} \qquad (\pi_2)_* \gamma = \mu_2. $$
where $\pi_1(x,y) = x$ and $\pi_2(x,y) = y$. Such a  measure is called a {\it coupling} of $\mu$ and $\nu$. In other words, we now look at {\it transport plans} rather than {\it transport maps}.
The advantage is that the space of all couplings is a convex set. The relaxed optimal transport problem involves minimizing the average distance that points travel, namely we look at
$$ \inf_{(\pi_1)_* \gamma = \mu, (\pi_2)_* \gamma = \nu} \int_{\RR^n \times \RR^n} |x-y| \, \gamma(dx,dy). $$
Hence we minimize a linear function on a convex set, this is Linear Programming or Functional Analysis (see e.g. Kantorovich and Akilov \cite[Section VIII.4]{Kantorovich_Akilov}).

\begin{theorem} (The dual problem) Let $\mu_1, \mu_2$ be two absolutely continuous measures in $\RR^n$ with the same total mass. Assume that 
	$$ \int_{\RR^n} |x| \, \mu_1(dx)  < \infty \qquad \text{and} \qquad \int_{\RR^n} |x| \, \mu_2(dx)  < \infty. $$
Denote $\mu = 	\mu_2 - \mu_1$. Then the following quantities are equal:
\begin{enumerate}
	\item The minimum over all couplings $\gamma$ of $\mu_1$ and $\mu_2$ of the integral
	$$ \int_{\RR^n \times \RR^n} |x-y| \, \gamma(dx,dy). $$
	\item The maximum over all $1$-Lipschitz functions $u: \RR^n \rightarrow \RR$ of
	$$ \int_{\RR^n} u (x) \, \mu(dx) $$
	\item The minimum over all maps $T$ with $T_* \mu_1 = \mu_2$ of
	$$ \int_{\RR^n} |x - Tx| \, \mu_1(dx). $$
\end{enumerate}
	\label{thm_1015}
\end{theorem}

\begin{proof}[Proof sketch] We refer to Ambrosio \cite{Am} for full details. 
For the easy direction of the linear programming duality, pick a $1$-Lipschitz map $u$ and a coupling $\gamma$. For any points $x,y \in \RR^n$,
$$ u(y) - u(x)  \leq |x-y|. $$
Integrating with respect to $\gamma$, we get
\begin{equation}  \int_{\RR^n} u d \mu = \int_{\RR^n \times \RR^n} [u(y) - u(x)] \, \gamma(dx,dy) \leq \int_{\RR^n \times \RR^n} |x-y| \, \gamma(dx,dy).
\label{eq_1541} \end{equation}
Hence we need to find $u$ and $\gamma$ so that equality is attained in (\ref{eq_1541}).
The argument goes roughly as follows.
A compactness argument  shows that the infimum over all couplings is attained. Indeed, by Alaoglu's theorem,
the collection of all couplings is compact in the $w^*$-topology (integration against continuous functions on $\RR^n$ whose limit at infinity exists). The 
functional $\gamma \mapsto \int_{\RR^n \times \RR^n} |x-y| \, \gamma(dx,dy)$ is lower semi-continuous in $w^*$-topology, hence its minimum is attained. 

\medskip Similarly to the Monge heuristics, the optimality implies that the support of $\gamma$ must be cyclically monotone: If $(x_i,y_i) \in Supp(\gamma) \subseteq \RR^n \times \RR^n$ for $i=1,\ldots,N$ then for any permutation $\sigma \in S_N$,
\begin{equation}  \sum_{i=1}^N |x_i - y_i| \leq \sum_{i=1}^N |x_i -  y_{\sigma(i)}|. \label{eq_1028} \end{equation}
Indeed, otherwise one may pick small balls around $x_i$ and $y_i$ and rearrange them to contradict optimality. 
Similarly to Rockafellar's theorem \cite{roc_thm} from convex analysis, condition (\ref{eq_1028}) implies that there exists a $1$-Lipschitz function $u: \RR^n \rightarrow \RR$ with
\begin{equation}  (x,y) \in Supp(\gamma) \qquad \Longrightarrow \qquad u(y) - u(x) = |y - x|. \label{eq_1038} \end{equation}
Indeed, fix $(x_0,y_0) \in Supp(\gamma)$ and define $u(x)$ as the supremum over all lower bounds with $u(x_0) = 0$,
$$ u(x) = \sup_{N, (x_1,y_1),\ldots,(x_N, y_N) \in Supp(\gamma)} \left \{ |x_0 - y_0| - |y_0 - x_1| + |x_1 - y_1| - |y_1 - x_2| + ... - |y_N - x| \right \}
$$
It follows from (\ref{eq_1028}) that $u(x_0) = 0$. The function $u$ is a $1$-Lipschitz function as a supremum of $1$-Lipschitz functions. 
It follows from the definition of $u$ that (\ref{eq_1038}) holds true. Hence we found $u$ and $\gamma$ so that equality is attained in (\ref{eq_1541}).
The proof that $\gamma$ can also be replaced by a transport map is due to Evans and Gangbo \cite{EG}.
This relies on  analysis of the structure of $u$ that will be described next.
\end{proof}

\begin{remark} The minimizers $\gamma$ or $T$ are not at all unique. It is actually the $1$-Lipschitz function $u$ which is essentially determined. More precisely, the gradient $\nabla u$ is determined $\mu$-almost everywhere.
\end{remark}

We move on to discuss the structure of $1$-Lipschitz functions.
Observe that when a $1$-Lipschitz function $u$ satisfies $|u(x) - u(y)| = |x - y|$, for some points $x, y \in \RR^n$, 
it necessarily grows in speed one along the segment from $x$ to $y$. A maximal open segment $I$ on which $u$ grows with speed one, 
i.e., $|u(x) - u(y)| = |x-y|$ for all $x,y \in I$, is called a {\it transport ray}.
Theorem \ref{thm_1015} tells us that optimal transport only happens only along transport rays, we only rearrange mass along transport rays.

\medskip It is illuminating to draw the transport rays of the function $u(x) = x_1$
in connection with Fubini's theorem 
$$ \int_{\RR^2} \vphi = \int_{-\infty}^{\infty} \left( \int_{-\infty}^{\infty} \vphi(x_1, x_2) dx_1 \right) dx_2, $$
and of the function  $u(x) = |x|$ on $\RR^2 \cong \CC$ in connection with integration in polar coordinates:
$$ \int_{\RR^2} \vphi = \int_{0}^{2 \pi} \left( \int_{0}^{\infty} \vphi(r e^{i \theta}) r dr\right) d \theta. $$
Note that the Jacobian factor on the needle is log-concave in both examples.

\medskip
The next step is to understand the {\it disintegration of measure} or {\it conditional probabilities} induced by the partition into
transport rays. 
 Let $u$ be a maximizer as above, with $$ \mu = \mu_2 - \mu_1, $$ and with the two measures 
satisfying the requirements of Theorem \ref{thm_1015}. As it turns out, it is guaranteed
that transport rays of positive length form a partition of the entire support of the measure $\mu$, up to a set of measure zero.
Write
$$ f = \frac{\, d\mu}{d \lambda} $$
where $\lambda$ is any
log-concave reference measure in $\RR^n$ (not necessarily finite; it could be the Lebesgue measure for instance). The assumption that
$\mu_1(\RR^n) = \mu_2(\RR^n)$ is equivalent to the requirement that 
\begin{equation}  \int_{\RR^n} f d\lambda = 0. \label{eq_1035} \end{equation}
The following theorem requires careful regularity analysis, and in addition to Evans and Gangbo \cite{EG} it builds upon works 
by Caffarelli, Feldman and McCann \cite{CFM} as well as \cite{K_needle}. It is 
analogous to integration in polar coordinates, yet with respect
to a general $1$-Lipschitz guiding function, rather than just $u(x) = |x|$.
In the following theorem a line segment could also mean a singleton, a ray or a line.

\begin{theorem} [Evans and Gangbo \cite{EG}, Caffarelli, Feldman and McCann \cite{CFM}, Klartag \cite{K_needle}]
Let $\lambda$ be an absolutely-continuous, log-concave measure on $\RR^n$, and let $f \in L^1(\lambda)$ satisfy (\ref{eq_1035}). Then
there is a collection $\Omega$ of line segments that form a partition of $\RR^n$,
a family of measures $\{ \lambda_{\cI} \}_{\cI \in \Omega}$, and a measure $\nu$ on the space of segments $\Omega$,  such that
		\begin{enumerate}
			\item For any $\cI \in \Omega$ the measure $\lambda_{\cI}$ is supported on the line segment $ \cI$.
If $\cI$ is of non-zero length, then it is a transport ray of the $1$-Lipschitz function $u$.

			\item Disintegration of measure
			$$ \lambda = \int_{\Omega} \lambda_{\cI} \, \nu(d\cI). $$
		
			\item Mass balance condition: for $\nu$-almost any $\cI \in \Omega$,
			$$ \int_{\cI} f d \lambda_{\cI} = 0. $$
			\item For $\nu$-almost any $\cI \in \Omega$, the measure $\lambda_{\cI}$
has a $C^{\infty}$-smooth, positive density $\rho$ with respect to the Lebesgue measure 
on the segment $\cI$ which is log-concave. \\ (In fact, in the case where $\lambda$ is the Lebesgue measure, it is a polynomial of degree $n-1$ with real roots, that does not vanish in the support of $\lambda_{\cI}$).
		\end{enumerate}
	\label{thm_1812}
\end{theorem}

\begin{remark} 
This theorem may be generalized to any Riemannian manifold with non-negative Ricci curvature. We replace 
the line segment $\cI$ by a unit-speed geodesic $\gamma = \gamma_\cI$, and set $\kappa(t) = Ricci(\dot{\gamma}(t), \dot{\gamma}(t)), n = \dim(M)$. Denote by $\rho = \rho_\cI$ the density of $\mu_\cI$ with respect to arclength on the geodesic $\gamma = \gamma_\cI$. Then,
$$ \left( \rho^{\frac{1}{n-1}} \right)^{\prime \prime} + \frac{\kappa}{n-1} \cdot \rho^{\frac{1}{n-1}} \leq 0. $$
The Riemannian version may be used to prove isoperimetric inequalities under lower bounds on the Ricci curvature, as well as Poincar\'e inequalities, log-Sobolev inequalities, Brunn-Minkowski inequalities and more, see~\cite{K_needle}. A generalization to the context of synthetic Ricci bounds 
was introduced by Cavalleti and Mondino \cite{CaMo}. See also Ohta \cite{ohta} for the non-symmetric, Finslerian case. 
\end{remark}

\begin{proof}[Some ideas from the proof of Theorem \ref{thm_1812}]
The proof of Theorem \ref{thm_1812} does not use sophisticated results from Geometric Measure Theory, but it consists of several steps. Essentially, 
\begin{itemize}
\item Show that a $1$-Lipschitz $u$ is always differentiable in the relative interior of a transport ray. 
\item The next step is to show that $\nabla u$ is a locally-Lipschitz function on a set which is only slightly smaller than the union of all transport rays, 
and that the restriction of $u$ to this set may be extended to a $C^{1,1}$-function on $\RR^n$.
\item  This is just enough regularity in order to allow change of variables in an integral,
which yields the disintegration. 
\item By differentiating the Jacobian one sees that the logarithmic derivative of the needle density is the mean curvature 
of the level set of $u$, and the inverse principal curvatures grow linearly along the needle. This yields log-concavity along each needle.
\item The mass balance condition follows from the fact that $\gamma$ is a coupling 
between $\mu_1$ and $\mu_2$, and that transport happens only along transport rays (thanks to S. Szarek for this remark). 
Alternatively, one can use a perturbative argument based on the maximality of the integral $\int u f d \lambda$.
\end{itemize}
We refer to~\cite{K_needle} for the details. 
\end{proof}

 As an application of this theorem, let us prove the reverse Cheeger inequality of Buser \cite{B} and Ledoux \cite{ledoux-buser}, and in fact a refinement due to E. Milman \cite{Emil}.
 In Section \ref{sec_emil} above we saw another proof, using semi-group methods, of the following:

\begin{proposition} \label{prop_1116}
Let $\mu$ be a log-concave probability measure on $\RR^n$ and $R > 0$. Assume that for any $1$-Lipschitz function $u: \RR^n \rightarrow \RR$ there exists $\alpha \in \RR$ with
	\begin{equation} \int_{\RR^n} |u(x) - \alpha| d \mu(x) \leq R. \label{eq_E1139} \end{equation}
(this is a weaker condition than requiring $C_P(\mu) \leq R^2$).
Then for any measurable set $S \subseteq \RR^n$ and $0 < \eps < R$,
\begin{equation}  \mu(S_{\eps} \setminus S) \geq  c \cdot \frac{\eps}{R} \cdot \mu(S) \cdot (1 - \mu(S)), \label{eq_1} \end{equation}
	where $c > 0$ is a universal constant, and where $S_{\eps}$ is the $\eps$-neighborhood of $S$. In particular the Cheeger constant of $\mu$ (see section~\ref{sec_emil}) satisfies
\[ 
 \psi_\mu \lesssim R . 
\]   
\end{proposition}
\begin{proof} Denote $t = \mu(S) \in [0,1]$ and set $f(x) = 1_S(x) - t$ for $x \in \RR^n$. Then $\int f \, d\mu = 0$. We then consider the Monge transportation problem between $f_+ \, d\mu$ 
and $f_- \, d\mu$. Let $u$ be a $1$-Lipschitz function maximizing
	$$ \int_{\RR^n} u f d \mu. $$
After adding a constant to $u$, we may assume that
$$\int_{\RR^n} |u| d \mu \leq R. $$
By Theorem \ref{thm_1812}, we obtain a needle decomposition: measures $\{ \mu_{\cI} \}_{\cI \in \Omega}$ on $\RR^n$, and a measure $\nu$ on
the space $\Omega$ of transport rays which yield a disintegration of measure. 
Observe that the equality 
\[ 
\int_\Omega \mu_{\cI} ( \RR^n) \, \nu(d\cI) = \mu(\RR^n ) = 1   
\]
implies in particular that for $\nu$-almost every $\mathcal I$ the 
measure $\mu_\cI$ is finite. 
We may normalize and assume that they are all probability measures. More 
precisely we can replace each of the measures $\mu_{\cI}$ by $\mu_{\cI} / \mu_\cI ( \RR^n)$ 
and replace $\nu$ by the measure having density $\cI \mapsto \mu_{\cI} ( \RR^n) $ with respect to $\nu$.  Hence,
$$ \int_{\Omega} \left( \int_{\cI} |u| \, d \mu_{\cI} \right) \, \nu(d\cI) = \int_{\RR^n} |u| \, d \mu \leq R.
$$
Denote
$$  B = \left \{ \cI \in \Omega \, ; \, \int_{\cI} |u| \, d \mu_{\cI} \leq 2 R  \right \}. $$
By the Markov-Chebyshev inequality,
\begin{equation}
\nu(B) \geq 1/2.  \label{eq_E1144_}
\end{equation}
	For $\nu$-almost all intervals $\cI \in \Omega$ we know that $\int_{\cI} f d \mu_{\cI} = 0$, hence
$$
	\mu_{\cI}(S) = t \cdot \mu_{\cI}(\RR^n) = t.
	$$
We would like to prove that for any $\cI \in B$ and any $0 < \eps < R$,
	\begin{equation}
	\mu_{\cI}(S_{\eps} \setminus S) \geq c \cdot \frac{\eps}{R} \cdot  t (1 -t),
	\label{eq_E1215}
	\end{equation}
	for a universal constant $c > 0$. Once (\ref{eq_E1215}) is proven, the bound (\ref{eq_1}) follows by integrating (\ref{eq_E1215}) with respect to $\nu$ and using (\ref{eq_E1144_}), since
$$\mu(S_{\eps} \setminus S) \geq \int_{B} \mu_{\cI}(S_{\eps} \setminus S) \, \nu(d\cI) \geq \nu(B) \cdot c \cdot \frac{\eps}{R} \cdot  t (1 -t) \geq \frac{c}{2} \cdot \frac{\eps}{R} \cdot  t (1 -t). $$
What remains to be proven is a one-dimensional statement about log-concave measures: If $\eta = \mu_{\cI}$ is a log-concave probability measure on $\RR$ with $\int_{\RR} |t| d \eta(t) \leq 2R$, then (\ref{eq_E1215}) holds true. This follows from Corollary~\ref{cor_1844} and a scaling argument.
\end{proof}

The same proof applies for any complete Riemannian manifold with non-negative Riemannian curvature. In fact, completeness in unneeded, the weaker geodesic-convexity assumption suffices here.
There are quite a few other applications for this theorem, which helps reduce the task of proving an $n$-dimensional inequality to the task of proving a $1$-dimensional inequality (``localization''). In a simply-connected space of constant sectional curvature, most of these applications -- like reverse H\"older inequalities for polynomials -- may also be proven using a localization method based on hyperplane 
bisections that go back to Payne and Weinberger \cite{PW}, Gromov and Milman \cite{GM} and Kannan, Lov\'asz and Simonovits \cite{KLS}. Proposition 
\ref{prop_1116} seems to be an exception, our proof requires the $1$-Lipschitz guiding function.

\begin{exercise} [reverse H\"older inequalities for polynomials] 
Let $X$ be a log-concave random vector in $\RR^n$, 
and let $f: \RR^n \rightarrow \RR$ be a
polynomial of degree at most $d$. Then for any $0 < p \leq q$,
$$ \| f(X) \|_q \leq C_{q,d} \cdot \| f(X) \|_p, $$
for some constant $C_{q,d}$ depending only on $q$ and $d$.
\\
\textit{Hint:} In one dimension, following Bobkov \cite{bobkov7}, we 
may assume that $f$ is a monic polynomial in one real variable, hence
$$ f(X) = \prod_{i=1}^d (X - z_i) $$
for some $z_1,\ldots,z_d \in \CC$. Consequently, by H\"older inequality and by Corollary \ref{cor_1749},
$$ \| f(X) \|_q = \left \| \prod^d_{i=1} (X - z_i) \right \|_{q} \leq \prod_{i=1}^d \| X - z_i \|_{dq} \leq \prod_{i=1}^d C d(q+1) \| X - z_i \|_{0}
= (C d (q+1))^d \| f(X) \|_0. $$
Now use needle decomposition to extend this to higher dimensions. 
\end{exercise} 

\subsection{Isoperimetry and the Poincar\'e inequality}\label{sec_needle_milman}

Recall that the Cheeger inequality \cite{cheeger} states that for any absolutely continuous probability measure on $\RR^n$ satisfying some mild regularity assumptions,
\begin{equation}  C_P(\mu) \leq 4 \psi_{\mu}^2 . \label{eq_616} \end{equation}
The proof is sketched in section~\ref{sec_emil}. Combining this with 
Proposition~\ref{prop_1116} we thus recover the aforementioned result by Buser, that $C_P(\mu)$ and $\psi_{\mu}^2$ are within a constant factor of 
each other when $\mu$ is log-concave. Proposition \ref{prop_1116} moreover implies that in the log-concave case, 
there exists a $1$-Lipschitz function $f$ such that 
$$ \psi_{\mu}^2 \leq C \cdot \var_{\mu}(f). $$
This provides another proof of E. Milman's theorem (Corollary~\ref{cor_Emil}). 
\pagebreak

\section{Bochner identities and curvature}
\label{sec_lich}

In this lecture we discuss a technique that originated in Riemannian Geometry and connects 
the Poincar\'e inequality and Curvature. It started with the works of Bochner in the 1940s and also Lichnerowicz in the 1950s. The approach fits well with convex bodies and log-concave 
measures in high dimension. In a nutshell, the idea is to make local
computations involving something like curvature, as well as integrations by parts, and then dualize and obtain Poincar\'e-type inequalities. This may sound pretty vague, let us explain what we mean.

\medskip
Suppose that $\mu$ is an absolutely continuous log-concave probability measure in $\RR^n$. Then $\mu$ is supported in an open, convex set $K \subseteq \RR^n$
and it has a positive, log-concave density $\rho = e^{-\psi}$ in $K$.
We will measure distances using the Euclidean distances in $\RR^n$, but we will measure volumes using the measure $\mu$. We thus look at the {\it weighted Riemannian manifold} or the {\it metric-measure space}
$$ (K, | \cdot |, \mu). $$
Thus  the Dirichlet energy of a smooth function $f: \RR^n \rightarrow \RR$ is
$$ \| f \|_{\dot{H}^1(\mu)}^2 = \int_{K} |\nabla f|^2 d \mu. $$
Indeed, we measure the length of the gradient with respect to the Euclidean metric, while we integrate with respect to the measure $\mu$.
As was already defined in Section \ref{sec_semigroup}, the Laplace-type operator associated with this measure-metric space is defined, initially for $u \in C_c^{\infty}(K)$, via
$$ L u = L_{\mu} u =  \Delta u - \nabla \psi \cdot \nabla u = e^{\psi} div(e^{-\psi} \nabla u). $$
This reason for this definition is that  for any smooth functions $u,v:
\RR^n \rightarrow \RR$, with one of them compactly-supported in $K$,
$$ \int_{\RR^n} (L u) v d \mu = -\int_{\RR^n} [\nabla u \cdot \nabla v] e^{-\psi} \, dx. $$
and in particular
$$ \langle -L u, u \rangle_{L^2(\mu)} = \int_{\RR^n} |\nabla u|^2 d \mu. $$
Thus $L$ is a symmetric operator in $L^2(\mu)$, defined initially for $u \in C_c^{\infty}(K)$.
It can have more than one self-adjoint extension, for example corresponding to the Dirichlet or Neumann boundary 
conditions when $K$ is bounded. When discussing the Bochner technique, it is customary and possible to find ways to circumvent spectral theory of the operator $L$. Still, spectral theory 
helps us understand and form intuition, and we will at least quote the relevant spectral theory.

\medskip It will be convenient to make an (inessential) regularity assumption on $\mu$, 
so as to avoid all boundary terms in all integrations by parts. We say that $\mu$ is a regular, log-concave 
measure in $\RR^n$ if its density, denoted by $e^{-\psi}$, is smooth and positive in $\RR^n$ 
and the following two requirements hold:
\begin{enumerate}
	\item[(i)] Log-concavity amounts to $\psi$ being convex, so $\nabla^2 \psi \geq 0$ everywhere in $\RR^n$. We require a bit more, 
that there exists $\eps > 0$ such that for all $x \in \RR^n$, 
	\begin{equation}
	\eps \cdot \id \leq \nabla^2 \psi(x) \leq \frac{1}{\eps} \cdot \id.
	\label{eq_959}
	\end{equation}
	\item[(ii)] The function $\psi$, as well as each of its partial derivatives, grows at most polynomially at infinity.
\end{enumerate}

\begin{exercise}[regularization process] 
Begin with an arbitrary log-concave measure $\mu$
on $\RR^n$, convolve it by a tiny Gaussian, and then multiply its density by $\exp(-\eps |x|^2)$ for small $\eps > 0$. 
Show that the resulting measure is regular, log-concave, with approximately the same covariance matrix,
and that the Poincar\'e constant cannot jump down by much under this regularization process.
\end{exercise}
From now on, we assume that our probability measure $\mu$ is a regular, log-concave measure. 
It turns out that in this case, the operator $L$, initially defined on $C_c^{\infty}(\RR^n)$, is essentially self-adjoint, positive semi-definite operator in $L^2(\mu)$
with a discrete spectrum. Its eigenfunctions $1 \equiv \vphi_0, \vphi_1,\ldots$ constitute an orthonormal basis, and the eigenvalues of $-L$ are
$$ 0 = \lambda_0(L) < \lambda_1(L) = \frac{1}{C_P(\mu)} \leq \lambda_2(L) \leq \ldots $$
with the eigenfunction corresponding to the trivial eigenvalue $0$ being the constant function. The eigenfunctions 
are smooth functions in $\RR^n$ that do not grow too fast at infinity: each function $$ \vphi_j e^{-\psi/2} $$ decays exponentially at infinity.
Also $(\partial^{\alpha} \vphi_j) e^{-\psi/2}$ decays exponentially at infinity for any partial derivative $\alpha$.
This follows from known results on exponential decay of eigenfunctions of Schr\"odinger operators.
The eigenvalues are given by the following infimum of Rayleigh quotients
$$ \lambda_k(L) = \inf_{f \perp \vphi_0,\ldots,\vphi_{k-1}} \frac{\int_{\RR^n} |\nabla f|^2 d \mu}{\int_{\RR^n} f^2 d \mu} $$
where the infimum runs over all (say) locally-Lipschitz functions $f \in L^2(\mu)$. Since $\vphi_0 \equiv 1$, we indeed see that the first 
eigenfunction $\vphi_1$ saturates the Poincar\'e inequality for $\mu$. 
For proofs of these spectral theoretic facts, see references in \cite{K_root}. 

\medskip Let us return to Geometry.  In Riemannian geometry, the Ricci curvature appears when we commute the Laplacian and the gradient. Analogously, here we have the easily-verified commutation relation
$$ \nabla (L u) = L(\nabla u) - (\nabla^2 \psi) (\nabla u), $$
where $L(\nabla u) = (L(\partial^1 u),\ldots,L(\partial^n u))$.
Hence the matrix $\nabla^2 \psi$ corresponds to a curvature term, analogous to the Ricci curvature.

\begin{proposition}[Integral Bochner's formula]
	For any  $u \in C_c^{\infty}(\RR^n)$,
	\begin{equation*}
	\int_{\mathbb{R}^n}(Lu)^2\ \, d\mu=\int_{\mathbb{R}^n}\left(\nabla^2\psi\right)\nabla u\cdot \nabla u\ \, d\mu+ \int_{\mathbb{R}^n} \| \nabla^2 u \|_{HS}^2  \, d\mu,
	\end{equation*}
	where $\| \nabla^2 u \|_{HS}^2 = \sum_{i=1}^n |\nabla\partial_i u|^2$.
\end{proposition}
\begin{proof} Integration by parts gives
\begin{align*}
\int_{\mathbb{R}^n} (Lu)^2\ \, d\mu& =-\int_{\mathbb{R}^n}\nabla (Lu)\cdot\nabla u\ \, d\mu \\
 & =-\int_{\mathbb{R}^n} L(\nabla u)  \cdot \nabla u\ \, d\mu+ \int_{\mathbb{R}^n} \left[ (\nabla^2 \psi) \nabla u \cdot \nabla u \right] \ \, d\mu\\
&=\sum_{i=1}^n\int_{\mathbb{R}^n}|\nabla\partial_i u|^2\ \, d\mu+\int_{\mathbb{R}^n}\left(\nabla^2\psi\right)\nabla u\cdot \nabla u\ \, d\mu. \qedhere
\end{align*}
\end{proof}

The assumption that $u$ is compactly-supported was used in order to discard the boundary terms when integrating by parts. 
In fact, it suffices to know that $u$ is $\mu$-tempered. We say that $u$ is $\mu$-tempered if it is a smooth function,
and $(\partial^{\alpha} u) e^{-\psi/2}$ decays exponentially at infinity for any partial derivative $\partial^{\alpha} u$.
Any eigenfunction of $L$ is $\mu$-tempered. If $f$ is $\mu$-tempered, then so is $L f$.

\medskip The following inequality from \cite{K_root} is analogous to some investigations of Lichnerowicz \cite{Lich}. 
It is concerned with distributions that are more log-concave than a Gaussian distribution, in the sense 
that their logarithmic Hessian is uniformly bounded by that of the Gaussian.

\begin{theorem}[improved log-concave Lichnerowicz inequality] Let $t > 0$ and assume that $\nabla^2 \psi(x) \geq t$ for all $x \in \RR^n$. Then,
$$ C_P(\mu) \leq \sqrt{ \|\cov(\mu) \|_{op}  \cdot \frac{1}{t} }. $$
\label{thm_1022}
\end{theorem}

Equality in Theorem \ref{thm_1022} is attained when $\mu$ is a Gaussian measure, 
with any covariance matrix. Indeed in that case 
$C_P(\mu)$ and $\Vert \cov(\mu)\Vert_{op}$ coincide, and they also 
coincide with the inverse lower bound on the Hessian of the potential. 
Write $\gamma_s$ for the
law of distribution of a Gaussian random vector of mean zero and covariance 
matrix $s \cdot \id$ in $\RR^n$. Then $\gamma_s$
satisfies the assumptions of Theorem \ref{thm_1022} for $t = 1/s$ while $C_P(\gamma_s) = \|\cov(\gamma_s) \|_{op} = s$. 

\begin{proof}[Proof of Theorem \ref{thm_1022}] Denote $f = \vphi_1$, the first eigenfunction, normalized so that $\| f \|_{L^2(\mu)} = 1$. Set  $\lambda = 1 / C_P(\mu)$. By the Bochner formula
and the Poincar\'e inequality for $\partial^i f \ (i=1,\ldots,n)$,
	\begin{align} \nonumber \lambda^2 & = \int_{\RR^n} (L f)^2 d \mu = \int_{\RR^n} [(\nabla^2 \psi) \nabla f \cdot \nabla f] d \mu + \int_{\RR^n} \| \nabla^2 f \|_{HS}^2 d \mu
	\\ \nonumber  & \geq  t \int_{\RR^n} |\nabla f|^2 d \mu + \lambda \left[ \int_{\RR^n} |\nabla f|^2 d \mu - \left| \int_{\RR^n} \nabla f d \mu \right|^2 \right]
	\\ & =  (t + \lambda) \cdot \lambda - \lambda \left| \int_{\RR^n} \nabla f d \mu \right|^2.
	\label{eq_1731}
	\end{align}
Therefore the first eigenfunction has a ``preferred direction'', i.e., 
\begin{equation}  \left| \int_{\RR^n} \nabla f d \mu \right|^2 \geq t. 
	\label{eq_1029_} \end{equation}
We remark that in the general case, under log-concavity assumptions it is known that $\int_{\RR^n} \nabla f d \mu \neq 0$, see \cite{K_uncond}, 
and this leads to a bound on the dimension of the first eigenspace. The lower bound (\ref{eq_1029_}) is a quantitative version, relying on the assumption of a uniform lower bound on the log-concavity.
Using that the $i^{th}$ coordinate of $\nabla f$ is $\nabla f \cdot \nabla x_i$ and integrating by parts we have
$$ \int_{\RR^n} \nabla f d \mu = -\int_{\RR^n} (L f) x d \mu = \lambda \int_{\RR^n} f x d \mu $$
Since $\int f d \mu = 0$, by Cauchy-Schwartz, for some  $\theta \in S^{n-1}$,
\[ 
\begin{split} 
\left| \int_{\RR^n} \nabla f d \mu \right| & = \int_{\RR^n} \langle \nabla f, \theta \rangle d \mu = \lambda \int_{\RR^n} f(x) \langle x, \theta \rangle \, \mu(dx) \\
& \leq \lambda \| f \|_{L^2(\mu)} \cdot \sqrt{ \cov(\mu) \theta \cdot \theta} \leq \lambda \|\cov(\mu) \|_{op}. 
\end{split} 
\]
This expression is at least $t$, and the theorem follows.
\end{proof}

Since $\|\cov(\mu) \|_{op} \leq C_P(\mu)$, we deduce from Theorem \ref{thm_1022} that
\begin{equation}  C_P(\mu) \leq \frac{1}{t}. \label{eq_1022} \end{equation}
Inequality (\ref{eq_1022}) is sometimes referred to as the log-concave Lichnerowicz inequality. 
Therefore the bound in Theorem \ref{thm_1022} is a geometric 
average of the Lichnerowicz bound and the conjectural KLS bound.

\medskip The Bochner identity has quite a few additional applications in the study 
of log-concave measures, beyond the improved log-concave Lichnerowicz inequality.
Especially if one introduces the semigroup $(e^{tL})_{t \geq 0}$ associated with the operator $L$
(see e.g. Ledoux \cite{ledoux_book}), as we saw in Section \ref{sec_semigroup}. Yet even simple integrations by parts and duality arguments 
based on the Bochner identity lead to  non-trivial conclusions. 
One example is the Brascamp-Lieb inequality \cite{BL} from the 1970s:

\begin{theorem}[Brascamp-Lieb]\label{theorem-paris-1}
	For any $C^{1}$-smooth $f\in L^2(\mu)$,
	\begin{equation*}
	\var_{\mu}(f)\leq\int_{\mathbb{R}^n}\left(\nabla^2\psi\right)^{-1}\nabla f\cdot\nabla f\ \mu(dx),
	\end{equation*}
	where $\var_{\mu}(f)=\int_{\mathbb{R}^n}(f-E)^2\ \mu(dx),$ and $E=\int_{\mathbb{R}^n} f d \mu$.
\end{theorem}

\begin{proof} 
We will only prove this inequality for regular, log-concave measures, though it holds true under weaker regularity assumptions.
The space of all $\mu$-tempered functions is denoted by $\cF_{\mu}$. 
It is clearly a dense subspace of $L^2(\mu)$ and in fact its image under $L$ is dense in 
$$
  \vphi_0^{\perp} =  \left \{ g \in L^2(\mu) \, ; \, \int_{\RR^n} g d \mu = 0 \right \}. 
  $$  
Indeed, the image contains all finite linear combinations of all eigenfunctions $\vphi_1,\vphi_2,\ldots$  (without $\vphi_0$) which
is dense in $H$. Assume $\int f\ \, d\mu=0, \eps > 0$ and pick $u\in \cF_{\mu}$ such that
\begin{equation*}
\|Lu-f\|_{L^2(\mu)}<\varepsilon.
\end{equation*}
Then,
\begin{align*}
\var_{\mu}(f)=\|f\|_{L^2(\mu)}^2&=\|Lu-f\|_{L^2(\mu)}^2+2\int fLu\ \, d\mu-\int (Lu)^2\ \, d\mu \\
&\leq\varepsilon^2-2\int \nabla f\cdot\nabla u\ \, d\mu-\int (\nabla^2\psi)\nabla u\cdot\nabla u\ \, d\mu\\
&\leq \varepsilon^2+\int (\nabla^2\psi)^{-1}\nabla f\cdot\nabla f\ \, d\mu,
\end{align*}
where we have used the fact that
\begin{equation*}
\int (Lu)^2\ \, d\mu\geq\int (\nabla^2\psi)\nabla u\cdot\nabla u\ \, d\mu,
\end{equation*}
which follows from Bochner's formula and
\begin{equation*}
-2x\cdot y-Ax\cdot x\leq A^{-1}y\cdot y\Longleftrightarrow |\sqrt{A}x+\sqrt{A^{-1}}y|^2\geq 0.
\end{equation*}
The desired inequality follows by letting $\eps$ tend to zero.
\end{proof}

{\bf Remark.} The Brascamp-Lieb inequality is an infinitesimal version of the Pr\'ekopa-Leindler inequality. Suppose that
$f_0,f_1: \RR^n \rightarrow [0, \infty)$ are integrable, log-concave functions and
$$ f_t(x) = \sup_{x = (1-t) y + y z} f_0(y)^{1-t} f_1(z)^t. $$
The Pr\'ekopa-Leindler inequality implies that $\log \int_{\RR^n} f_t$ is concave in $t$. The second derivative in $t$ is non-negative, and this actually amounts to the Brascamp-Lieb inequality. Thus the Brascamp-Lieb inequality is yet another incarnation of the Brunn-Minkowski inequality.

\medskip We say that a function $\psi$ on the orthant $\RR^n_+$ is $p$-convex if $\psi(x_1^{1/p},\ldots, x_n^{1/p})$
is a convex function of $(x_1,\ldots,x_n) \in \RR^n_+$.

\begin{corollary}\label{proposition-paris-1}
	Let $\mu$ be a probability measure in the orthant $\RR^n_+$, set $e^{-\psi} = d \mu / dx$ and assume that $\psi$ is $p$-convex for $p=1/2$.
 Then for any $C^1$-smooth function $f \in L^2(\mu)$,
	\begin{equation*}
	\var_{\mu}(f)\leq 4\int_{\mathbb{R}^n}\sum_{i=1}^n x_i^2 |\partial_i f|^2\ \mu(dx).
	\end{equation*}
For general $p > 1$, replace the coefficient $4$ by $p^2 / (p-1)$.
\end{corollary}

\begin{proof} Change variables and use
the Brascamp-Lieb inequality. Denote $\frac{\, d\mu}{dx}=e^{-\psi}.$ Then for
\begin{equation*}
\pi(x_1,\cdots,x_n)=(x_1^2,\cdots, x_n^2),
\end{equation*}
the function $\psi(\pi(x))$ is convex. Set
\begin{equation*}
\varphi(x)=\psi(\pi(x))-\sum_{i=1}^n\log (2x_i).
\end{equation*}
Then $\pi^{-1}$ pushes-forward $\mu$ to the measure with density $e^{-\vphi}$. Moreover,
\begin{equation*}
\nabla^2\varphi(x)\geq\nabla^2\left(-\sum_{i=1}^n\log (2x_i)\right)=
\begin{pmatrix}
\frac{1}{x_1^2} & 0 & \cdots & 0 \\
0 & \frac{1}{x_2^2} & \cdots & 0 \\
\vdots  & \vdots  & \ddots & \vdots  \\
0 & 0 & \cdots & \frac{1}{x_n^2}
\end{pmatrix}>0,
\end{equation*}
and therefore 
\begin{equation*}
\left(\nabla^2\varphi(x)\right)^{-1}\leq
\begin{pmatrix}
x_1^2 & 0 & \cdots & 0 \\
0 & x_2^2 & \cdots & 0 \\
\vdots  & \vdots  & \ddots & \vdots  \\
0 & 0 & \cdots & x_n^2
\end{pmatrix}.
\end{equation*}
Set $g(x)=f(\pi(x)).$ By the Brascamp-Lieb inequality, 
\begin{equation*}
\var_{e^{-\varphi}}(g)\leq\int_{\mathbb{R}^n_{+}} \left[ \left(\nabla^2 \vphi\right)^{-1}\nabla g\cdot\nabla g \right] e^{-\varphi(x)}\ dx
\leq\int_{\mathbb{R}^n_{+}}\sum_{i=1}^n x_i^2|\partial_i g(x)|^2 e^{-\varphi(x)}\ dx.
\end{equation*}
The corollary follows since
\begin{equation*}
\var_{e^{-\varphi}}(g)=\var_{e^{-\psi}}(f).
\end{equation*}
and since when $y=\pi(x)=(x_1^2,\cdots,x_n^2)$ we have
\begin{equation*}
x_i \partial_i g(x)= 2 y_i \partial_i f(y). \qedhere 
\end{equation*}
\end{proof}
\begin{exercise} 
If $\psi: \RR^n_+ \rightarrow \RR$ is convex and increasing in all of the coordinate directions, then $\psi$ is $p$-convex for $p=1/2$,
i.e., $\psi(x_1^2,\ldots,x_n^2)$ is convex in the orthant.
\end{exercise} 
A function $\psi:\RR^n \rightarrow \RR$ is invariant under coordinate reflections (a.k.a. unconditional) if
$$ \psi(x_1,\ldots,x_n) = \psi(|x_1|,\ldots,|x_n|) \qquad \text{for all} \ x \in \RR^n. $$
If $\psi$ is moreover convex, then $\psi|_{\RR^n_+}$ is increasing in all coordinate directions. Similarly a random vector is called unconditional if its law is invariant 
under reflection by coordinate hyperplanes. When the vector has a density this amounts 
to saying the density is unconditional in the above sense.    
The following thin-shell bound is from \cite{K_uncond}.  

\begin{corollary} Suppose that $X$ is a random vector that is log-concave, isotropic and unconditional in $\RR^n$. Then,
$$\var(|X|^2) \leq Cn. $$
\end{corollary}

\begin{proof} According to the exercise the density of $X$ 
is of the form $\e^{-\psi}$, where $\psi$ is $p$-convex for $p=1/2$. 
Corollary~\ref{proposition-paris-1} applies and we get
\begin{align*}
\var(|X|^2) \leq 4 \sum_{i=1}^n \EE X_i^2 (2 X_i)^2 = 16 \sum_{i=1}^n \EE X_i^4 \lesssim \sum_{i=1}^n (\EE X_i^2)^2 = n . 
\end{align*}
where we used reverse H\"older inequalities in the last passage.
\end{proof}
It should be noted that the result is optimal, in the sense that 
there exist unconditional isotropic log-concave random vectors $X$ 
for which $\var( \vert X\vert^2)$ is of order $n$, see the exercise below. 
We also remark that as of October 2024, the state of affairs is that the KLS conjecture is still open already in the particular case of unconditional convex bodies. A logarithmic bound for the Poincar\'e constant in this case is known for years, see \cite{K_uncond}, and it is subsumed by recent bounds for the general case.

\begin{exercise} \label{ex_thinshell} 
If $X$ is a standard Gaussian vector in $\RR^n$,  
\[ 
\var ( |X|^2 ) = 2n.  
\] 
\end{exercise} 
\pagebreak

\section{Gaussian localization}
\label{sec_gl}

In the previous section we discussed localization of a log-concave measure into {\it needles}, one-dimensional segments.
We proceed by discussing Gaussian localization, decomposing the given measure 
into a mixture of measures, each of which involves multiplying the given measure by a Gaussian.
The Gaussians bring with them a wealth of connections and elegant formulae, as we see below. 
The method was invented by Ronen Eldan \cite{Eldan1} and it is coined {\it Eldan's Stochastic Localization}.
We first present a rather degenerate case of Eldan's method, 
in which the time parameter is somewhat fixed, so that the method does 
not require stochastic processes.

\medskip Let $Z$ be a standard Gaussian random vector in $\RR^n$, of mean zero and 
identity covariance matrix $\id$. 
Recall that for $s > 0$ we write $\gamma_s$ for the density of $\sqrt{s} \cdot Z$.
Let $X$ be a log-concave random vector in $\RR^n$ independent of $Z$, with density $\rho$. For $s \geq 0$ consider the random vector
$$ Y_s = X + \sqrt{s} Z $$
whose density is $ \rho * \gamma_s$. 

\medskip One could think of $(Y_s)$ as a process parameterized by $s$, perhaps as a Brownian motion starting at the initial distribution of $X$. This point of view, with the time reversal $t = 1/s$, is emphasized in Section \ref{sec_esl}. In the present lecture do not consider a stochastic process parameterized by $s$, and view $s > 0$ as a parameter whose value will be fixed later on.  One of the simplest examples of Gaussian localization of the probability density $\rho$ is given by the following:

\begin{proposition} Fix $s > 0$. For each $y \in \RR^n$, consider the probability density 
$$ \rho_{s, y}(x) = \frac{\rho(x) \gamma_s(x - y)}{\rho * \gamma_s(y)}, $$
which we view as a localized ``Gaussian needle'' or ``Gaussian piece'' relative to $\rho$. Then the original density $\rho$ is a certain 
average of these Gaussian needles:
$$ \rho = \EE \rho_{s, Y_s}. $$
One says that this is a disintegration of $\rho$ into the localized Gaussian pieces $(\rho_{s,y})_{y \in \RR^n}$.
\label{prop_1632}
\end{proposition}

\begin{proof} The joint density of $(X, Y_s)$ in $\RR^n \times \RR^n$ is
$$ (x,y) \mapsto \rho(x) \gamma_s(y-x). $$
The family of densities $\rho_{s,y}$ give us the conditional distribution of $X$ with respect to $Y_s$. That is, for any test function $f(x,y)$,
$$ \int_{\RR^n}\int_{\RR^n}  f(x,y) \rho(x) \gamma_s(y-x) dx dy = \int_{\RR^n} \left[ \int_{\RR^n} f(x,y) \rho_{s,y} (x) dx \right] \rho * \gamma_s(y) dy $$
In particular, if the function $f(x,y)$ depends only on $x$, we get
\[
 \int_{\RR^n} f \rho = \int_{\RR^n} \left[ \int_{\RR^n} f \rho_{s,y} \right] \rho* \gamma_s(y) dy = \EE \int_{\RR^n} f \rho_{s,Y_s}.  \qedhere
 \]
\end{proof}

From the proof of Proposition \ref{prop_1632} we see that the densities $\rho_{s,y}$ give us the conditional distribution of $X$ with respect to $Y_s$. The conditional expectation operator is 
denoted by 
$$ Q_s f(y) = \int_{\RR^n} f \rho_{s, y}, $$
whenever the integral converges. Thus 
$$ Q_s f(Y_s) = \EE \left[ f(X) | Y_s \right]. $$

\medskip 
Assume that the original density $\rho$ is log-concave. Then each of the elements $\rho_{s,y}$ in the decomposition is 
more log-concave than the Gaussian $\gamma_s$. We have thus expressed our log-concave density 
as a mixture of measures that are uniformly log-concave. This decomposition is determined by the 
choice of the parameter $s > 0$. 

\medskip The critical value of $s$ turns out to be $s \sim C_P(X)$. 
Roughly speaking, for much smaller values of $s$, we decompose into highly localized measures, maybe even resembling Dirac masses. For much larger values of $s$ the decomposition is trivial for another reason: the localized pieces 
resemble the original measure. Abbreviate 
$$ \rho_s = \rho_{s, Y_s}, $$
a random probability density. Recall that $\EE \rho_s = \rho$ by Proposition \ref{prop_1632}. As usual, for a function $f$ on $\RR^n$ we write 
$$ \var_{\rho_s}(f) = \int_{\RR^n} f^2 \rho_s - \left(\int_{\RR^n} f \rho_s \right)^2, $$
provided that the integrals converge. Similarly, we also write $\var_{\rho}(f) = \var f(X)$. Then by the law of total variance,
\begin{equation}\label{eq_lawoftotal}
\var f(X) = \EE\var(f(X) | Y_s) +\var( \EE(f(X) | Y_s) ) = \EE \var_{\rho_s}(f) +\var(Q_s f(Y_s)) . 
\end{equation} 
When $s \gtrsim C_P(X)$, it is the first summand that is dominant:

\begin{lemma} Assume that $X$ is log-concave. 
For any $s > 0$ and a function $f$ on $\RR^n$ with $\EE f^2(X) < \infty$, 
	$$ \EE \var_{\rho_s}(f) \leq  \var_{\rho}(f) \leq \left( 2 + \frac{C_P(X)}{s} \right)  \EE \var_{\rho_s}(f). $$
	\label{lem_1131}
\end{lemma}

\begin{proof} We need to show that $\var Q_s f(Y_s) $ is not much larger than $\EE \var_{\rho_s}(f)$. To this end, we will use the Poincar\'e inequality 
for the random vector $Y_s$. By the subadditivity property of the Poincar\'e constant
(see exercise~\ref{exo_sub}
	$$C_P(Y_s) = C_P(X + \sqrt{s} Z) \leq C_P(X) + C_P(\sqrt{s} Z) = C_P(X) + s. 
$$
Hence
$$ \var( Q_s f(Y_s) )\leq (C_P(X) + s) \cdot \EE |\nabla Q_s f (Y_s)|^2. $$
Recall that 
$$ Q_s f(y) = \int_{\RR^n} \rho_{s, y}(x) f(x) dx =  \int_{\RR^n} \frac{\rho(x) \gamma_s(x - y)}{\rho * \gamma_s(y)} f(x) dx. $$
Differentiating a Gaussian is easy, we have $\nabla \gamma_s(x) = -\gamma_s (x) \cdot x/s$. It follows that 
$$ \nabla Q_s f(y) =  \int_{\RR^n} \frac{x - a_{s}}{s} \rho_{s, y}(x) f(x) dx, $$
where $a_s = a_{s,y} = \int_{\RR^n} x \rho_{s,y}(x) dx$ is the barycenter of the local measure $\rho_{s,y}$. 
Write $A_s = A_{s,y} =\cov(\rho_{s,y})$. 
By the Cauchy-Schwartz inequality, for $\theta \in S^{n-1}$,
\[
\begin{split}  
\nabla Q_s f(y) \cdot \theta & = \int_{\RR^n} \frac{(x - a_s) \cdot \theta}{s} \rho_{s, y}(x) f(x) dx  \\
& \leq \frac{1}{s} \sqrt{ \int_{\RR^n} \left| (x - a_s) \cdot \theta \right|^2 \rho_{s,y}(x) dx }
\sqrt{ \var_{\rho_{s,y}}(f) } \\ 
& \leq \frac{1}{s} \sqrt{ \| A_s \|_{op} } \cdot \sqrt{ \var_{\rho_{s,y}}(f) }. 
\end{split} 
\] 
Then by taking the supremum over $\theta \in S^{n-1}$,
$$ \var (Q_s f(Y_s))  \leq  \frac{C_P(X) + s}{s^2} \cdot \EE \left( \| A_s \|_{op}  \var_{\rho_{s}}(f) \right). $$
However, the random probability density $\rho_{s}$ is always more log-concave than the Gaussian $\gamma_s$, 
and hence $A_s \leq s\cdot \id$. Consequently,
$$\var (Q_s f(Y_s))  \leq \frac{C_P(X) + s}{s} \cdot \EE \var_{\rho_{s}}(f). $$
This, together with~\eqref{eq_lawoftotal}, proves the proposition. 
\end{proof}

To summarize, for $s \gtrsim C_P(\mu)$, the local measure $\rho_s$ is typically close enough 
to the original measure, so the variance of any fixed function with respect to $\rho$ 
is roughly the averaged variance with respect to $\rho_s$. 

\medskip 
{\bf Remark.} By differentiating with respect to $s$, one may improve upon Proposition \ref{lem_1131} in two respects. First, it turns out that log-concavity is actually not needed in Proposition \ref{lem_1131}. 
It is proven in Klartag and Ordentlich \cite{KO} that for any random vector $X$ and a function $f$ with $\EE f^2(X) < \infty$,
\begin{equation}   \var_{\rho}(f) \leq \left( 1 + \frac{C_P(X)}{s} \right) \EE \var_{\rho_s}(f).
\label{eq_1821} \end{equation}
This is a better bound than that of Lemma \ref{lem_1131}. 

\begin{corollary} For any $s > 0$, setting $\alpha = s / C_P(X)$,
	$$ C_P(X) \leq C \left( 1+\frac{1}{\alpha} \right) \cdot \EE C_P(\rho_s),
$$
	where $C > 0$ is a universal constant. \label{cor_1804}
\end{corollary}

\begin{proof} 	Let $f: \RR^n \rightarrow \RR$ be a $1$-Lipschitz function with
$$ \var_{\mu}(f)  \geq c \cdot C_P(X), $$
whose existence in guaranteed by Corollary \ref{cor_Emil} due to E. Milman. 
By Proposition \ref{lem_1131} and the Poincar\'e inequality,
\[
\begin{split}
\var_{\mu}(f) & \leq \left(2 + \frac{1}{\alpha} \right) \EE \var_{\rho_s}(f) \\
& \leq \left(2 + \frac{1}{\alpha} \right)  \EE\left( C_P(\rho_s) \cdot \int_{\RR^n} |\nabla f|^2 \rho_s \right) \\
&\leq \left(2 + \frac{1}{\alpha}\right) \EE C_P(\rho_s).  \qedhere
\end{split} 
\]
\end{proof}

Thus, in order to bound the Poincar\'e constant of $X$, we may apply Gaussian localization with $s \gtrsim C_P(\mu)$
and try to bound the Poincar\'e constant of $\rho_s$. An advantage of $\rho_s$ over $\rho$ is that $\rho_s$ is more log-concave 
than the Gaussian $\gamma_s$. Hence, by the improved log-concave Lichnerowicz inequality, which is Theorem \ref{thm_1022} above,
$$ C_P(\rho_s) \leq \sqrt{ s \cdot \| A_s \|_{op}} $$
where we recall that $A_s =\cov(\rho_s)$. Therefore, Corollary \ref{cor_1804} leads to another corollary:

\begin{corollary} For any $s > 0$, 
$$ C_P(X) \leq C \left(1 + \frac{C_P(X)}{s} \right) \cdot \sqrt{ \EE  \| A_s \|_{op} \cdot s}.  $$
\label{cor_1652}
\end{corollary}

What do we know about $\EE \| A_s \|_{op}$? Assume from now on that $X$ is log-concave and isotropic, so for large $s > 0$ we might expect 
$A_s$ to be roughly $\cov(X) = \id$. However, the operator norm involves a supremum, and this complicates matters. The evolution 
of the operator norm of the covariance matrix is analyzed 
in great detail in Section \ref{sec_esl2} using stochastic processes 
and computations involving $3$-tensors, leading to the following 
estimate. 

\begin{theorem}\label{thm_something}
Define
$$ s_0 = \min \{ s > 0 \, ; \, \forall r > s, \ \EE \| A_r \|_{op} \leq 5 \}. $$
Then, 
\begin{equation}  s_0 \leq C \log^2 (n+1) \label{eq_1041} \end{equation}
where $C > 0$ is a universal constant. This bound utilizes the improved Lichnerowicz inequality, proven only recently.
A slightly older bound that suffices here (e.g. \cite{K_chen,KL}) is $$ s_0 \leq C \log(n+1) \cdot \sup C_P(\mu) $$
where the supremum runs over all isotropic, log-concave probability measures $\mu$ on $\RR^n$. 

Moreover, $s_0 \geq c \log n$ in some examples, in particular 
when $1+X_1,\dotsc,1+X_n$ are i.i.d. Exponential random variables with 
parameter $1$.
\medskip \label{thm_1653}
\end{theorem}

One could conjecture that stochastic processes and pathwise analysis of are 
not essential for the proof of Theorem \ref{thm_1653}, 
and that an analytic proof is possible to find. 
There are other applications of stochastic localization which seem to rely heavily 
on pathwise analysis (e.g., the complex waist inequalities in \cite{K_complex}).
By using Theorem \ref{thm_1653} and Corollary \ref{cor_1652} with $s = C \log^2 (n+1)$ we thus arrive at

\begin{corollary}[``best known bound for KLS''] For any isotropic, log-concave random vector $X$ in $\RR^n$,
\begin{equation}  C_P(X) \leq C \log(n+1) \label{eq_1043} \end{equation}
where $C > 0$ is a universal constant.
\label{cor_1604}
\end{corollary}

\begin{proof} We have
$$ C_P(X) \leq C \left(1 + \frac{C_P(X)}{\log^2 (n+1)} \right) \cdot \sqrt{ \log^2 (n+1)},  $$
which implies (\ref{eq_1043}).
\end{proof}
\pagebreak

\section{A dynamic perspective on Gaussian localization} 
\label{sec_esl}

Formally when $s$ tends to $\infty$, the variable $X + \sqrt s G$ 
becomes independent of $X$, so the conditional law of $X$ given $X +\sqrt s G$
tends to the law of $X$. In this section we will study the 
dynamic of this measure-valued process as time $s$ evolves. 
 
\subsection{The Eldan equation}
The process solves a certain stochastic differential 
equation which was first considered by Eldan and which we present now.  
We are given a probability measure on $\RR^n$, and a standard Brownian motion 
$(W_t)$ on $\RR^n$. We consider the following infinite 
system of SDE whose unknown is the family $(p_t)$ 
of functions from $\RR^n$ to $\RR_+$: 
\[ 
\begin{cases}
p_0 (x) = 1  \\
d p_t (x) = p_t (x) \, (x - a_t) \cdot d W_t , 
\end{cases} 
\] 
where $a_t$ is the barycenter $p_t(x) \mu (dx)$, namely 
\[ 
a_t = \frac{ \int_{\RR^n} x \cdot  p_t(x) \, \mu (dx) }
{ \int_{\RR^n} p_t (x) \, \mu(dx) } . 
\]
Note that we have only one Brownian motion $(W_t)$ which is 
used for every $x$. Actually in Eldan's original paper~\cite{Eldan1}
the equation is slightly more intricate than that. Here we 
consider the simplified version that was introduced by Lee and Vempala~\cite{LV}. 

Since we have an equation for each $x$ and they are all coupled 
together by the condition on the barycenter, it is not at all clear 
that such a process should actually exists. Let us leave that 
matter aside for now, we will come back to that later on. 
Let us also take for granted the fact that $p_t(x) >0$ for 
all $t$, almost surely. 
The barycenter condition then ensures
that the total mass of $p_t \, d\mu$ remains constant. Indeed, at least
formally we have 
\[ 
d  \int_{\RR^n} p_t(x) \, \mu (dx) = 
\int_{\RR^n} d p_t (x) \, \mu(dx) = \left( \int_{\RR^n} (x-a_t) p_t (x) \, \mu (dx) \right) \cdot d W_t  ,  
\]
which is $0$ by definition of $a_t$. Therefore 
$p_t \, d\mu$ is a random probability measure for 
all time, and we call that measure $\mu_t$ from 
now on. The second feature is that $p_t(x)$ is 
a martingale for all $x$. In particular 
$\EE p_t (x) = p_0 (x) = 1$ for all $x$. Therefore 
the random measure $\mu_t$ 
equals $\mu$ on average
\[ 
\EE \mu_t = \mu .
\]  
The third observation is that the equation 
\[ 
d p_t (x) = p_t(x) (x - a_t) \cdot d W_t
\] 
can be solved explicitly. Indeed applying It\^o's formula to $\log p_t (x)$ we get 
\[ 
d \log p_t (x) = (x - a_t) \cdot d W_t - \frac 12 \vert  x - a_t \vert^2 \, dt , 
\] 
hence 
\[ 
p_t (x) = \exp \left( \int_0^t ( x - a_s ) \cdot  d W_s 
- \frac 12 \int_0^t \vert  x - a_s \vert^2 \, ds \right)  
= \exp \left ( c_t +  x \cdot \theta_t - \frac t2 \vert x\vert^2 \right) ,  
\] 
where $(c_t)$ and $(\theta_t)$ are certain random processes 
not depending on $x$. 
%Let us record for future 
%reference the expression for $\theta_t$:
%\begin{equation}\label{eq_thetat} 
%\theta_t = W_t + \int_0^t a_s \, ds . 
%\end{equation} 
This shows that
the density $p_t$ of $\mu_t$ with respect to $\mu$ 
is just a certain Gaussian factor. The linear term 
and the normalizing constant are random but the 
quadratic term is deterministic, equal to $\frac t2 \vert x\vert^2$. 
As a result if the original measure was log-concave then 
the measure $\mu_t$ is $t$-uniformly log-concave, almost 
surely. The process becomes more and more \emph{peaked} 
as $t$ grows. For this reason Eldan coined the name 
\emph{stochastic localization process}. It  
allows us to write a log-concave measure
as a mixture of $t$-uniformly log-concave measures. 
Moreover this mixture is constructed by solving a certain 
stochastic differential equation, 
so that its behavior over time can be somehow controlled using 
It\^o's formula. 
%Let us be a bit more precise. Let $f$ be some 
%test function. By formally exchanging the It\^o
%derivative with the integral sign we see that
%\begin{equation}\label{eq_df}
%d \int f \, d \mu_t = \left( \int f(x) (x-a_t) \, d\mu_t \right) \cdot d W_t.
%\end{equation} 
%Therefore in order to control the quadratic 
%variation of the It\^o process $\int f \, d\mu_t$, 
%we need to control the norm of the random vector 
%\[ 
%\int f(x) (x-a_t) \, d\mu_t . 
%\] 
%As will be clear later on this boils down 
%to bounding the norm of the covariance 
%of the random measure $\mu_t$. 

\subsection{Proper construction of a solution}
We will now give a rigorous construction of the 
stochastic localization process. As we said earlier
this process was introduced by Eldan~\cite{Eldan1} 
(a variant of it actually), it was used in a number of subsequent works 
\cite{LV,chen,KL}. The construction that we give here is 
somewhat original, but very much inspired by
Klartag-Putterman~\cite{KP}. 

Start with a standard $n$-dimensional 
Brownian motion $(\theta_t)$ defined on some 
probability space $(\Omega,\mathcal F, \PP)$
equipped with a filtration $(\mathcal F_t)$. 
This is an odd name for a Brownian motion, 
you'll see the reason for this choice shortly. 
Observe that for every fixed $x\in\RR^n$ the process $(E_t)$ 
given by 
\[ 
E_t = \exp \left( x\cdot \theta_t - \frac t2 \vert x\vert^2 \right) 
\]
is a martingale. Indeed, since the Brownian motion 
has independent and stationary increments, for every $s\leq t$, the ratio 
$E_t/E_s$ is independent of whatever happens before times $s$ and has 
expectation $1$.  
Using Fubini, we deduce that given a test function $f$, 
the process $(N_t)$ given by 
\begin{equation}\label{eq_Nt}
N_t = \int_{\RR^n} f(x) \cdot \exp \left( x\cdot \theta_t - \frac t2 \vert x\vert^2 \right) \, \mu (dx). 
\end{equation} 
also is a martingale. In particular its expectation is 
what we have at time $0$, namely $\int_{\RR^n} f \, d\mu$. 
Let $\mu_t$ be the random probability measure on $\RR^n$ 
given by 
\begin{equation} \label{eq_mut123} 
\mu_t (dx) = \frac 1{D_t} \exp \left( x \cdot \theta_t - \frac t2 \vert x\vert^2 \right) \, \mu (dx) , 
\end{equation}
where $D_t$ is the normalization constant, namely  
\begin{equation} \label{eq_Dt} 
D_t = \int_{\RR^n} \exp \left( x \cdot \theta_t - \frac t2 \vert x\vert^2 \right) \, \mu (dx)
\end{equation} 
We can then rewrite $N_t$ as 
\begin{equation}\label{eq_Nt2}
N_t = D_t \cdot  \int_{\RR^n} f(x) \, d \mu_t . 
\end{equation} 
We will interpret the normalizing factor $D_t$ as 
a change in the probability 
space. 

Fix a large but finite time horizon $T$. Since 
the process $(D_t)_{t\leq T}$ is a positive martingale with expectation $1$,
we can define a new probability 
measure $\mathbb Q$ on $(\Omega,\mathcal F)$ by saying that $\mathbb Q$ 
has density $D_T$ with respect to $\PP$. 
Then it is easy to see that a process $(X_t)$ defined on $[0,T]$ 
is a $\mathbb Q$-martingale
if and only if the process $(X_t D_t)$ is a $\mathbb P$-martingale.
Recall that the process $(N_t)$ given by~\eqref{eq_Nt} was a $\mathbb P$-martingale.
In view of~\eqref{eq_Nt2} we thus the following.
\begin{fact} 
For any test function $f$, the process $(M_t)$ given by 
$M_t = \int_{\RR^n} f \, d\mu_t$
is a $\mathbb Q$-martingale. 
\end{fact} 
Getting an It\^o equation for this process is a little 
more involved. It relies on the Girsanov change of measure 
formula which we spell out now. 
\begin{proposition}[Girsanov change of measure] 
If $X$ is a $\mathbb P$-local martingale on $[0,T]$ 
then the process $\tilde X$ given by 
\[ 
\tilde X_t = X_t - \int_0^t \frac{ d \langle X , D \rangle_s }{D_s} 
\] 
is a $\mathbb Q$-local martingale on $[0,T]$. 
Moreover, $\tilde X$ 
and $X$ have the same quadratic variation. 
In particular if $X$ is a $\mathbb P$-Brownian motion 
on $[0,T]$ then $\tilde X$ is a $\mathbb Q$-Brownian motion
on $[0,T]$.  
\end{proposition} 
\begin{remark} 
The bracket denotes the quadratic covariation of continuous 
semimartingales.  
Note that the quadratic variation under $\mathbb P$ is the same 
as the quadratic variation under $\mathbb Q$. Indeed, quadratic 
variation is defined as the limit in probability of sums of 
squared increments along partitions of the interval whose mesh sizes
tend to $0$. This is easily seen to be left unchanged 
by an absolutely continuous change of probability measure. 
\end{remark} 
\begin{remark} 
In the statement the process $X$ is $\RR$-valued but the result also 
works for vector valued martingales by applying it to each coordinate. 
\end{remark} 
\begin{proof}
This is a very standard tool in stochastic calculus, 
we only give a very brief sketch of proof and refer 
to \cite[section IV.38]{RW} for more details. This 
amounts to proving that $\tilde X D$ is a $\mathbb P$-martingale. 
But, from It\^o's integration by parts formula we get 
\[
\begin{split} 
d ( \tilde X D ) 
& = (d \tilde X ) D + \tilde X (d D) + d \langle \tilde X, D \rangle \\
& = (d X ) D - d \langle X,D\rangle + \tilde X (d D) + d \langle X, D \rangle .  
\end{split} 
\] 
The quadratic covariation of $X$ and $D$ thus cancels out and we are left 
with martingale increments only. 
\end{proof} 
Coming back to our situation, we see that the change of 
measure is of the form 
\begin{equation}\label{eq_Dt2}
D_t = \exp ( \phi (t,\theta_t) )  
\end{equation}
where $\phi \colon \RR_+ \times \RR^n \to \RR$ is the function 
given by
\begin{equation}\label{eq_phi}
\phi (t,\theta) = \log \left( \int_{\RR^n} \exp \left( \langle x, \theta\rangle -\frac t2 \vert x\vert^2  \right) \, \mu (dx) \right) . 
\end{equation} 
This is not quite essential but let us assume for simplicity 
that $\e^{ x\cdot \theta}$ is $\mu$-integrable for all $\theta\in\RR^n$
in which case $\phi$ is smooth on $[0,\infty) \times \RR^n$. From 
It\^o's formula we get
\[ 
d D_t = D_t \, \nabla \phi (t,\theta_t) \cdot d \theta_t . 
\] 
Here and in the sequel, $\nabla$ and $\Delta$ always mean gradient and 
Laplacian with respect to the space 
variable. The derivative with respect to the time variable will 
be denoted $\partial_t$. 
Then from Gisanov, we see that the process $(W_t)$ given by 
\[
\begin{split}  
d W_t 
& = d \theta_t - \frac{ d \langle \theta_t , D_t\rangle }{D_t } \\
& = d \theta_t - \nabla \phi (t,\theta_t) \,dt
\end{split}   
\] 
is a $\mathbb Q$-Brownian motion. We rewrite this equation as 
\begin{equation}\label{eq_eqtheta}
d \theta_t = d W_t + \nabla \phi (t,\theta_t) \, dt . 
\end{equation}
We are now in a position to prove the following. 
\begin{fact} 
The It\^o derivative of the $\mathbb Q$-martingale $M_t= \int_{\RR^n} f \, d\mu_t$ is given by 
\[
d M_t = \left( \int_{\RR^n} f(x) (x-a_t) \, d\mu_t \right) \cdot d W_t  , 
\] 
where $a_t = \int_{\RR^n} x \, d\mu_t$ is the barycenter of $\mu_t$. 
\end{fact} 
\begin{proof} 
First of all, 
by differentiating~\eqref{eq_phi} under the integral sign, we obtain 
$\nabla \phi ( t, \theta_t) = a_t$. 
We see $M_t = \int f \, d\mu_t$ as a function of $t$ and $\theta_t$,
denoted $F(t,\theta_t)$. By It\^o's formula and~\eqref{eq_eqtheta}, we have
\[ 
d M_t = \nabla F(t,\theta_t ) \cdot \left( d W_t + \nabla \phi (t, \theta_t) \, dt \right)+ \frac 12 \Delta F(t,\theta_t) \, dt + \partial_t F(t,\theta_t ) \, dt. 
\] 
The gradient of $F$ is 
\[ 
\nabla F (t, \theta_t ) = \int_{\RR^n} f(x)  ( x - a_t )  \, d\mu_t .   
\]
Moreover, since we have seen above that $(F(t,\theta_t))$ is 
a martingale for some filtration for which $(W_t)$ is also a martingale 
it must be the case that the $dt$ part above cancels out. 
It can indeed be checked that $F$ satisfies the following PDE
\[
\partial_t F = - \nabla F \cdot \nabla \phi -\frac 12 \Delta F .
\] 
This concludes the proof of the fact. 
\end{proof} 
\begin{remark} 
Strictly speaking this only gives a construction of the process $(\mu_t)$
on a bounded time interval $[0,T]$. This will be sufficient for our needs
but let us note that one could extend this construction to the 
whole half-line by some abstract argument \`a la Carath\'eodory. 
Beware though that  the change of measure is only absolutely continuous 
when we restrict our processes to a bounded time interval.
\end{remark}
As a byproduct of this construction we obtain a simple description of 
the law of the process $(\theta_t)$. This observation is not 
present in the works of Eldan, Lee-Vempala, and Chen. Its first 
explicit mention is in the paper of Klartag and Putterman.
\begin{proposition} \label{prop_thetalaw}
The process $(\theta_t)$ has the same law as the 
process $(tX+W_t)$, where $(W_t)$ is a standard Brownian motion, 
and $X$ is a random vector having law $\mu$ independent of $(W_t)$. 
\end{proposition} 
\begin{proof} 
Recall that we only work on some finite time interval $[0,T]$. 
The process $(\theta_t)$ is a Wiener process perturbed by some 
absolutely continuous change of probability measure: $d \mathbb Q = D_T d \mathbb P$.
From the equation~\eqref{eq_Dt2}, 
we see that this can be reformulated 
as follows: The law of the process $(\theta_t)$ is absolutely continuous 
with respect to the Wiener measure, 
with density $w \mapsto \e^{\phi ( T, w_T )}$. \\
Now set $\eta_t = t X + W_t$ for every $t\leq T$.
Conditionally on the vector $X$, the process 
$(\eta_t)$ is just a Brownian motion plus a constant speed 
deterministic drift. As a result its law is explicit, given by 
a very basic version of the Cameron-Martin formula, see for instance~\cite[section 40]{RW}.
For any test function $H$ we have 
\[ 
\EE ( H(\eta) \mid X ) = \EE \left( 
H (W) \cdot \e^{  X\cdot W_T  - \frac T2 \vert X \vert^2 } \mid X \right). 
\] 
Taking expectation again, and using Fubini and the definition~\eqref{eq_phi} of $\phi$, 
we obtain 
\[ 
\EE H(\eta) = \EE  H(W) \cdot \e^{ \phi ( T , W_T ) } . 
\] 
Therefore, the law of $(\eta_t)$ also has density 
$w \mapsto \e^{ \phi ( T, w_T)}$ with respect to the Wiener measure. 
\end{proof} 
%
%
%One may wonder why we did not simply 
%set $\theta_t = t X + W_t$ and proved that the corresponding measure 
%valued process $(\mu_t)$ is a martingale using It\^o's formula. 
%This simplified approach does not work: While the process $(\mu_t)$ 
%is then a martingale with respect to its own filtration,  
%it is generally not a martingale with respect to the natural filtration of the 
%Brownian motion $(W_t)$. 
Let us illustrate this result with a simple 
example where we can compute everything explicitly. 
\begin{example} 
In dimension $1$, take $\mu$ to be the standard  Gaussian 
measure. In that case 
we have an explicit formula for $\phi$ namely
\[ 
\phi (t, \theta ) = \frac{ \theta^2 }{2(1+t)} -\frac 12 \log (1+t) ,   
\]
which gives $\nabla \phi (t,\theta) = \frac \theta {1+t}$. 
The equation for the tilt process $(\theta_t)$ is thus 
\[ 
d \theta_t = d W_t + \frac {\theta_t}{1+t} \, dt ,  
\] 
which can be solved explicitly: 
\[ 
\theta_t = (1+t) \int_0^t \frac{dW_s}{1+s} . 
\] 
%Observe that the barycenter $(a_t)$ of $\mu_t$ satisfies 
%\[ 
%a_t =\nabla \phi (t,\theta_t)  = \frac{\theta_t}{1+t} = \int_0^t \frac{dW_s}{1+s}  
%\]
%which is indeed a martingale with respect to the 
According to our theorem this should 
have the same law as the process $(\eta_t)$ 
given by $\eta_t = W_t + t X$, where $X$ is a standard 
Gaussian variable independent of $(W_t)$. 
Of course this can be checked directly in 
this case. Indeed, both processes clearly are centered 
Gaussian processes and the two covariance structures coincide, 
since
\[ 
\EE \theta_s \theta_t = \EE \eta_s \eta_t = st + s \wedge t . 
\]
for every $s,t>0$. 
We leave this computation as an exercise.  
\end{example} 
\subsection{Time reversal}
We will now clarify the link between the stochastic localization of 
Eldan and the Gaussian localization of the previous section. Recall 
the definition~\eqref{eq_mut123} of $\mu_t$. Letting $\rho$ be the 
density of $\mu$ with respect to the Lebesgue measure we can reformulate
this definition as 
\begin{equation}\label{mu_t}
\int_{\RR^n} f \, d\mu_t = \frac{ \int_{\RR^n} f(x)\rho(x) \exp\left( \theta_t \cdot x - t \vert x\vert^2/2 \right) \, dx } { \int_{\RR^n} \rho(x) \exp \left( \theta_t \cdot x - t \vert x\vert^2/2 \right) \, dx } , 
\end{equation} 
for any test function $f$. 
Let us introduce the heat semi-group 
\[ 
P_t f (x) = \EE f ( x + B_t ) = f * g_t
\] 
where $g_t (x) = (2\pi t)^{-n/2} \e^{- \vert x\vert^2 / 2t }$ is
the density of the Gaussian measure with mean $0$ and covariance matrix 
$t \cdot \id$. Warning: from now on $(P_t)$ will denote the heat semigroup, 
and not the Langevin semigroup associated to $\mu$ from section~\ref{sec_semigroup}. 
Then~\eqref{mu_t} rewrites as  
\[ 
\int_{\RR^n}  f \, d\mu_t = \frac { P_{1/t} ( f \rho )  } { P_{1/t} \rho } 
\left( \frac{\theta_t}t \right)  . 
\] 
Now set $s = 1/t$. By Proposition~\ref{prop_thetalaw} we 
have the following equality in law 
\[
\frac {\theta_t} t  = \frac{t X + B_t}{t} =  X + s B_{1/s} .   
\] 
Since $\tilde B_s := s B_{1/s}$ is again a standard Brownian motion
(this is the time reversal property of the Brownian motion) 
we obtain the following: Up to the time reversal $t= 1/s$, the process  
$(\int f \, d\mu_t)_{t\geq 0}$ has the same distribution as
$(Q_s f ( X + B_s ))_{s\geq 0}$, where $Q_s$ is the operator 
defined by  
\[ 
Q_s f = \frac{ P_s ( f\rho ) }{ P_s \rho } . 
\]  
Moreover, using the fact that the heat semigroup is self-adjoint in $L^2 (dx)$ 
it is easy to see that
\[ 
Q_s f ( X + B_s ) = \EE [ f (X) \mid X+B_s ] .
\]  
Putting everything together we see that the stochastic localization 
process $(\mu_t)$ initiated from $\mu$ has the same law as the measure-valued 
process obtained by looking at the conditional law of $X$ given 
$X+B_s$ and then reversing time by setting $t=1/s$. In particular 
if we take a snapshot at some fixed time $s = 1/t$, then for 
every test function $f$ the variable $\int_{\RR^n} f \, d\mu_t$
has the same law as $\EE ( f(X) \mid X + \sqrt s G )$ where $G$ is 
a standard Gaussian vector independent of $X$. 
\begin{remark}
It is clear from this description that this process was looked at 
in many other contexts. Apparently it is an important tool in filtering theory, and it is also very much related to what Bauerschmidt, Bodineau and Dagallier~\cite{BBD}
call the Polcinski equation, which is used in their recent series of works 
on log-Sobolev inequalities for various particles systems. 
\end{remark}
\pagebreak

\section{Estimates for the conditional covariance}  
\label{sec_esl2}

Our main task fin this section is to prove Theorem~\ref{thm_something}, which 
we reformulate here for convenience. 
\begin{theorem} \label{thm_main} Let $X$ be log-concave and isotropic,
i.e. $\EE X=0$ and $\cov(X) = \id$, and let $G$ be a standard Gaussian 
vector independent of $X$, then 
\[ 
\EE \Vert \cov ( X \mid X + \sqrt s G ) \Vert_{op} \lesssim 1
\] 
for every $s$ such that $(\log n)^2 \lesssim s$. 
\end{theorem} 
Recall that $\lesssim$ means up to a universal multiplicative constant 
(here a factor $10$ is probably OK). 
Also the norm is the operator norm, which is also the maximal 
eigenvalue. 

We shall derive this by combining arguments from Eldan~\cite{Eldan1}, Lee-Vempala~\cite{LV}, Chen~\cite{chen}, Klartag-Lehec~\cite{KL}, 
with the improved Lichnerowicz inequality from Section~\ref{sec_lich}. 
Actually the improved Lichnerowicz allows 
to bypass many ideas of the aforementioned papers. 

Recall also that we have seen in Section~\ref{sec_gl} that the improved Lichnerowicz
allows to show that if $X$ is log-concave and if  
\[ 
\EE \Vert \cov ( X \mid X + \sqrt s \cdot G ) \Vert_{op} \lesssim 1, 
\] 
for all $s\geq s_0$ then $C_P (X) \lesssim \sqrt{s_0}$. 
So the theorem indeed yields 
\[ 
C_P (X) \lesssim \log n ,
\] 
which is the current best bound for the Poincar\'e constant 
of an isotropic log-concave random vector. 

To prove the theorem we will reverse time and 
rewrite everything in terms of the stochastic localization process $(\mu_t)$ 
associated to $\mu$. We thus rephrase Theorem~\ref{thm_main} as 
follows. 
\begin{theorem}
\label{thm_cov123}
Suppose $\mu$ is a log-concave and isotropic probability measure on $\RR^n$
and let $(\mu_t)$ be the stochastic localization process initiated at $\mu$. 
Then 
\[ 
\EE \Vert \cov ( \mu_t ) \Vert_{op} \lesssim 1 ,
\] 
for all $t \leq c \cdot (\log n)^{-2}$, where $c>0$ is a universal constant. 
\end{theorem} 
The point of this time reversal is that we can now control 
everything using It\^o's formula and some convexity 
inequalities.  
The proof of the theorem requires some preliminaries.
There will be a number of them, but taken individually, each of 
these is pretty easy. 
\subsection{The equation for the covariance} 
As we have seen in the previous section, 
for any test function $f$ 
the martingale $M_t = \int_{\RR^n} f \, d\mu_t$ satisfies 
\[
d M _t = \left( \int_{\RR^n} f(x) (x-a_t) \, d\mu_t \right) \cdot dW_t , 
\] 
where $(W_t)$ is some standard Brownian motion.
This obviously extends to vector valued functions. 
If $F \colon \RR^n \to \RR^k$ is 
a vector valued function that grows fairly reasonably 
at infinity then the process $(M_t)$ given by  
\[ 
M_t = \int_{\RR^n} F \, d\mu_t 
\]
is a martingale, and moreover 
\[ 
d M_t = \left( \int_{\RR^n} F (x)\otimes (x-a_t) \, d\mu_t \right) \cdot d W_t 
\]
A bit more 
explicitly, writing $x_i$ for the $i$-th coordinate of a vector 
$x\in \RR^n$ we have 
\begin{equation}\label{eq_dm}
d M_t = \sum_{i=1}^n  \left( \int_{\RR^n} F(x) (x-a_t)_i \, d\mu_t  \right)
\, d W_{t,i} .
\end{equation}  
\begin{lemma} \label{lem_At}
Let $a_t$ and $A_t$ be the barycenter and covariance 
matrix of $\mu_t$, respectively. Then
\[ 
\begin{split}
& da_t = A_t d W_t \\
& d A_t = \sum_{i=1}^n \left( \int_{\RR^n} (x-a_t)^{\otimes 2} \, (x-a_t)_i \, d\mu_t \right) \, d W_{t,i} - A_t^2 \, dt . 
\end{split} 
\]
\end{lemma}
This is obtained by applying~\eqref{eq_dm} to the tensors $F(x)=x$ and
$F(x) = x\otimes x$ and then rearranging the terms appropriately. The 
details are left as an exercise.
\\
This shows that the stochastic localization process has some moment generating 
property. The derivative for the barycenter is expressed in terms of the 
covariance, and the derivative for the covariance depends on $3$-tensors. 
\subsection{Some matrix inequalities} 
%
%we will proceed somewhat 
%differently and use some matrix martingale techniques. 
%Recall the Chernov bound: for a random variable $S$ we have 
%\[ 
%\PP ( S \geq u ) \leq \inf_{\lambda >0} \left\{ \e^{-\lambda u} \cdot \EE \e^{\lambda S} \right\}  ,  
%\] 
%This is particularly well suited when the variable $S$ 
%can be expressed as a martingale with sufficiently 
%well behaved increments. When $S$ is the 
%norm of a symmetric matrix a natural
%idea is to combine this with the inequality 
%\[ 
%\e^{ \lambda \Vert A\Vert } \leq \tr \, \e^{\lambda A} . 
%\] 
%The issue is that non commutativity makes the right hand side
%harder to evaluate than in the scalar case, even when $A$ is 
%nice, say a sum of independent increments. For instance, 
%the Golden-Thompson inequality asserts that 
%for any two symmetric matrices $A,B$ we have  
%\[ 
%\tr \, \e^{A+B} \leq \tr (\e^A \e^B) , 
%\] 
%but the natural extension of this for three of more matrices fails.
%Matrix martingale inequalities offer a way around 
%this type of issues, see these lectures notes by Joel Tropp~\cite{tropp}
%for an overview of the topic. 
%Unfortunately the inequalities developed there do not quite match 
%our needs, so we will reprove everything from scratch, but what 
%follows is very much related to these techniques.  
%%
%
\begin{lemma} 
Suppose $K,H$ are symmetric matrices, and $K$ is positive semi-definite. Then for every positive $\alpha,\beta$ we have
\[ 
\tr ( K^\alpha H K^\beta H ) \leq \tr ( K^{\alpha +\beta} H^2 ) .
\] 
\end{lemma} 
\begin{proof} 
Let $K = \sum \lambda_i x_i \otimes x_i$ be the spectral decomposition of $K$. Then
\[
\begin{split}
\tr ( K^\alpha H K^\beta H ) & 
= \sum_{ij} \lambda_i^\alpha \lambda_j^\beta \langle H x_i , x_j \rangle^2   \\
& \leq \sum_{ij} \lambda_i^{\alpha + \beta} \langle H x_i , x_j \rangle^2  \\
& = \sum_i \lambda_i^{\alpha + \beta} \vert H x_i \vert^2 \\
& = \sum_i \lambda_i^{\alpha + \beta} \langle H^2 x_i , x_i \rangle = \tr ( K^{\alpha + \beta } H^2 ) . 
\end{split}   
\]
The only inequality in the above display follows from Young's inequality
\[ 
\lambda_i^\alpha \lambda_j^\beta \leq \frac{\alpha}{\alpha+\beta} \lambda_i^{\alpha +\beta}
+ \frac{\beta}{\alpha+\beta} \lambda_j^{\alpha +\beta} ,  
\]
and the fact that the expression $\langle H x_i , x_j \rangle^2$ is symmetric in $i$ and $j$. 
\end{proof} 
\begin{corollary} \label{cor_matrix} 
Let $\phi$ 
be the map defined on the space $S_n (\RR)$ 
of symmetric matrices 
by $\phi(A) = \tr \, \e^A$. 
Then for every symmetric matrices $A,H$ we have 
\[ 
\nabla^2 \phi ( A) ( H ,H ) \leq \nabla \phi  ( A ) \cdot H^2  = \tr ( \e^A H^2 ) ,   
\]
where $\nabla^2 \phi(A)$ stands for the Hessian matrix at $A$, 
viewed as a bilinear form on $S_n (\RR)$. 
\end{corollary} 
\begin{proof} 
Assume first that the matrix $A$ is positive. Then 
by the previous lemma we have 
\[ 
\begin{split} 
\nabla^2 \phi(A) (H,H) 
& = \sum_{k\geq 1} \frac 1{k!} \sum_{l=0}^{k-1} \tr ( A^l H A^{k-1-l} H ) \\
& \leq \sum_{k\geq 1} \frac 1{k!} \cdot k \cdot \tr ( A^{k-1} H^2 ) 
= \tr ( \e^A H^2 ),   
\end{split} 
\]
which is the desired inequality. This argument does not work if $A$ 
has some negative eigenvalues, but observe that the function $\phi$ has the
property that 
\[ 
\phi ( A + t \cdot \id ) = \e^t \phi ( A ) 
\] 
By differentiating this equality with respect to $A$ we see also $\nabla \phi$ 
and $\nabla^2 \phi$ satisfy the same equation, which means that adding a 
multiple of identity to $A$ does not perturb the desired inequality. Therefore 
it is enough to prove it for positive $A$. 
\end{proof}
\subsection{Inequalities for $3$-tensors}
Recall the equation for $A_t$ 
\[ 
d A_t = \sum_{i=1}^n H_{i,t} d W_i  - A_t^2 \,dt ,  
\]  
where 
\[ 
H_{i,t} = \int_{\RR^n} (x-a_t)^{\otimes 2} (x-a_t)_i \, d\mu_t . 
\] 
Recall that $a_t$ is the barycenter of $\mu_t$. So the matrix 
$H_{i,t}$ is of the form $\EE X_i X^{\otimes 2}$ for some 
random vector with mean $0$. We need to control such 
quantities. This is the purpose of the next two lemmas. 
\begin{lemma} \label{lem_kappa2}
Let $X$ be a centered log-concave vector. Then 
\[ 
\sup_{ u \in \mathbb S^{n-1} } \{ \Vert \EE (X\cdot u) X^{\otimes 2} \Vert_{op} \} \lesssim \Vert \cov (X) \Vert_{op}^{3/2} .  
\] 
\end{lemma} 
\begin{proof} 
Let $u,v$ be unit vector and let $H_u =  \EE (X\cdot u) X^{\otimes 2}$. 
By Cauchy-Scwharz 
\[ 
H_u  v \cdot v = \EE (X \cdot u) (X\cdot v)^2 \leq ( \EE (X \cdot u)^2 )^{1/2} 
(\EE (X \cdot v)^4 )^{1/2} . 
\] 
Now we use log-concavity. The variable $X\cdot v$ is a log-concave 
random variable centered at $0$. We saw in the first section that moments 
of 1D log-concave measures satisfy a reverse H\"older inequality. In particular the 
fourth moment and second moment squared are of the same order. We thus 
get 
\[ 
H_u v \cdot v \leq C ( \EE (X \cdot u)^2 )^{1/2} 
\EE (X \cdot v)^2 \leq C \Vert \cov (X) \Vert_{op}^{3/2} .  
\] 
Taking the supremum in both $u$ and $v$ yields the result. 
\end{proof} 
\begin{lemma} \label{lem_kappa}
Let $X$ be a centered random vector satisfying the Poincar\'e inequality. 
Then
\[
\Vert \sum_{i=1}^n ( \EE X_i X^{\otimes 2} )^2  \Vert_{op}
\leq 4 C_P (X) \cdot  
\Vert \cov(X)\Vert_{op}^2 .
\]
\end{lemma} 
\begin{proof} 
Recall the definition of $H_u$. When $u$ 
is a coordinate vector $e_i$ we write $H_i$
rather than $H_{e_i}$. 
We need to show that for every unit vector $u$ 
\[ 
\sum_{i=1}^n H_i^2 u \cdot u \leq 4 C_P (X) \cdot  \Vert \cov(X)\Vert_{op}^2 . 
\] 
An elementary computation 
shows that $\sum H_i^2 u \cdot u = \tr ( H_u^2 )$.
Moreover, since $X$ is centered, we get from Cauchy-Schwarz and 
the Poincar\'e inequality
\[ 
\begin{split} 
\tr H_u^2 
& = \EE (X\cdot u) (H_u X \cdot X) \\
& \leq ( \EE (X \cdot u)^2 )^{1/2} \cdot ( \var (H_u X \cdot X) )^{1/2} \\     
& \leq ( \EE (X \cdot u)^2 )^{1/2} \cdot ( 4 C_P(X) \EE \vert H_u X\vert^2  )^{1/2} \\     
& = ( \cov(X)u \cdot u )^{1/2} \cdot ( 4 C_P(X) \tr ( H_u^2 \cov(X) ) )^{1/2} \\
& \leq \Vert \cov(X)\Vert_{op} \cdot ( 4 C_P(X) \tr ( H_u^2 ) )^{1/2} . 
\end{split} 
\]
Thus $\tr  H_u^2 \leq 4 C_P (X) \Vert \cov(X)\Vert_{op}^2$, 
which is the result. 
\end{proof} 
\begin{remark} 
We only applied Poincar\'e to a quadratic form 
so in a sense we only need a weak notion of 
Poincar\'e here. This observation will not be needed
in the subsequent analysis but it was crucial in the 
original work of Eldan.
\end{remark}  
\subsection{Freedman's inequality} 
Lastly we need a relatively classical deviation 
inequality for martingales, which is usually attributed to Freedman~\cite{freedman}. 
\begin{lemma}\label{lem_freedman}
Let $(M_t)_{t \geq 0}$ be a continuous local martingale
satisfying $M_0 = 0$. 
Then for every positive $u$ and $\sigma^2$ we have 
\[
\PP ( \exists t >0 \colon M_t \geq u \;\; \mbox{and} \;\; \langle M \rangle_t \leq \sigma^2 ) \leq \e^{ - u^2 / 2 \sigma^2 } .
\]
\end{lemma}
\begin{proof} We only sketch the argument and leave the details 
as an exercise. Start by proving the following statement:
If $(Z_t)$ is a square integrable martingale satisfying $\langle Z\rangle_t \leq \sigma^2$ for all $t>0$ and almost surely, then $Z_\infty = \lim_{t\to+\infty} Z_t$
exists and satisfies 
\[ 
\PP ( Z_\infty \geq u ) \leq \e^{- u^2 / 2 \sigma^2 } 
\]
for all $u >0$. Coming back to Freedman's inequality, introduce 
the stopping time 
\[ 
\tau = \inf \{ t >0 \colon \langle M\rangle_t > \sigma^2 \} 
\] 
and apply the above statement to the martingale $(M_t)$ stopped at time 
$\tau$. 
\end{proof} 
\subsection{The bound on the covariance matrix} 
\begin{theorem}\label{thm_At} 
Suppose $\mu$ is log-concave and isotropic on $\RR^n$, and let $(A_t)$ 
be the covariance process of the stochastic localization associated to $\mu$. 
Then
\[ 
\PP \left( \exists s\leq t \colon \Vert A_s\Vert_{op} \geq 2 \right) 
\leq \exp \left( - \frac 1{Ct} \right) ,  \qquad \forall t \leq  \frac1{C \log^2 n}. 
\] 
\end{theorem} 
\begin{remark} 
We will see later on that this bound is pretty much sharp. 
\end{remark} 
\begin{proof} A common method to control the norm
of a symmetric random matrix $A$ is to use the Schatten norm $(\tr A^p )^{1/p}$ where $p$ is an even integer of order $\log n$ 
as a proxy for $\Vert A\Vert_{op}$. This is what Eldan does in 
his 2014 paper. For some reason we prefer to use another proxy, 
namely 
\[ 
h_\beta (M) := \frac 1\beta \log \tr\, \e^{ \beta M } .  
\] 
Note that $h_\beta$ is a smooth function. Also  
\[ 
\lambda_{\rm max} (M) \leq \frac 1\beta \log \tr\, \e^{ \beta M } 
\leq \lambda_{\rm max} (M) + \frac{\log n}\beta . 
\] 
Therefore if $\beta$ is of order $\log n$ then 
$h_\beta(M)$ is approximately the same as the maximal eigenvalue of 
$M$, up to an additive constant. Recall the equation for $(A_t)$. 
From It\^o's formula we get (omitting the time variable) 
\[ 
d h_\beta ( A ) = \nabla h_\beta (A) \cdot \sum_{i=1}^n H_i d B_i
 - \nabla h_\beta (A) \cdot A^2 \,dt + 
 \frac 12 \sum_{i=1}^n \nabla^2 h_\beta ( A ) ( H_i,H_i)
\, dt  .  
\]
Let  
\[ 
M = \nabla h_\beta (A) = \frac {\e^{\beta A} } {  \tr ( \e^{\beta A} ) } , 
\]
and note that this is a positive semi-definite matrix of trace $1$.
Using Corollary~\ref{cor_matrix}, we see that the second derivative 
of $h_\beta$ satisfies 
\[ 
\nabla^2 h_\beta (A) (H_i,H_i) \leq \beta \tr ( M  H^2_i ) . 
\] 
Dropping some negative terms we finally arrive at 
\[
d h_\beta (A) \leq \sum_{i=1}^n \tr ( M H_i ) d B_i 
+ \frac \beta2  \tr \left( M \sum_{i=1}^n H_i^2 \right) \, dt  . 
\]
Let us deal with the absolutely continuous part. 
Since $M$ is positive and has trace $1$, 
we get from Lemma~\ref{lem_kappa} 
\[ 
\tr \left( M \sum_{i=1}^n H_i^2 \right) \leq \left\Vert \sum_{i=1}^n H_i^2 \right\Vert_{op} 
\leq 4 C_P(\mu_t)  \Vert A_t \Vert_{op}^2  . 
\] 
Recall that $(\mu_t)$ gets more and more log-concave 
along time. In particular if the original measure 
$\mu$ is log-concave then $\mu_t$ is $t$-uniformly log-concave,
almost surely. From the improved Lichnerowicz inequality, 
Theorem~\ref{thm_1022}, we get
\[ 
C_P (\mu_t ) \leq \left( \frac{ \Vert A_t\Vert_{op} } t \right)^{1/2} , 
\]
hence 
\[ 
d h_\beta (A) \leq \sum_{i=1}^n \tr ( M H_i ) d B_i 
+ \frac {2 \beta}{\sqrt t} \cdot \Vert A_t\Vert_{op}^{5/2} dt . 
\]
Let us now bound the quadratic variation of the martingale part. 
For any unit vector $u$, 
letting $H_u = \sum H_i u_i$ we get from Lemma~\ref{lem_kappa2}
\[
\sum_{i=1}^n\tr ( M H_i ) u_i = \tr (M H_u) \leq \Vert H_u  \Vert_{op} \leq C_0 \Vert A_t\Vert_{op}^{3/2} . 
\]
Therefore,  
\[ 
\sum_{i=1}^n \tr ( M H_i )^2 \leq C_0^2 \Vert A_t\Vert_{op}^3 . 
\] 
Let us summarize what we have obtained so far:
\begin{equation}\label{eq_sofar}
\begin{split}  
\Vert A_t\Vert_{op} 
\leq h_\beta ( A_t ) & \leq h_\beta ( A_0 ) + Z_t 
+ 2 \beta\int_0^t s^{-1/2} \Vert A_s\Vert_{op}^{5/2}\, ds \\
& = 1 + \frac{\log n}\beta  + Z_t
+ 2 \beta \int_0^t s^{-1/2} \Vert A_s\Vert_{op}^{5/2} \, ds 
\end{split}
\end{equation} 
where $(Z_t)$ is a continuous martingale starting from $0$ 
whose quadratic variation satisfies 
\begin{equation}\label{eq_Mt} 
\langle Z \rangle_t  \leq C_1 \int_0^t \Vert A_s \Vert_{op}^3 \, ds. 
\end{equation}
Now choose $\beta = 2 \log n$,
and assume that there exists $s\leq t$ 
such that $\Vert A_s\Vert_{op} \geq 2$. If $s$ is the smallest 
such time then before time $s$ the operator norm of $A$ 
is less than $2$, so by~\eqref{eq_sofar} 
\[ 
2 = \Vert A_s\Vert_{op} \leq \frac 32 + Z_s + C_2 s^{1/2} \log n 
\leq  \frac 32 + Z_s + C_2 t^{1/2} \log n
\] 
where $C_2$ is some constant. If $t$ is a sufficiently 
small multiple of $(\log n)^{-2}$ 
then the latest inequality implies that $Z_s \geq \frac14$.
Moreover, thanks to~\eqref{eq_Mt} we also have  
$\langle Z \rangle_s \leq C_3 s\leq C_3t$. 
Therefore, 
\[
\PP ( \exists s\leq t \colon \Vert A_s\Vert_{op} \geq 2 ) 
\leq \PP ( \exists s >0 \colon Z_s \geq \frac 14 \;\; \mbox{and} \;\; \langle Z\rangle_s \leq C_3t ).
\]
We conclude with Freedmann's inequality, Lemma~\ref{lem_freedman}.
\end{proof} 
Now we prove the bound for the expectation of $A_t$. 
\begin{proof}[Proof of Theorem~\ref{thm_cov123}]  
Since $\mu_t$ is $t$-uniformly 
log-concave, its covariance matrix is bounded above by $(1/t) \id$. 
This was already mentioned in Section \ref{sec_lich}.
Therefore we have $\Vert A_t\Vert_{op} \leq 1/t$, almost 
surely. As a result 
\[ 
\EE \Vert A_t\Vert_{op} \leq 2 + \frac 1t \PP( \Vert A_t\Vert_{op} > 2 ) . 
\]
Now we apply the latest theorem.  
Since $x \cdot\e^{- c_1 x }$ is a bounded function 
of $x$ we indeed get $\EE \Vert A_t \Vert_{op} \lesssim 1$ on the 
time range $[0, (C\log n)^{-2}]$.
\end{proof} 
\begin{remark} 
Instead of the improved Lichnerowicz inequality, we could 
have bounded $C_P (\mu_t)$ by the KLS constant. Namely if 
$C_n$ is the largest Poincar\'e constant  of an isotropic 
log-concave measure then it is easy to see that for any log-concave $X$ 
\[ 
C_P (X) \leq C_n \Vert \cov(X)\Vert_{op} .
\] 
Therefore 
\[ 
C_P (\mu_t) \leq C_n \Vert A_t\Vert_{op} . 
\] 
Using this estimate instead of the improved Lichnerowicz inequality 
leads to the following statement: 
\begin{equation}\label{eq_iiiii}
\EE \Vert \cov(X \mid X + \sqrt s G) \Vert_{op} \lesssim 1 ,
\quad \text{provided } C_n \log n \lesssim s . 
\end{equation} 
This is also good enough for the $\log n$ bound for $C_n$. 
Indeed we have seen that if the expected norm 
of $\cov(X \mid X + \sqrt s G)$ is of order $1$ for all $s\geq s_0$ 
then  $C_P (\mu) \lesssim \sqrt{s_0}$. So 
the latest display actually gives
$C_n  \lesssim \sqrt{C_n \log n}$,   
hence $C_n \lesssim \log n$.
\end{remark} 
\begin{remark} 
We will see later on an example of a measure for which 
$\Vert \cov(X \mid X + \sqrt s G) \Vert_{op}$ explodes at times $s = \log n$
(but is bounded at time $10 \log n$). In a sense this is evidence 
for the KLS conjecture $C_n \lesssim 1$ to be indeed correct. Namely if 
KLS is correct then~\eqref{eq_iiiii} is sharp and the above analysis
of the conditional covariance is essentially the best one can do. 
\end{remark} 

\subsection{Life before improved Lichnerowicz}
The improved Lichnerowicz estimate is only from 2023, and it was not available 
to Eldan, Lee-Vempala, Chen, Klartag-Lehec. Still these authors gave non trivial 
estimate on the KLS constant using this localization technique. 
In particular the KL bound was polynomial
in $\log n$. Here we will only say a few words about 
the original argument of Eldan.

First, let us derive a bound on the Poincar\'e constant 
of $\mu$ from a bound on the covariance 
of the stochastic localization in a slightly different manner 
than what was done in the previous section. 
Let $f$ be the function given by E. Milman's result 
(see section~\ref{sec_emil}). Namely $f$ is $1$-Lipschitz and such that 
\[ 
\var_\mu (f) \approx \Vert f\Vert_\infty^2  \approx C_P (\mu) . 
\]  
By the decomposition of variance 
\[ 
\var_\mu (f) 
= \EE \var_{\mu_t} (f) + \var \left( \int_{\RR^n} f\, d\mu_t  \right) 
\]
For the first term we proceed in the same way as before: by improved 
Lichnerowicz and since $f$ is $1$-Lipschitz, we have
\[ 
\EE \var_{\mu_t} (f) \leq \frac{\EE \Vert A_t\Vert_{op}^{1/2}} {\sqrt t} . 
\] 
For the second term, we proceed differently. 
The process $M_t= \int f \, d\mu_t$ is a 
martingale, whose derivative is 
\[ 
dM_t = \left( \int_{\RR^n} f(x) (x-a_t) \, d\mu_t \right)\cdot d W_t . 
\] 
Since $\Vert f\Vert_\infty^2 \lesssim C_P(\mu)$ we get from Cauchy-Schwarz 
\[
\left\vert \int_{\RR^n} f(x) (x-a_t) \, d\mu_t \right\vert^2 \lesssim C_P(\mu) \Vert A_t\Vert_{op} . 
\] 
Hence 
\[ 
\var \left( \int f\, d\mu_t  \right)
\lesssim C_P(\mu) \int_0^t \EE \Vert A_s\Vert_{op}\,ds . 
\] 
If $\EE \Vert A_t \Vert_{op} \lesssim 1$ up until time $t_0\lesssim 1$ we finally get 
\[ 
C_P (\mu) \lesssim t^{-1/2} + t \cdot C_P (\mu) , 
\] 
for all $t \leq t_0$, which indeed implies $C_P (\mu) \lesssim t_0^{-1/2}$.  
One thing that we can notice from this proof is that if 
we replace the improved Lichnerowicz inequality by the usual one, 
namely $C_P (\mu) \leq 1/t$ in the $t$-uniformly log-concave 
case, we also get something non trivial, namely 
\begin{equation}\label{eq_something}
C_P ( \mu ) \lesssim t_0^{-1}.
\end{equation} 
This is obviously a lot worse than what we get from improved Lichnerowicz, 
but still non trivial. As a matter of fact, all the aforementioned works on the KLS 
conjecture (prior to the latest one by Klartag 
in which the improved Lichnerowicz inequality is established) 
rely on this estimate, one way or another. 
The other argument to get Poincar\'e 
from the bound on the conditional covariance 
(see section~\ref{sec_gl})
may be more elementary and more natural in a way, but it only works if 
one happens to know the improved Lichnerowicz inequality. If 
you combine it with the usual Lichnerowicz inequality you get
nothing. 
\\
Now we define the constants $K_n$ and $S_n$ by 
\[ 
K_n 
= \sup \left\{ \left\Vert \sum_{i=1}^n (\EE X_i X^{\otimes 2})^2 \right\Vert_{op} \right\} , 
\quad S_n = \sup \left\{ \frac1n \var \vert X\vert^2  \right\} 
\] 
where both sup are taken over all log-concave 
isotropic random vectors on $\RR^n$. The constant $S_n$ is called 
the thin-shell constant. The thin-shell conjecture asserts that
the sequence $(S_n)$ is bounded. This is a weak form of the KLS 
conjecture as we only require Poincar\'e for a very specific 
function, namely the Euclidean norm squared. 
It was mentioned in the first section
in connection with the central limit problem for convex sets. 
A variant of what we have done above shows that 
in the isotropic log-concave case we have
$\EE \Vert A_t\Vert_{op} \lesssim 1$ up until times $(C K_n \log n)^{-1}$. 
From~\eqref{eq_something} we then obtain the following bound 
\[ 
C_n \lesssim  K_n \log n , 
\] 
for the KLS constant $C_n$. 
Moreover, by definition of $S_n$, given a log-concave 
and isotropic vector $X$ on $\RR^n$, 
a unit vector $u$, and an orthogonal projection $P$ of rank $k$,
we have $\EE (X\cdot u) \vert PX\vert^2 \leq \sqrt{k S_k}$. Applying 
this to suitable chosen projections $P$ one can estimate the 
eigenvalues of $\EE (X\cdot u) X^{\otimes 2}$ and then arrive at 
the bound 
\[ 
K_n \lesssim \sum_{k=1}^n \frac{ S_k }k \lesssim S_n \log n .  
\] 
We refer to~\cite{Eldan1} for the details. 
Altogether this gives 
\begin{equation}\label{eq_eldan}
C_n \lesssim  S_n (\log n)^2 . 
\end{equation} 
In other words thin-shell implies KLS up to polylog. This 
was the original result of Eldan. Note that Exercise~\ref{ex_thinshell} 
implies in particular that $S_n \geq 2$ for all $n$. As a result~\eqref{eq_eldan}
has become irrelevant now that we know that $C_n \lesssim \log n$.  
%Lastly, I'd like to mention 
%something that puzzles me. 
%
%Oddly enough, the more recent 
%bound $C_n = O (\log n)$ shows that it cannot be the case 
%that both inequalities 
%\[ 
%K_n \lesssim S_n \log n, \quad C_n \lesssim K_n \log n  
%\] 
%are sharp. On the other hand it doesn't imply an improvement 
%of any of the two individual inequalities. 
%We know they can't both be sharp but we don't know which one
%of the two is not.  Well, from improved Lichnerowicz we know that 
%\[ 
%C_n = O ( \sqrt{ K_n \log n} ) 
%\] 
%but this does not seem to help. 
%
\pagebreak

\section{Further localization results}

\subsection{An obstruction to a full solution of KLS}
As we have seen above, the KLS conjecture would be implied by 
the following statement:
in the isotropic log-concave case the expected operator norm 
of $\cov ( X \mid X + \sqrt s G)$ remains of order $1$ for 
all $s$. Unfortunately such an estimate cannot be true as we 
shall see now. 

Let $X=(X_1,\dotsc,X_n)$ be a random
vector whose coordinates are i.i.d. and such that 
$1+X_i$ is an exponential variable of parameter $1$. 
This is clearly an isotropic log-concave vector on $\RR^n$.
\begin{proposition} \label{prop_expo}
We have $\EE \Vert \cov (X\mid X + \sqrt s G) \Vert_{op} \leq C$ for all 
$s \geq C \log n$. On the other hand, if $s\leq \log n$ then 
$\EE \Vert \cov (X\mid X + \sqrt s G) \Vert_{op} \geq c s$. 
\end{proposition} 
Note that from the tensorization property of the Poincar\'e inequality (see exercise~\ref{exo_tens}), 
we have $C_P (X) = C_P (X_1)$. 
In particular the Poincar\'e constant of $X$ 
does not depend on $n$ and $X$ is not a counterexample to the KLS 
conjecture. 
Recall also that we always have the bound 
\[ 
\EE \Vert \cov (X\mid X + \sqrt s G) \Vert_{op} \leq s
\]
(acually this is true almost surely, not only in expectation). This 
example shows that this bound can be essentially sharp on a time range 
$[0,s_0]$ with $s_0 \to \infty$, namely $s_0 = \log n$. In particular
at time $s_0$ we have 
\[ 
\EE \Vert \cov (X\mid X + \sqrt{s_0} G) \Vert_{op} \geq c \log n 
\]
so the expected norm of the conditional covariance is not bounded 
for all times. In view of this example, the best one could hope for is 
\begin{equation}\label{eq_quarter} 
\EE \Vert \cov(X\mid X +\sqrt s G) \Vert_{op} \lesssim 1 , \quad \forall s \geq C \log n 
\end{equation} 
and for every log-concave isotropic $X$. 
Notice that there is still a gap between this and the bound that we obtained
in our theorem (in which $\log n$ is replaced by $\log^2 n$). If 
true the estimate~\eqref{eq_quarter} would imply the bound 
\[ 
C_n \lesssim  (\log n)^{1/4} , 
\]
for the KLS constant. This seems to be the limit one could reach within this framework. Going below this mark would have to rely on different arguments. 

The proof of the proposition only relies on some analysis in one 
dimension. Indeed since the coordinates of $X$ and $G$ are all independent it
is clear that the conditional law of $X$ given $X+\sqrt s G$ is just the 
$n$-fold product of the law of $X_1$ condition on $X_1 + \sqrt s G_1$. 
So the conditional covariance is diagonal with i.i.d. entries, and we 
just have to estimate the expected maximum of these. 

As a preliminary step, we need to compute the variance of a truncated Gaussian. 
\begin{lemma}  
Let $g$ be as standard Gaussian variable, then 
\[
\var ( g \mid g\geq x ) \approx \frac 1{1+x_+^2} .
\] 
\end{lemma}  
\begin{proof} 
Observe that the conditional law of $g$ given $g\geq x$ is log-concave. 
Let us use the "How to think about 1D log-concave measures'' proposition from 
the first section. It implies in particular that for log-concave
random variable $X$ having density $f$ we have 
\[
\Vert f \Vert_\infty^2  \cdot \var (X) \approx 1 ,  
\] 
Applying this to the conditional 
law of $g$ given $g\geq x$  (which is indeed log-concave), we get 
$\var (g \mid g \geq x) \approx \left( \int_{x}^\infty \e^{-y^2/2} \, dy \right)^{-2}$
if $x\leq 0$ and
\[ 
\var (g \mid g \geq x) \approx \left( \e^{x^2/2} \int_{x}^\infty \e^{-y^2/2} \, dy \right)^{-2}
\]
for every positive $x$. The result follows easily. 
\end{proof} 

\begin{proof}[Proof of Proposition~\ref{prop_expo}]
The conditional law of $X_1$ given $X_1 + \sqrt s G_1$ is just a 
truncated Gaussian. After some elementary computation we get 
\begin{equation}\label{eq_stepexpo}
\var (X_1 \mid X_1 + \sqrt s G_1 ) = s \cdot v ( \sqrt s - \frac 1{\sqrt s} Y_1 -G_1) 
\end{equation}
where $Y_1 = X_1 +1$ and $v$ is the function given by 
\[ 
v(x) = \var ( G_1 \mid G_1 \geq x ) .  
\] 
Note that $Y_1$ is an exponential variable independent of $G_1$, 
hence 
\[
\PP ( Y_1 \geq s, G_1 \geq 0 ) = \frac 12 \e^{ - s } . 
\]
Since the function $v$ is bounded away from $0$ on $\RR_-$, 
if $Y_1 \geq s$ and $G_1\geq 0$ then 
\[ 
\var (X_1 \mid X_1 + \sqrt s G_1 )  \geq cs .
\] 
As a result 
\[
\PP ( \var (X_1 \mid X_1 + \sqrt s G_1 )  \geq cs ) \geq \frac 1{2} \e^{-s}  .
\] 
By independence we get 
\[
\PP \left( \Vert \cov (X \mid X + \sqrt s G ) \Vert_{op}  \geq cs \right) 
\geq 1 - ( 1 - \frac 12 \e^{-s} )^n 
\] 
If $s \leq \log n$, the right-hand side is at least $1 -\e^{-1/2}$. 
By Markov's inequality, this implies that 
\[ 
\EE \Vert \cov (X \mid X + \sqrt s G )\Vert_{op} \geq c' . 
\] 
For the other inequality, since $v(x)\lesssim x^{-2}$ for large $x$, equation~\eqref{eq_stepexpo} and the union bound imply in particular that  
if $C$ is a sufficiently large constant
\[ 
\PP ( \vert \var(X_1 \mid X_1 + \sqrt s G_1 ) \vert \geq C ) \leq \PP ( Y_1 \geq \frac s{4} ) 
+ \PP ( G_1\geq \frac {\sqrt s} {4} ) \leq 2 \e^{-cs} .
\] 
Hence, by the union bound again, 
\[ 
\PP( \Vert \cov(X \mid X+ \sqrt s G ) \Vert_{op} \geq C ) \leq 2 n \e^{-cs} . 
\] 
Since $\Vert \cov(X \mid X+ \sqrt s G ) \Vert_{op} \leq s$ almost surely 
this implies 
\[
\EE\Vert \cov(X \mid X+ \sqrt s G ) \Vert_{op} \leq 2n s\cdot \e^{-cs} + C .  
\]
This becomes $O(1)$ as soon as $s$ exceeds a sufficiently large multiple of 
$\log n$. 
\end{proof}

\subsection{Concentration of measure}  

Recall from section~\ref{sec_cheeger} the definition of the concentration function of 
$\mu$:
\[ 
\alpha_\mu  (r) = \sup \left\{ 1-\mu (S_r) \colon \mu (S) \geq 1/2 \right\} .  
\]
We saw that log-concave measures satisfy 
exponential concentration and moreover than the Poincar\'e constant 
and the exponential concentration constant squared are 
of the same order. 
In particular the current best estimate for KLS amounts to 
the following 
\begin{equation}\label{eq_weaker} 
\alpha_\mu (r)  \leq 2 \exp \left( - c \cdot \frac{ r} { \sqrt{ \log  n} } \right).    
\end{equation} 
One of the points of this section is to show that 
one can go a bit beyond this estimate. 
\begin{theorem} \label{thm_concentration} 
If $\mu$ is log-concave and isotropic then its concentration function 
satisfies 
\begin{equation}\label{eq_better}
\alpha_\mu (r) \leq 2 \exp \left( - c \cdot \min \left( r , \frac{ r^2 }{\log^2 n} \right) \right)  , \quad \forall  r \geq 0. 
\end{equation} 
\end{theorem} 
We should make some comments on this result. 
First of all, the rate provided by Theorem~\ref{thm_concentration} is not smaller 
than that of~\eqref{eq_weaker} on the whole halfline. In particular 
combining with E. Milman's theorem would only lead to a 
$(\log n)^2$ bound for the Poincar\'e constant of an isotropic 
log-concave measure (rather than $\log n$).  That being said, the theorem yields in particular the rate $\e^{ - c r }$, which is predicted by the KLS conjecture, as soon as $r$ is larger than $\log^2 n$ or so. As far as we know, this information cannot be inferred from the bound $C_n\lesssim \log n$ alone.  
Let us also mention that one can prove the following 
variant of~\eqref{eq_better}, in which the concentration depends 
on the KLS constant $C_n$:
\begin{equation}
\label{eq_bizeul} 
\alpha_\mu (r) \leq 2 \exp \left( - c \cdot \min \left( r , \frac{ r^2 }{C_n \cdot \log n } \right) \right)  , \quad \forall  r \geq 0.
\end{equation} 
This inequality is taken from Bizeul~\cite{bizeulIHP}.

This concentration is reminiscent of the Gu\'edon-Milman estimate from 2011, see~\cite{GuMi}.    They proved that every isotropic log-concave measure $\mu$ satisfies 
\[ 
\mu \left( \vert \vert x \vert -\sqrt n \vert \geq r \right) \leq 2
\exp \left( -c \cdot \min \left( r , \frac{ r^2 }{n^{2/3}} \right) \right)  , \quad \forall  r \geq 0.
\]
This is weaker than~\eqref{eq_better} in two ways, first of all 
the constant is much worse 
($n^{2/3}$ vs $\log^2n$) and the deviation inequality is only 
for the Euclidean norm, and not for every $1$-Lipschitz function, as
in~\eqref{eq_better}. This application of stochastic localization 
to concentration dates back to Lee and Vempala~\cite{LV}. Their 
main result in that paper is the bound $C_n = O (n^{1/2})$ 
for the KLS constant, but they also obtain the inequality 
\begin{equation}\label{eq_LV1}
\alpha_\mu (r) \leq 2
\exp \left( - c \cdot \min \left( r , \frac{ r^2 }{n^{1/2}} \right) \right) . 
\end{equation}
In contrast with~\eqref{eq_bizeul}, they do not 
loose a logarithm when they pass from the bound on the KLS constant 
to the deviation inequality. Their argument is very delicate and 
clever but it only works with a polynomial estimate for $C_n$ and it 
does not allow to remove the logarithm from~\eqref{eq_bizeul} now that we have  
a logarithmic estimate for $C_n$.  

The proof of Theorem~\ref{thm_concentration} relies on the fact that 
uniformly log-concave measures satisfy Gaussian concentration. This was 
already mentioned in section~\ref{sec_conc}. 
%   nce again, the situation is well understood for 
%uniformly log-concave measures. They satisfy a  Gaussian 
%concentration property with constants not depending on the 
%dimension.   
%
\begin{proposition} \label{prop_Gconct}
Let $\mu$ be a $t$-uniformly log-concave measure. Then 
for every measurable set $S$ and every $r\geq 0$ we have 
\[
\mu (S) (1-\mu(S_r)) 
\leq \exp ( - c \cdot t r^2 )  , \quad \forall r \geq 0 ,
\]
where $c$ is a universal constant. In particular 
the Gaussian concentration constant of $\mu$ is $O ( t^{-1/2} )$. 
\end{proposition} 
\begin{proof} 
We give a short proof based on the Pr\'ekopa-Leindler inequality:
if $f,g,h$ are non negative functions on $\RR^n$ 
satisfying the inequality
\[ 
\sqrt{f(x)g(y)} \leq h \left( \frac{x+y}2 \right) 
\] 
for every $x,y\in\RR^n$ then 
\[ 
\sqrt{ \int_{\RR^n} f(x) \, dx \int_{\RR^n} g(x) \, dy } 
\leq \int_{\RR^n} h(x) \, dx. 
\]
If $\mu$ is $t$-uniformly log-concave 
then its potential $V$ satisfies 
\[ 
V \left( \frac{x+y}2 \right) \leq \frac{V(x)+V(y)}2 - \frac t8 \vert x-y\vert^2. 
\]  
Given a set $S$ and $\theta >0$, one can then see 
that the hypothesis of Pr\'ekopa-Leinder 
applies to the functions 
$f(x)=\mathbf 1_S(x) \e^{-V(x)}$, 
$g (y)= \e^{ \theta d (y,S) - V(y)}$ and 
$h = \e^{-V + 2 \theta^2/t}$.
From the conclusion of Pr\'ekopa, we get 
\[ 
\mu (S) \cdot \int_{\RR^n} \e^{ \theta d(x,S) } \, d\mu 
\leq \e^{ 2 \theta^2 / t } . 
\]  
The conclusion then follows from 
Chernoff inequality. 
\end{proof} 
One can be a bit more precise. As we already mentioned, in the Gaussian 
case we know the exact value of the concentration 
function $\alpha_{\gamma_n}$. 
Indeed, an integrated version of the isoperimetric 
inequality of Sudakov-Tsirelson / Borell asserts that $\gamma_n ( S_r)$ 
is maximized when $S$ is a halfspace. In particular 
\[ 
\alpha_{\gamma_n} (r) = 1 - \Phi (r) , \quad \forall r\geq 0.
\]
where $\Phi$ is the distribution function of the standard Gaussian variable.
Moreover, a deep result of Caffarelli asserts that a measure $\mu$ 
that is more log-concave than a given Gaussian measure is 
the image of that Gaussian by a $1$-Lipshitz map. Besides, it 
is clear that a pushforward by a contraction can only lower the concentration function. As a result if $\mu$ is $t$-uniformly log-concave 
then its concentration function is upper bounded by that 
of the Gaussian measure with covariance $\id/t$, namely we have   
\[
\alpha_\mu (r) \leq 1 - \Phi ( \sqrt t \cdot r ) , \quad \forall r\geq 0.
\]  
Since $1 - \Phi (r) \leq \frac 12 \e^{-r^2 / 2}$ for $r\geq 0$, this implies 
Gaussian concentration. However this only improves upon Proposition~\ref{prop_Gconct} at the level of the 
value of the universal constant $c$, which is irrelevant for our purposes. 

\begin{proof}[Proof of Theorem~\ref{thm_concentration}] 
Fix a set $S$ of measure $1/2$ and write  
\[ 
1-\mu ( S_r ) = \EE (1-\mu_t ( S_r )) \leq 
\EE (1-\mu_t ( S_r )) \mathbbm 1_{\{ \mu_t (S) \geq 1/4 \}} + 
\PP ( \mu_t ( S ) \leq 1/4 ) , 
\] 
where $(\mu_t)$ is the stochastic localization 
of $\mu$. 
Since $\mu_t$ is $t$-uniformly log-concave, 
the first term is at most $4 \e^{ - ctr^2}$, by
Proposition~\ref{prop_Gconct}. 
To handle the second term recall that the martingale 
$M_t:= \mu_t ( S)$ satisfies
\[ 
d M_t =  \left( \int_S (x-a_t) \, d\mu_t \right) \cdot dW_t . 
\]
Applying Cauchy-Schwarz we obtain 
\[
\left\vert \int_S (x-a_t) \, d\mu_t \right\vert^2 \leq \mu_t (S) \Vert A_t\Vert_{op} \leq \Vert A_t\Vert_{op}. 
\]
Hence the inequality
\[
\langle M \rangle_t \leq \int_0^t \Vert A_s \Vert_{op} \, ds . 
\]
In particular if $\Vert A_s\Vert_{op} \leq 2$ on $[0,t]$ then 
$\langle M \rangle_t \leq  2 t$. Therefore 
\[
\PP ( M_t \leq \frac14 ) \leq \PP ( M_t \leq \frac 14 \;\;\mbox{and} \;\; \langle M \rangle_t \leq 2t ) 
+ \PP ( \exists s \leq t \colon \Vert A_s\Vert_{op} \geq 2 ). 
\]
By Theorem~\ref{thm_At} from the previous section, the second term is at 
most $\exp ( -(Ct)^{-1} )$, provided $t\leq (C \log^2 n)^{-1}$.
On the other hand since $M_0 = \mu(S)= 1/2$, 
Freedman's inequality (Lemma~\ref{lem_freedman}) insures that 
\[ 
\PP ( M_t \leq \frac 14 \;\;\mbox{and} \;\;  \langle M \rangle_t \leq 2t ) 
\leq \exp \left( - \frac 1{C_1 t} \right).  
\] 
Putting everything together we get
\[ 
\mu( S_r^c  ) \leq 4 \exp(- c \cdot t r^2 ) + 2 \exp ( - (C_2 t)^{-1} ) 
\] 
for every $t\leq (C \cdot\log n)^{-2}$. Choosing $t = \min ( r^{-1} , (C\cdot \log n)^{-2} )$ yields 
\[ 
\mu(  S_r^c  ) \leq 
6 \exp\left( - c' \cdot \min \left( r , \frac{r^2}{\log^2 n} \right) \right) .
\]  
One can replace the prefactor $6$ by $2$ by changing a bit the constant $c'$ in the 
exponent. 
\end{proof}
Here is an example of an application of the theorem. 
\begin{corollary}[Paouris theorem] 
Suppose $\mu$ is log-concave and isotropic then 
\[ 
\mu ( \vert x \vert \geq r ) \leq \exp ( - cr ) , 
\quad \forall r \geq C \sqrt n, 
\]
where as usual $c,C$ are universal constants. 
\end{corollary} 
The inequality is due to Paouris~\cite{paouris}, see also ~\cite{ALLOPT} for another proof. 
The inequality can also be expressed in terms of moments. 
It asserts that if $X$ is log-concave and isotropic on $\RR^n$ then the moments of the Euclidean norm of $X$ remain constant for quite a 
while, namely 
\[ 
( \EE \vert X\vert^p )^{1/p} \approx ( \EE \vert X\vert^2 )^{1/2} 
\] 
for $p$ as large as $\sqrt n$. 
\begin{proof} We apply the concentration estimate to the 
$1$-Lipschitz function $f(x) = \vert x\vert$. We get in particular 
\[ 
\mu ( \vert x\vert \geq m + r ) \leq \e^{ - c r } , 
\]
provided that $r \geq C \cdot \log^2 n$, where $m$ is a median for 
$\vert x\vert$. 
Since $m \leq 2 \int \vert x\vert \, d\mu \leq 2 \sqrt n$, the latest display 
is easily seen to imply the desired inequality.
\end{proof} 
\begin{remark} As is apparent from the proof,  
inequality~\eqref{eq_better} is an overkill for this application, and 
the argument would go through using~\eqref{eq_LV1} instead. As a matter
of fact this application to the Paouris inequality is taken from Lee and Vempala's paper~\cite{LV}.
\end{remark} 
\subsection{Logarithmic Sobolev inequality and a variant of the KLS conjecture} 
Recall from section~\ref{sec_ls} that the log-Sobolev constant 
of a probability measure on $\RR^n$, denoted $C_{LS} (\mu)$ is 
the best constant in the inequality 
\[ 
D(\nu \mid \mu ) \leq \frac12  C_{LS} ( \mu) I ( \nu \mid \mu ) 
\] 
where $D$ and $I$ denote the relative entropy and Fisher information, respectively. 
Once again, the uniformly log-concave case is 
well understood. 
\begin{theorem}[Bakry-\'Emery criterion~\cite{BaEm}]
If $\mu$ is $t$-uniformly log-concave then $C_{LS} (\mu) \leq t^{-1}$. 
The inequality is sharp, equality is attained for the Gaussian measure 
of covariance $t^{-1} \cdot \id$.  
\end{theorem} 
There are many ways to prove this inequality, see for instance~\cite{ChaL}
for an overview. Again we are interested in the log-concave case. However, 
in contrast with the Poincar\'e inequality, not every log-concave measure 
satisfies log-Sobolev. Indeed, recall that log-Sobolev implies Gaussian 
concentration, with explicit control of the constants, and 
that this can be reversed in the log-concave case: 
the log-Sobolev constant and the Gaussian 
concentration constant are actually of the same order 
for log-concave measures.
To insure log-Sobolev, one has to impose 
another condition on top of log-concavity, such as having bounded support. 
The following result is due to Lee-Vempala~\cite{LV}. 
\begin{theorem}
Suppose $\mu$ is log-concave, isotropic, and supported 
on a set of diameter $D$. Then $C_{LS} (\mu) \lesssim D$. 
\end{theorem}
Let us remark that because of the equivalence between 
log-Sobolev and Gaussian concentration in the log-concave 
case, a log-concave measure supported on a set of diameter 
$D$ trivially has $O(D^2)$ log-Sobolev constant. Since the diameter 
of the support of an isotropic
measure is at least $\sqrt n$ the theorem improves
greatly upon the trivial bound in the isotropic case. It should also 
be noted that for the uniform measure on the 
$\ell_1$ ball rescaled to be isotropic, the diameter of the support 
and the log-Sobolev constant both are of order $n$. 
\begin{proof} 
This is actually an easy consequence of our concentration result Theorem~\ref{thm_concentration}. Indeed, the latter asserts that if $\mu$ is 
log concave and isotropic then 
\[ 
\alpha_\mu (r) \leq 2 \exp \left( -c \cdot \min ( r , \frac{ r^2} {\log^2 n} ) \right) , \quad \forall r \geq 0 . 
\]
On the other hand if $\mu$ is supported on a set of diameter $D$ then 
trivially $\alpha_\mu(r) = 0$ if $r>D$. On the interval $[0,D]$ we 
have $r \leq r^2 / D$, and since $D \geq \sqrt n \geq \log^2 n$, 
we finally obtain  
\[ 
\alpha_\mu (r) \leq 2 \exp \left( - c' \cdot \frac{ r^2} {D} \right)  . 
\] 
The Gaussian concentration constant is thus $O( D)$, 
which implies the desired inequality by E. Milman's result, 
Theorem~\ref{thm_milman_gauss}.
\end{proof} 
Let us try to relax the bounded support assumption. We know that log-Sobolev 
implies Gaussian concentration. In particular linear functions should have 
sub-Gaussian tails, at a rate controlled by the log-Sobolev constant. Let 
us be a bit more precise. 
\begin{definition} Suppose $f$ is a function having mean zero for $\mu$. 
We denote by $\Vert f \Vert_{\psi_2 (\mu)}$ the Orlicz norm of $f$ 
associated to the Orlicz function $\e^{r^2} -1$, namely 
the best constant $C$ in the inequality 
\[ 
\mu ( \vert f \vert \geq r ) \leq 2 \cdot \exp \left( - \frac{ r^2 } { C^2 } \right). 
\]
\end{definition} 
\begin{remark}
The usual definition of the Orlicz norm of $f$ associated to an Orlicz function $\phi$ 
and a measure $\mu$ is the smallest constant $C$ such that $\int \phi ( \vert f \vert / C ) \, d\mu \leq 1$. The above definition is slightly different but equivalent, up to universal constants. We chose this modified definition so that the connection with the notion of Gaussian concentration from section~\ref{sec_conc} becomes more apparent. 
\end{remark} 
The discussion above shows that for any probability measure and any direction 
$\theta$  we have 
\[ 
\Vert x\cdot \theta \Vert_{ \psi_2 (\mu) }^2 \lesssim C_{LS} (\mu) . 
\] 
It is natural to conjecture 
that this inequality could be reversed in the log-concave case. 
This amounts to saying that the log-Sobolev constant is 
the largest $\psi_2$-norm squared of a linear function, very much like the 
KLS conjecture predicts that the Poincar\'e constant 
of a log-concave measure is up to a constant the largest 
$L^2$-norm squared of a linear function. 
\begin{definition}[Log-Sobolev version of KLS constant, Bizeul~\cite{bizeul_LS}]
Let $D_n$ be the largest log-Sobolev constant of a
log-concave measure for which $\Vert x \cdot \theta \Vert_{\psi_2 (\mu)} \leq 1$ for every direction $\theta$. 
\end{definition} 
\begin{conjecture}[Log-Sobolev KLS conjecture, Bizeul~\cite{bizeul_LS}]
\[
D_n = O (1) . 
\]
\end{conjecture} 
It follows from some result of Bobkov~\cite{bobkov} from 2007 that 
$D_n = O (n )$. Using stochastic localization, one can show the 
following. 
\begin{theorem}[Bizeul~\cite{bizeul_LS}]
\[
D_n = O ( n^{1/2} ) . 
\] 
\end{theorem}
\begin{proof} 
The idea is to combine Theorem~\ref{thm_concentration}
with a rather crude net argument. Again, 
by E. Milman's theorem it is enough to prove that 
if $\mu$ log-concave is $\psi_2$ with norm 
at most $1$ in all directions, then its concentration function satisfies 
\begin{equation}
\label{eq_target} 
\alpha_\mu (r) \leq C \e^{ - c r^2 / \sqrt n } . 
\end{equation}
Note that the $\psi_2$ norm is larger than the $L^2$ norm, 
maybe up to a constant. So the covariance of $\mu$ 
has operator norm $O(1)$.  
The concentration function of $\mu$  
thus satisfies 
\[ 
\alpha_\mu (r) 
\leq 2 \exp \left( -c \cdot \min ( r ,  \frac{ r^2 }{ \log^2 n } ) \right) ,
\] 
for every $r>0$. Here there is a small gap which we can leave as an exercise: 
show that having an upper bound for the concentration function 
of every isotropic log-concave $\mu$ of the form $\alpha_\mu  \leq \alpha_*$ 
implies that if $\mu$ is log-concave but not necessarily isotropic, then 
$\alpha_\mu (r) \leq \alpha_* ( r / \sqrt{\Vert \cov(\mu) \Vert_{op}})$. 
We thus get an estimate that is smaller than our target 
concentration if $r > c r^2 / \sqrt n$, namely $r < C \sqrt n$.
Therefore, it is enough to prove~\eqref{eq_target} when 
$r$ is a sufficiently large multiple of $\sqrt n$. 
Moreover, by Markov's inequality we have 
\[ 
\mu ( \vert x\vert \geq 2 \sqrt n ) \leq 
\frac 1{4n} \int_{\RR^n} \vert x\vert^2 \,d \mu \leq \frac 14.  
\]  
So if $S$ is a set of measure $1/2$ then 
$S$ intersects the ball of radius $2 \sqrt n$. 
If $r \geq 2 \sqrt n$ this implies easily that
$S_{2r} \supset \{ \vert x \vert \leq r \}$, hence
\[
\alpha_\mu (2r) \leq \mu ( \vert x \vert > r  ) . 
\] 
So it is enough to prove that $\mu ( \vert x \vert > r  ) \leq \e^{-c r^2 / \sqrt n}$ for $r \geq C \sqrt n$. Now recall the $\psi_2$ hypothesis: 
For every direction $\theta$ and every $r$, we have 
\begin{equation}\label{eq_psi2step}
\mu \{ | x \cdot \theta | > r \} \leq 2 \e^{ - r^2 } . 
\end{equation}
It is well-known that there exists $1/2$-net of the unit sphere of
cardinality $5^n$ at most. Let $N$ be such a set. Since any 
element $x$ in the sphere is at distance $1/2$ at most from 
a point of $N$ we have 
\[ 
\vert x\vert \leq 2 \max_{\theta\in N} \{ x \cdot \theta \} , 
\]
for every $x\in \RR^n$. 
Applying~\eqref{eq_psi2step} to every $\theta$ in the net and 
the union bound we get 
\[ 
\mu ( \vert x\vert > r ) \leq 2 \cdot 5^n \e^{-r^2 /4 } . 
\]
If $r > C \sqrt n$ for a sufficiently large constant $C$, 
we deduce from this inequality 
\[ 
\mu ( \vert x\vert > r ) \leq  \e^{-r^2 /8 } 
\] 
which is even better than what we needed. 
\end{proof} 
\begin{remark} This proof seems to have lots of slack. 
It does not seem like the concentration result (Theorem~\ref{thm_concentration})
nor the  $\psi_2$ hypothesis were fully exploited. In particular, it should be noted
that the argument would go through with the weaker concentration estimate 
from Lee and Vempala~\eqref{eq_weaker}. Nevertheless, as far as the log-Sobolev version of the KLS conjecture is concerned this is the best result around, as of today. 
\end{remark} 
\pagebreak

\section{Bourgain's slicing problem}

Consider a centrally-symmetric convex body $K \subseteq \RR^n$ (i.e. $K = - K$). The maximal function operator associated 
with $K$, defined for $f: \RR^n \rightarrow \RR$ via
 $$ M_K f(x) = \sup_{r > 0} \int_K f(x + r y) \frac{dy}{Vol_n(K)}. $$
Bourgain  \cite{bou_max} proved that $\| M_K \|_{L^2(\RR^n) \to L^2(\RR^n)} \leq C$ for a universal constant $C > 0$. 
This led him to study on another question, seemingly innocent:

\begin{question} Let $n \geq 2$ and suppose that $K \subseteq \RR^n$ is a convex body of volume one. Does there exist a hyperplane $H \subseteq \RR^n$
	such that 
\begin{equation} Vol_{n-1}(K \cap H) > c \label{eq_1808} 
	\end{equation}
for a universal constant $c > 0$? \label{q1}
\end{question}

This question is still not completely answered, and in the last four decades it emerged as an ``engine'' for the development of the research direction 
discussed in these lectures. 
It is shown in \cite{K_root} that the bound (\ref{eq_1808})  holds true if we replace the universal constant $c$ on the right-hand 
side by $c / \sqrt{\log n}$. This is the currently best known result in the general case.

\begin{theorem}[Hensley \cite{hensley}, Fradelizi \cite{fradelizi}] 
Let $K \subseteq \RR^n$ be a convex body whose barycenter lies at the origin. Let $X$ be a random vector distributed uniformly in $K$, 
and assume that $\cov(X)$ is a scalar matrix. Then for any $\theta_1, \theta_2 \in S^{n-1}$,
	$$ Vol_{n-1}(K \cap \theta_1^{\perp}) \leq C \cdot Vol_{n-1}(K \cap \theta_2^{\perp})   $$
where $C > 0$ is a universal constant. In fact, $C \leq \sqrt{6}$.
\end{theorem}

\begin{proof} Let $\theta \in S^{n-1}$ and denote 
	$$ \sigma = \sqrt{ \EE (X \cdot \theta)^2 } = \sqrt{\cov(X) \theta \cdot \theta },  $$ which is independent of $\theta$. Write $$ \rho_{\theta}(t) = \frac{Vol_{n-1}(K \cap (t \theta + \theta^{\perp}))}{Vol_n(K)}, $$ the density of the random variable $X \cdot \theta$. By the Brunn-Minkowski inequality, $\rho_{\theta}$ is a log-concave probability density. 
	The log-concave random variable  $ X\cdot \theta / \sigma$ 	has mean zero and variance one, and its density is $x \mapsto \sigma \rho_{\theta}(x \sigma)$. 	According to Proposition \ref{thm_1526} above, for any $x \in\RR$,
	$$    c' 1_{ \{ |x| \leq c'' \} } \leq \sigma \rho_{\theta}(x \sigma) \leq C e^{-c |x|} $$
	In particular, $c \leq \rho_{\theta}(0) \leq C$, for some universal constants $c, C > 0$. 
\end{proof}

From this proof we may obtain a few more conclusions.  First, that among all hyperplane sections parallel to a given hyperplane, the hyperplane section through the barycenter has the largest volume, up to a multiplicative universal constant. Second, that when $K \subseteq \RR^n$ is a centered convex body of volume one, for any $\theta \in S^{n-1}$,
$$  Vol_{n-1}(K \cap \theta^{\perp}) \cdot \sqrt{\EE (X \cdot \theta)^2} \sim 1. $$
Here $\theta^{\perp} = \{ x \in \RR^n \, ; \, x \cdot \theta = 0\}$ and we abbreviate $A \sim B$ if $c \cdot A \leq B \leq C \cdot A$ for universal constants $c, C > 0$. This leads to the following conclusion:

\begin{corollary} Let $K \subseteq \RR^n$ be a convex body of volume one
	and let $X$ be a random vector distributed uniformly on $K$. Then,
	$$ \sup_{H} Vol_{n-1}(K \cap H) \sim \frac{1}{\sqrt{ \|\cov(X) \|_{op} }}, $$
where the supremum runs over all hyperplanes $H \subseteq \RR^n$. \label{cor_1614}
\end{corollary}

We thus see that Bourgain's slicing problem can be formulated as a question on the relation between the covariance of a convex body and its volume. 
Note that the logarithm of the volume of a convex body is the differential entropy of a random vector $X$ that is distributed uniformly over the convex body. 
In general, when the random vector $X$ has density $\rho$ in $\RR^n$, its differential entropy is 
$$ Ent(X) = -\int_{\RR^n} \rho \log \rho. $$

\begin{definition} For a convex body $K \subseteq \RR^n$ we define its {\it isotropic constant} to be
	$$ L_K = \left( \frac{\det\cov(K)}{Vol_n(K)^2} \right)^{\frac{1}{2n}} $$
	where $\cov(K)$ is the covariance matrix of the uniform probability distribution on $K$. More generally, the isotropic constant of an absolutely continuous, log-concave random vector $X$ in $\RR^n$ 
is
\begin{equation}  L_X = \left( \frac{\det \cov(X)}{e^{2 Ent(X)}} \right)^{\frac{1}{2n}}. \label{eq_1701} \end{equation}
\end{definition}

The isotropic constant of a convex body $K \subseteq \RR^n$ of volume one governs the volumes of its hyperplane sections. From Corollary \ref{cor_1614}
we see that when $Vol_n(K) = 1$, there always exists a hyperplane section $H \subseteq \RR^n$ with $$ Vol_{n-1}(K \cap H) \geq c / L_K. $$
Moreover, if we additionally assume that $\cov(K)$ is a scalar matrix, then for any hyperplane $H \subseteq \RR^n$ through the barycenter of $K$, 
$$ Vol_{n-1}(K \cap H) \sim \frac{1}{L_K}. $$
The slicing 
problem thus asks whether $L_K$ is universally bounded from above. 

\medskip 
\noindent
{\bf Remark on the definition of the isotropic constant in the log-concave case.} Some variants of this definition exist, sometimes 
one replaces $Ent(X)$ by $-\log \sup \rho$ or by $-\log \rho(\EE X)$ or by $2 \log \EE \rho^{-1/2}(X)$, where $\rho$ is the density of $X$. See for instance~\cite{BobMad}
where this is discussed in more details. 
These variants differ at most by a multiplicative universal constant, because of the following lemma:

\begin{lemma} Denoting by $\psi = -\log \rho$ the convex potential of $X$, we have
	$$ \psi(\EE X) \leq Ent(X) \leq \inf \psi + n $$
	and
	$$ \EE e^{ \frac{\psi(X)}{2} } \leq e^{\frac{\inf \psi}{2} + (\ln 2 ) n}. $$
	\label{lem_1654}
\end{lemma}

\begin{proof}   
We may assume that $\rho$ is continuous in $\RR^n$ in order to neglect boundary terms in the integration by parts below. Let $y \in \RR^n$. Then by Jensen's inequality and by the fact that any  convex function 
	lies above its tangent at $X$,
	$$ \psi(\EE X) \leq \EE \psi(X) = Ent(X) = \EE \psi(X) \leq \EE \left[ \psi(y) - \nabla \psi(X) \cdot (y-X) \right] = \psi(y) + n. $$
	Additionally,
	$$ \EE e^{ \frac{\psi(X)}{2} } = e^{\frac{\psi(y)}{2} } \int_{\RR^n} e^{-\frac{\psi(x) + \psi(y)}{2}} dx \leq e^{\frac{\psi(y)}{2} } \int_{\RR^n} e^{-\psi\left(\frac{x + y}{2}\right)} dx 
	= 2^n e^{\frac{\psi(y)}{2} } \int e^{-\psi} = 2^n e^{\frac{\psi(y)}{2} }.
	$$
The lemma follows by taking the infimum over all $y \in \RR^n$ in these two inequalities.
\end{proof}

\medskip
It what follows we work with the definition (\ref{eq_1701}). While here we are interested only in the log-concave case, the definition makes sense for any absolutely continuous random vector $X$ with finite second moments in $\RR^n$.
The isotropic constant measures the difference between two ways to measure the ``size'' of a random vector: its entropy and its covariance.
Here are some basic properties of the isotropic constant:

\begin{enumerate}
	\item It is an affine invariant, $L_{T(X)} = L_X$ for any invertible linear-affine map $T: \RR^n \rightarrow \RR^n$.
	
	\item If $X_1, X_2 \in \RR^n$ are independent log-concave random vectors, then for $X = (X_1,X_2) \in \RR^{2n} \cong \RR^n \times \RR^n$,
	$$ L_X = \sqrt{L_{X_1} L_{X_2}}. $$
	\item For any dimension $n$ and an absolutely continuous random vector $X$ with finite second moments in $\RR^n$,
$$ L_X \geq \frac{1}{\sqrt{2 \pi e}}, $$
with equality when $X$ is Gaussian. Indeed, this amounts to showing that among all random vectors with a fixed covariance matrix in $\RR^n$, the differential entropy is maximal for the Gaussian distribution, which is a standard fact. We recall the short proof below. 
\begin{proof} 
Suppose that $X$ is a random vector of mean zero, density $\rho$, and 
let $G$ be a centered Gaussian vector with the same covariance matrix as $X$. 
Also let $\gamma$ be the density of $G$. Then 
\[
\EE \log \left(\frac{\gamma(X)}{\rho(X)}\right)\leq \EE \frac{\gamma(X)}{\rho(X)} -1  = 0 .
\] 
Since $\log \gamma$ is a quadratic function and $X$ and $G$ have the same 
covariance matrix, we get
\[ 
Ent(X) = -\EE \log \rho (X) \leq -\EE \log \gamma (X) = - \EE \log \gamma (G) = Ent(G) . 
\qedhere . 
\] 
\end{proof} 	
\begin{exercise} Explain why it is not a coincidence that this universal constant $\sqrt{2 \pi e}$ is ``the same number'' from the asymptotics 
$Vol_n(\sqrt{n} B^n)^{1/n} \approx 1 / \sqrt{2 \pi e}$.
\end{exercise} 
\item Some examples:
$$ L_{[0,1]^n} = \frac{1}{\sqrt{12}}, \qquad \qquad L_{\Delta^n} = \frac{(n!)^{1/n}}{(n+1)^{(n+1)/(2n)}  \sqrt{n+2} } \approx \frac{1}{e}. $$
where $\Delta^n$ is a regular simplex in $\RR^n$.
\end{enumerate}

There are quite a few equivalent formulations and conditional statements, relating the isotropic constant to classical conjectures and results:
\begin{itemize}
\item If the isotropic constant is maximized for the cube among all centrally-symmetric convex set, then the Minkowski lattice conjecture follows, see Magazinov \cite{mag} and references therein. The Minkowski lattice conjecture suggests that if $L \subseteq \RR^n$ is a lattice of determinant one, then each of its translates intersects the set
$$ \left \{ x \in \RR^n \, ; \, \prod_{i=1}^n |x_i| \leq \frac{1}{2^n} \right \}. 
$$
This was proven in two dimensions by Minkowski in 1908.
\item If the isotropic constant is maximized for the simplex among all convex bodies, then the Mahler conjecture follows in the non-symmetric case. This conjecture suggests that among all convex bodies $K \subseteq \RR^n$,
the volume product 
$$ Vol_n(K) \cdot Vol_n(K^{\circ}) $$
is minimized when $K$ is a centered simplex \cite{K_adv}. This was proven in two dimensions by Mahler in 1908. Here 
$$ K^{\circ} = \left \{ x \in \RR^n \, ; \, \forall y \in K, \ x \cdot y \leq 1 \right \} $$
is the dual body. Recall that $(K^{\circ})^{\circ} = K$ when $K$ is a closed, convex set containing the origin. The Bourgain-Milman inequality resolves this conjecture up to a factor that is only exponential in the dimension. It states that for any convex body $K \subseteq \RR^n$ containing the origin,
$$ Vol_n(K) \cdot Vol_n(K^{\circ}) \geq (c/n)^n, $$
for a universal constant $c > 0$.

\item Suppose that $K \subseteq \RR^n$ is a convex body. Is there an ellipsoid $\cE \subseteq \RR^n$ with $Vol_n(\cE) = Vol_n(K)$ such that 
$$ Vol_n(K \cap C \cE) \geq \frac{1}{2} \cdot Vol_n(K) $$
where $C > 0$ is a universal constant? This is an equivalent formulation of the slicing problem.
\begin{exercise}
Prove the equivalence using reverse H\"older inequalities for quadratic polynomials. 
\end{exercise} 
For any convex body $K \subseteq \RR^n$, Milman's ellipsoid theorem~\cite{milmanELL}  provides an ellipsoid $\cE \subseteq \RR^n$ with $$ Vol_n(K \cap C \cE) \geq c^n \cdot Vol_n(K). $$
This suffices for developing the Milman ellipsoid theory, which contains the quotient of subspace theorem and reverse Brunn-Minkowski and the Bourgain-Milman inequality. 
See Pisier \cite{pisier_book} and references therein.
The slicing problem is a conjectural strengthening of Milman's ellipsoids.  
\end{itemize}

We move on to discuss the $\sqrt{\log}$-bound for the isotropic constant, and the relation to the Poincar\'e constant and the thin shell constants.
We define
$$ \sigma_n = \sup_X \sqrt{ \var(|X|^2) / n} $$
where the supremum ranges over all isotropic, log-concave random vectors $X$ in $\RR^n$. By reverse H\"older inequalities for polynomials
we may show that $\var(|X|^2) / n \sim\var(|X|)$, and hence $\sigma_n$ is roughly the maximal width of the thin spherical shell that captures most 
of the mass of an isotropic, log-concave random vector. 

\medskip 
From Corollary \ref{cor_1604} we know that, 
$$ \sigma_n \leq \sup_X \sqrt{ C_P(X) \cdot 4 \EE |X|^2 / n } \leq \sup_X 2 \sqrt{C_P(X)} \leq C \sqrt{ \log n}. $$
Hence it remains to prove:

\begin{theorem}[Eldan, Klartag~\cite{EK}] 
\label{thm_1610}
For any convex body $K \subseteq \RR^n$,
\[ 
L_K \leq C \sigma_n.
\] 
\end{theorem} 
\begin{remark} 
In fact, it is shown in \cite{EK} that $L_X \leq C \sigma_n$ for any log-concave random vector $X$ in $\RR^n$, but for simplicity we confine ourselves here for the convex body case. The slicing problem for convex bodies and for log-concave measures are known to be equivalent, as shown by Ball \cite{ball_studia, K_quarter}. 
\end{remark} 

While we studied Gaussian convolution in sections \ref{sec_gl} and \ref{sec_esl}, 
the proof of Theorem \ref{thm_1610} utilizes the closely related {\it Laplace transform}.
Let us fix an isotropic, log-concave random vector $X$ with density $\rho$ in $\RR^n$. Its logarithmic Laplace transform is 
$$ \Lambda(y) = \Lambda_X(y) = \log \EE e^{X \cdot y}. $$
 Since a log-concave random vector has exponential moments, the logarithmic Laplace transform is finite near the origin.
In fact, it is smooth in the open convex set $\Omega = \{ \Lambda  < \infty \}$. 
For $y \in \Omega$ we write $X_{y}$ for a random vector with density
$$ \rho_y(x) = \frac{\rho(x) e^{x \cdot y}}{e^{\Lambda(y)}}. $$
It is again a log-concave random vector, not necessarily isotropic, and we think of it as a {\it tilted} version of the random vector $X$.
We comment that it is possible to view tilts using projective transformations, 
this leads to the conditional statement that the strong slicing conjecture implies the Mahler conjecture, see \cite{K_adv}.

\begin{lemma} For any $y \in \Omega$,
$$ \nabla \Lambda(y) = \EE X_{y}, \qquad \nabla^2 \Lambda(y) =\cov(X_{y}), \qquad \nabla^3 \Lambda(y) = \EE (X_y - a_y)^{\otimes 3}, $$
where $a_y = \EE X_y$. 
\label{lem_1619}
\end{lemma}

Lemma \ref{lem_1619} is proven by direct computation; the logarithmic Laplace transform is the cumulant generating function. 
We see from Lemma \ref{lem_1619} that $\Lambda$ is convex, even strongly-convex as its Hessian is positive definite. In particular
the gradient $\nabla \Lambda: \Omega \rightarrow \RR^n$ is a one-to-one map. Consider the ``tilted determinant'' function 
$$ F(y) = \log \det \nabla^2 \Lambda(y) = \log \det \cov(X_y). $$
It measures how the determinant of the covariance matrix changes when we tilt the given distribution. Occasionally we may view $F$ 
as a function that is defined only up to an additive constant. Write $[F]$ for the equivalence class of $F$ under the equivalence relation ``$F$ is equivalent to $G$ if and only if $F - G$ is a constant function''.

\begin{lemma} The following bound holds pointwise in all of $\Omega$:
\begin{equation}  (\nabla^2 \Lambda)^{-1} \nabla F \cdot \nabla F \leq n \sigma_n^2. 
\label{eq_1900} \end{equation}
\label{lem_1443}
\end{lemma}

\begin{proof} Let us prove this bound first for $y = 0$ using the isotropicity of $X$. Recalling how to differentiate a determinant, we see that for any unit vector $v \in S^{n-1}$,
$$ \partial_v F (0) = \Tr \left[ (\nabla^2 \Lambda)^{-1}(0) \cdot \partial_v \nabla^2 \Lambda(0) \right] = \EE (X \cdot v) |X|^2
\leq \sqrt{ \EE (X \cdot v)^2  \cdot\var(|X|^2) } \leq \sqrt{n} \sigma_n. $$
By considering the supremum over all $v \in S^{n-1}$, we obtain the desired bound at $y = 0$. 

\medskip In order to obtain the bound for any $y \in \Omega$ we may either make a computation, or alternatively, think invariantly without computing anything, as we now explain.

\medskip 
Define a Riemannian metric on $\Omega$
via the Hessian of the log-Laplace transform $\Lambda$. We look at the Hessian metric $(\Omega, g)$, where the scalar product of two tangent vectors $u,v \in T_x \RR^n \cong \RR^n$ is
$$ g_x(u,v) = \nabla^2 \Lambda(x) u \cdot v. $$
The main observation is that the expression on the left-hand side of (\ref{eq_1900}) is the squared Riemannian length of the Riemannian gradient of the function $F: \Omega \rightarrow \RR$.
We say that 
$$ \cM_X = (\Omega, g, [F]) $$
is the ``Riemannian package'' associated with $X$. This means that $(\Omega,g)$ is a Riemannian manifold and that $F$ is a function on $\Omega$ modulo an additive constant.
An isomorphism between two Riemannian packages 
is a bijective map which is a Riemannian isometry and transforms correctly the function modulo the additive constant. 

\medskip What happens to the Riemannian package associated with $X$ when we do various operations?
\begin{itemize}
\item When we translate $X$, the Riemannian metric stays the same, as well as the function $F$. We get the same Riemannian package. 

\item Tilting $X$ and switching to $X_y$ yields an isomorphism of the two Riemannian packages by {\it translation} by $y$: We translate $\Omega, g$ and $[F]$ by the vector $y \in \Omega$. Any translation corresponds to a tilt and vice versa. 
\item Applying an invertible linear transformation to $X$ induces an isomorphism of the Riemannian packages. We apply a linear transformation and push forward $\Omega, g$ and $[F]$. (See also the paragraph before the next lemma).   
\end{itemize}
By the first and last items, we proved (\ref{eq_1900}) at the point $y = 0$ for any log-concave random vector (not necessarily centered or isotropic).
By the middle item, we proved (\ref{eq_1900}) also at all other points of $\Omega$. 
We refer to~\cite{EK} for a more detailed proof.  
\end{proof}

It makes sense to say that we think of $X$ as a random vector defined on an abstract affine space, rather than on $\RR^n$,
and observe that the Riemannian manifold $(\Omega, g)$ is well-defined, as well as the function $F: \Omega \rightarrow \RR$ modulo additive constants.
What can we say about balls in this Riemannian manifold? 

\begin{lemma} Assume that $X$ is a centered, log-concave random vector in $\RR^n$. Then for any $r > 0$,
$$ \frac{1}{2} \cdot \{ \Lambda \leq r \} \subseteq B_g(0, \sqrt{r}). $$
\label{lem_2257}
\end{lemma}

\begin{proof} Let $y \in \Omega$ satisfy $\Lambda(2 y) \leq r$. We need to find a curve from $0$ to $y$ whose Riemannian length is at most $r$. Let us try a line segment:
\begin{align*} Length_g([0,y]) & = \int_0^1 \sqrt{ \nabla^2 \Lambda(ty) y \cdot y } d t = \int_0^1 \sqrt{ \frac{d^2}{dt^2} \Lambda(ty) } d t 
\\ & \leq \sqrt{ \int_0^2 (2-t) \frac{d^2}{dt^2} \Lambda(ty)  d t \cdot \int_0^1 \frac{1}{2-t} dt} 
\\ & = \sqrt{\log 2} \cdot \sqrt{ \Lambda(2 y) - [\Lambda(0) + \nabla \Lambda(0) \cdot (2y)] } \\ 
& = \sqrt{\log 2} \cdot \sqrt{\Lambda(2y)} \leq \sqrt{r}. \qedhere
\end{align*}
\end{proof}

Let $X$ be an isotropic random vector in $\RR^n$, distributed uniformly in a convex body
$K \subseteq \RR^n$. We need two estimates for the proof of Theorem \ref{thm_1610}:

\begin{enumerate}
	\item[(i)] First, we need to show that for $r = n / \sigma_n^2$,
	$$ Vol_n(K) \geq e^{-n} \cdot Vol_n( B_g(0, \sqrt{r})  ) , $$
	the Euclidean volume of the Riemannian ball. This is related to mass transport in a simple case.
	\item[(ii)] Second, we need to show that 
	$$ Vol_n( \{ \Lambda \leq r \} )^{1/n} \geq c \frac{r}{n}  L_K. $$
	This is related to the Bourgain-Milman inequality.
\end{enumerate}

\begin{proof}[Proof of Theorem \ref{thm_1610}] Since $X$ is isotropic and log-concave, by (i), (ii) and Lemma \ref{lem_2257},
\[
\begin{split}  
L_K & = Vol_n(K)^{-1/n} \leq C \cdot Vol_n( B_g(0,\sqrt{r}) )^{-1/n} \\
& \leq 2C \cdot Vol_n( \{ \Lambda \leq r \})^{-1/n} \leq C' \frac{n}{r L_K} = C' \frac{\sigma_n^2}{L_K}. 
\end{split} 
\]
Thus $L_K \leq C'' \cdot  \sigma_n$.
\end{proof}

\begin{proof}[Proof of estimate (i):] The function $F$ vanishes at the origin, and by Lemma \ref{lem_1443} it is a Riemannian Lipschitz function with Lipschitz constant at most $\sqrt{n} \sigma_n$. Hence, 
	$$ |F| \leq n \qquad \textrm{in} \qquad B_g(0, \sqrt{r}). $$
	Consequently, for any $y \in B_g(0, \sqrt{r})$, 
	$$ e^{-n} \leq \det \nabla^2 \Lambda(y) \leq e^n. $$
	We will use the fact that $\nabla \Lambda(y) = \EE X_y \in K$ and that $y \mapsto \nabla \Lambda(y)$ is one-to-one. Changing variables, we obtain
\[ 
Vol_n(K) \geq Vol_n \left( \nabla \Lambda ( B_g(0, \sqrt{r}) ) \right) = 
\int_{B_g(0, \sqrt{r})} \det \nabla^2 \Lambda(y) \, dy \geq e^{-n} \cdot
	Vol_n( B_g(0, \sqrt{r}) ). \qedhere 
	\] 
\end{proof} 

\begin{proof}[Proof of estimate (ii):] For any $y \in r K^{\circ}$,
	$$ \Lambda(y) = \log \EE e^{y \cdot X} \leq \log(e^r) = r. $$
	Therefore,
	$$ \{ \Lambda \leq r \} \supseteq r K^{\circ}. $$
	By the Bourgain-Milman inequality,
\[ 
 Vol_n( \{ \Lambda \leq r \} )^{1/n} \geq Vol_n( r K^{\circ} )^{1/n} \geq c \frac{r}{n} Vol_n(K)^{-1/n} = c \frac{r}{n} L_K.  \qedhere 
 \]
\end{proof}

We remark that the Bourgain-Milman inequality has several proofs, and in particular it may be proven using more delicate analysis of the log-Laplace transform 
as shown by Giannopoulos, Paouris and Vritsiou \cite{GPV}.

\bibliography{lectures_IHP}

\begin{thebibliography}{10}

\bibitem{ALLOPT}
R.~Adamczak, R.~Lata{\l}a, A.~E. Litvak, K.~Oleszkiewicz, A.~Pajor, and
  N.~Tomczak-Jaegermann.
\newblock A short proof of {Paouris}' inequality.
\newblock {\em Can. Math. Bull.}, 57(1):3--8, 2014.

\bibitem{A}
G.~Allaire.
\newblock {\`A} la recherche de l'in{\'e}galit{\'e} perdue.
\newblock {\em Matapli}, 98:52--64, 2012.

\bibitem{Am}
L.~Ambrosio.
\newblock Lecture notes on optimal transport problems.
\newblock In {\em Mathematical aspects of evolving interfaces ({F}unchal,
  2000)}, volume 1812 of {\em Lecture Notes in Math.}, pages 1--52. Springer,
  Berlin, 2003.

\bibitem{AAGM}
S.~Artstein-Avidan, A.~Giannopoulos, and V.~D. Milman.
\newblock {\em Asymptotic geometric analysis. {Part} {II}}, volume 261 of {\em
  Math. Surv. Monogr.}
\newblock Providence, RI: American Mathematical Society (AMS), 2021.

\bibitem{BaEm}
D.~Bakry and M.~Emery.
\newblock Hypercontractivit{\'e} de semi-groupes de diffusion
  ({Hypercontractivity} for diffusion semi-groups).
\newblock {\em C. R. Acad. Sci., Paris, S{\'e}r. I}, 299:775--778, 1984.

\bibitem{BGL}
D.~Bakry, I.~Gentil, and M.~Ledoux.
\newblock {\em Analysis and geometry of {Markov} diffusion operators}, volume
  348 of {\em Grundlehren Math. Wiss.}
\newblock Cham: Springer, 2014.

\bibitem{ball_studia}
K.~Ball.
\newblock Logarithmically concave functions and sections of convex sets in
  {${\bf R}^n$}.
\newblock {\em Studia Math.}, 88(1):69--84, 1988.

\bibitem{Blegacy}
K.~Ball.
\newblock The legacy of {Jean} {Bourgain} in geometric functional analysis.
\newblock {\em Bull. Am. Math. Soc., New Ser.}, 58(2):205--223, 2021.

\bibitem{BN}
K.~Ball and V.~H. Nguyen.
\newblock Entropy jumps for isotropic log-concave random vectors and spectral
  gap.
\newblock {\em Studia Math.}, 213(1):81--96, 2012.

\bibitem{BBD}
R.~Bauerschmidt, T.~Bodineau, and B.~Dagallier.
\newblock Stochastic dynamics and the {Polchinski} equation: an introduction.
\newblock {\em Probab. Surv.}, 21:200--290, 2024.

\bibitem{bayle}
V.~Bayle.
\newblock {\em Propri\'et\'es de concavit\'e du profil isop\'erim\'etrique et
  applications}.
\newblock PhD thesis, Universit\'e Joseph Fourier Grenoble, 2003.

\bibitem{bizeul_LS}
P.~Bizeul.
\newblock On the log-{Sobolev} constant of log-concave measures.
\newblock Preprint, {arXiv}:2306.12997 [math.{FA}] (2023), 2023.

\bibitem{bizeulIHP}
P.~Bizeul.
\newblock On measures strongly log-concave on a subspace.
\newblock {\em Ann. Inst. Henri Poincar{\'e}, Probab. Stat.}, 60(2):1090--1100,
  2024.

\bibitem{bobkov_extremal}
S.~G. Bobkov.
\newblock Extremal properties of half-spaces for log-concave distributions.
\newblock {\em Ann. Probab.}, 24(1):35--48, 1996.

\bibitem{bobkov7}
S.~G. Bobkov.
\newblock Remarks on the growth of {$L^p$}-norms of polynomials.
\newblock In {\em Geometric aspects of functional analysis}, volume 1745 of
  {\em Lecture Notes in Math.}, pages 27--35. Springer, Berlin, 2000.

\bibitem{bobkov}
S.~G. Bobkov.
\newblock On isoperimetric constants for log-concave probability distributions.
\newblock In {\em Geometric aspects of functional analysis. Proceedings of the
  Israel seminar (GAFA) 2004-2005}, pages 81--88. Berlin: Springer, 2007.

\bibitem{BCG}
S.~G. Bobkov, G.~Chistyakov, and F.~G\"{o}tze.
\newblock {\em Concentration and {G}aussian approximation for randomized sums},
  volume 104 of {\em Probability Theory and Stochastic Modelling}.
\newblock Springer, 2023.

\bibitem{BoGL}
S.~G. Bobkov, I.~Gentil, and M.~Ledoux.
\newblock Hypercontractivity of {Hamilton}-{Jacobi} equations.
\newblock {\em J. Math. Pures Appl. (9)}, 80(7):669--696, 2001.

\bibitem{BH}
S.~G. Bobkov and C.~Houdr{\'e}.
\newblock Isoperimetric constants for product probability measures.
\newblock {\em Ann. Probab.}, 25(1):184--205, 1997.

\bibitem{BobMad}
S.~G. Bobkov and M.~Madiman.
\newblock The entropy per coordinate of a random vector is highly constrained
  under convexity conditions.
\newblock {\em IEEE Trans. Inf. Theory}, 57(8):4940--4954, 2011.

\bibitem{BKsurvey}
V.~I. Bogachev and A.~V. Kolesnikov.
\newblock The {Monge}-{Kantorovich} problem: achievements, connections, and
  perspectives.
\newblock {\em Russ. Math. Surv.}, 67(5):785--890, 2012.

\bibitem{borell-gauss}
C.~Borell.
\newblock The {Brunn}-{Minkowski} inequality in {Gauss} space.
\newblock {\em Invent. Math.}, 30:207--216, 1975.

\bibitem{bou_max}
J.~Bourgain.
\newblock On high-dimensional maximal functions associated to convex bodies.
\newblock {\em Amer. J. Math.}, 108(6):1467--1476, 1986.

\bibitem{BL}
H.~J. Brascamp and E.~H. Lieb.
\newblock On extensions of the {B}runn-{M}inkowski and {P}r\'{e}kopa-{L}eindler
  theorems, including inequalities for log concave functions, and with an
  application to the diffusion equation.
\newblock {\em J. Functional Analysis}, 22(4):366--389, 1976.

\bibitem{BGVV}
S.~Brazitikos, A.~Giannopoulos, P.~Valettas, and B.-H. Vritsiou.
\newblock {\em Geometry of isotropic convex bodies}, volume 196 of {\em Math.
  Surv. Monogr.}
\newblock Providence, RI: American Mathematical Society (AMS), 2014.

\bibitem{B}
P.~Buser.
\newblock A note on the isoperimetric constant.
\newblock {\em Ann. Sci. \'{E}cole Norm. Sup. (4)}, 15(2):213--230, 1982.

\bibitem{CFM}
L.~A. Caffarelli, M.~Feldman, and R.~J. McCann.
\newblock Constructing optimal maps for {M}onge's transport problem as a limit
  of strictly convex costs.
\newblock {\em J. Amer. Math. Soc.}, 15(1):1--26, 2002.

\bibitem{CaMo}
F.~Cavalletti and A.~Mondino.
\newblock Sharp and rigid isoperimetric inequalities in metric-measure spaces
  with lower {R}icci curvature bounds.
\newblock {\em Invent. Math.}, 208(3):803--849, 2017.

\bibitem{cayley}
A.~Cayley.
\newblock On {Monge}'s m{\'e}moire sur la th{\'e}orie des d{\'e}blais et des
  remblais.
\newblock {\em Proc. Lond. Math. Soc.}, 14:139--142, 1883.

\bibitem{ChaL}
D.~Chafa{\"{\i}} and J.~Lehec.
\newblock On {Poincar{\'e}} and logarithmic {Sobolev} inequalities for a class
  of singular {Gibbs} measures.
\newblock In {\em Geometric aspects of functional analysis. Israel seminar
  (GAFA) 2017--2019. Volume 1}, pages 219--246. Cham: Springer, 2020.

\bibitem{cheeger}
J.~Cheeger.
\newblock A lower bound for the smallest eigenvalue of the {L}aplacian.
\newblock {\em Problems in analysis, Princeton Univ. Press}, pages 195--199,
  1970.

\bibitem{chen}
Y.~Chen.
\newblock An almost constant lower bound of the isoperimetric coefficient in
  the {KLS} conjecture.
\newblock {\em Geom. Funct. Anal.}, 31(1):34--61, 2021.

\bibitem{DM}
N.~De~Ponti and A.~Mondino.
\newblock Sharp {C}heeger-{B}user type inequalities in
  {$\textrm{RCD}(K,\infty)$} spaces.
\newblock {\em J. Geom. Anal.}, 31(3):2416--2438, 2021.

\bibitem{DF}
P.~Diaconis and D.~Freedman.
\newblock Asymptotics of graphical projection pursuit.
\newblock {\em Ann. Statist.}, 12(3):793--815, 1984.

\bibitem{Eldan1}
R.~Eldan.
\newblock Thin shell implies spectral gap up to polylog via a stochastic
  localization scheme.
\newblock {\em Geom. Funct. Anal.}, 23(2):532--569, 2013.

\bibitem{EK}
R.~Eldan and B.~Klartag.
\newblock Approximately {G}aussian marginals and the hyperplane conjecture.
\newblock In {\em Concentration, functional inequalities and isoperimetry},
  volume 545 of {\em Contemp. Math.}, pages 55--68. Amer. Math. Soc.,
  Providence, RI, 2011.

\bibitem{EG}
L.~C. Evans and W.~Gangbo.
\newblock Differential equations methods for the {M}onge-{K}antorovich mass
  transfer problem.
\newblock {\em Mem. Amer. Math. Soc.}, 137(653):viii+66, 1999.

\bibitem{fradelizi}
M.~Fradelizi.
\newblock Hyperplane sections of convex bodies in isotropic position.
\newblock {\em Beitr\"{a}ge Algebra Geom.}, 40(1):163--183, 1999.

\bibitem{freedman}
D.~A. Freedman.
\newblock On tail probabilities for martingales.
\newblock {\em Ann. Probab.}, 3:100--118, 1975.

\bibitem{GPV}
A.~Giannopoulos, G.~Paouris, and B.-H. Vritsiou.
\newblock The isotropic position and the reverse {S}antal\'{o} inequality.
\newblock {\em Israel J. Math.}, 203(1):1--22, 2014.

\bibitem{GRS}
N.~Gozlan, C.~Roberto, and P.-M. Samson.
\newblock From concentration to logarithmic {Sobolev} and {Poincar{\'e}}
  inequalities.
\newblock {\em J. Funct. Anal.}, 260(5):1491--1522, 2011.

\bibitem{GM}
M.~Gromov and V.~D. Milman.
\newblock A topological application of the isoperimetric inequality.
\newblock {\em American Journal of Mathematics}, 105(4):843--854, 1983.

\bibitem{GuMi}
O.~Gu{\'e}don and E.~Milman.
\newblock Interpolating thin-shell and sharp large-deviation estimates for
  isotropic log-concave measures.
\newblock {\em Geom. Funct. Anal.}, 21(5):1043--1068, 2011.

\bibitem{hensley}
D.~Hensley.
\newblock Slicing convex bodies---bounds for slice area in terms of the body's
  covariance.
\newblock {\em Proc. Amer. Math. Soc.}, 79(4):619--625, 1980.

\bibitem{KLS}
R.~Kannan, L.~Lov{\'a}sz, and M.~Simonovits.
\newblock Isoperimetric problems for convex bodies and a localization lemma.
\newblock {\em Discrete \& Computational Geometry}, 13:541--559, 1995.

\bibitem{Kantorovich_Akilov}
L.~V. Kantorovich and G.~P. Akilov.
\newblock {\em Functional analysis}.
\newblock Pergamon Press, Oxford-Elmsford, N.Y., second edition, 1982.
\newblock Translated from the Russian by Howard L. Silcock.

\bibitem{K_quarter}
B.~Klartag.
\newblock On convex perturbations with a bounded isotropic constant.
\newblock {\em Geom. Funct. Anal.}, 16(6):1274--1290, 2006.

\bibitem{K_clt}
B.~Klartag.
\newblock A central limit theorem for convex sets.
\newblock {\em Invent. Math.}, 168(1):91--131, 2007.

\bibitem{K_uncond}
B.~Klartag.
\newblock A {B}erry-{E}sseen type inequality for convex bodies with an
  unconditional basis.
\newblock {\em Probab. Theory Related Fields}, 145(1-2):1--33, 2009.

\bibitem{K_needle}
B.~Klartag.
\newblock Needle decompositions in {R}iemannian geometry.
\newblock {\em Mem. Amer. Math. Soc.}, 249(1180):v+77, 2017.

\bibitem{K_complex}
B.~Klartag.
\newblock Eldan's stochastic localization and tubular neighborhoods of
  complex-analytic sets.
\newblock {\em J. Geom. Anal.}, 28(3):2008--2027, 2018.

\bibitem{K_adv}
B.~Klartag.
\newblock Isotropic constants and {M}ahler volumes.
\newblock {\em Adv. Math.}, 330:74--108, 2018.

\bibitem{K_chen}
B.~Klartag.
\newblock On {Y}uansi {C}hen's work on the {KLS} conjecture.
\newblock {\em Lecture notes prepared for a winter school at the Hausdorff
  Institute}, 2021.
\newblock Available at
  \url{https://www.him.uni-bonn.de/fileadmin/him/Lecture_Notes/chen_lecture_notes.pdf"}.

\bibitem{K_root}
B.~Klartag.
\newblock Logarithmic bounds for isoperimetry and slices of convex sets.
\newblock {\em Ars Inveniendi Analytica}, Paper No. 4, 17pp, 2023.

\bibitem{KL}
B.~Klartag and J.~Lehec.
\newblock Bourgain's slicing problem and {KLS} isoperimetry up to polylog.
\newblock {\em Geom. Funct. Anal.}, 32(5):1134--1159, 2022.

\bibitem{KO}
B.~Klartag and O.~Ordentlich.
\newblock The strong data processing inequality under the heat flow.
\newblock Preprint, {arXiv}:2406.03427 [cs.{IT}] (2024), 2024.

\bibitem{KP}
B.~Klartag and E.~Putterman.
\newblock Spectral monotonicity under {Gaussian} convolution.
\newblock {\em Ann. Fac. Sci. Toulouse, Math. (6)}, 32(5):939--967, 2023.

\bibitem{ledoux-buser}
M.~Ledoux.
\newblock A simple analytic proof of an inequality by {P}. {Buser}.
\newblock {\em Proc. Am. Math. Soc.}, 121(3):951--959, 1994.

\bibitem{ledoux_book}
M.~Ledoux.
\newblock {\em The concentration of measure phenomenon}, volume~89 of {\em
  Mathematical Surveys and Monographs}.
\newblock American Mathematical Society, Providence, RI, 2001.

\bibitem{ledouxSG}
M.~Ledoux.
\newblock From concentration to isoperimetry: semigroup proofs.
\newblock In {\em Concentration, functional inequalities and isoperimetry.
  Mainly the proceedings of the international workshop on concentration,
  functional inequalities and isoperimetry, Florida Atlantic University, Boca
  Raton, FL, USA, October 29--November 1, 2009.}, pages 155--166. Providence,
  RI: American Mathematical Society (AMS), 2011.

\bibitem{ledoux_semi}
M.~Ledoux.
\newblock From concentration to isoperimetry: semigroup proofs.
\newblock In {\em Concentration, functional inequalities and isoperimetry.
  %Mainly the proceedings of the international workshop on concentration,
  functional inequalities and isoperimetry, Florida Atlantic University, Boca
  Raton, FL, USA, October 29--November 1, 2009.}, pages 155--166. Providence,
  RI: American Mathematical Society (AMS), 2011.

\bibitem{LV_survey}
Y.~T. Lee and S.~S. Vempala.
\newblock The {K}annan-{L}ov\'asz-{S}imonovits conjecture.
\newblock In {\em Current developments in mathematics 2017}, pages 1--36. Int.
  Press, Somerville, MA, 2019.

\bibitem{LV}
Y.~T. Lee and S.~S. Vempala.
\newblock Eldan's stochastic localization and the {KLS} conjecture:
  {I}soperimetry, concentration and mixing.
\newblock {\em Ann. of Math. (2)}, 199(3):1043--1092, 2024.

\bibitem{Lich}
A.~Lichnerowicz.
\newblock {\em G\'{e}om\'{e}trie des groupes de transformations}, volume III of
  {\em Travaux et Recherches Math\'{e}matiques}.
\newblock Dunod, Paris, 1958.

\bibitem{rogers}
T.~Lindvall and L.~C.~G. Rogers.
\newblock Coupling of multidimensional diffusions by reflection.
\newblock {\em Ann. Probab.}, 14:860--872, 1986.

\bibitem{mag}
A.~Magazinov.
\newblock A proof of a conjecture by {H}aviv, {L}yubashevsky and {R}egev on the
  second moment of a lattice {V}oronoi cell.
\newblock {\em Adv. Geom.}, 20(1):117--120, 2020.

\bibitem{Emil}
E.~Milman.
\newblock On the role of convexity in isoperimetry, spectral gap and
  concentration.
\newblock {\em Invent. Math.}, 177(1):1--43, 2009.

\bibitem{Emil2}
E.~Milman.
\newblock Isoperimetric and concentration inequalities: equivalence under
  curvature lower bound.
\newblock {\em Duke Math. J.}, 154(2):207--239, 2010.

\bibitem{milmanELL}
V.~D. Milman.
\newblock {An} inverse form of the {Brunn}-{Minkowski} inequality with
  applications to local theory of normed spaces.
\newblock {\em C. R. Acad. Sci., Paris, S{\'e}r. I}, 302:25--28, 1986.

\bibitem{ohta}
S.-I. Ohta.
\newblock Needle decompositions and isoperimetric inequalities in {F}insler
  geometry.
\newblock {\em J. Math. Soc. Japan}, 70(2):651--693, 2018.

\bibitem{OV}
F.~Otto and C.~Villani.
\newblock Generalization of an inequality by {Talagrand} and links with the
  logarithmic {Sobolev} inequality.
\newblock {\em J. Funct. Anal.}, 173(2):361--400, 2000.

\bibitem{paouris}
G.~Paouris.
\newblock Concentration of mass on convex bodies.
\newblock {\em Geom. Funct. Anal.}, 16(5):1021--1049, 2006.

\bibitem{PW}
L.~E. Payne and H.~F. Weinberger.
\newblock An optimal {P}oincar\'e inequality for convex domains.
\newblock {\em Arch. Rational Mech. Anal.}, 5:286--292, 1960.

\bibitem{pisier_book}
G.~Pisier.
\newblock {\em The volume of convex bodies and {B}anach space geometry},
  volume~94 of {\em Cambridge Tracts in Mathematics}.
\newblock Cambridge University Press, Cambridge, 1989.

\bibitem{roc_thm}
R.~T. Rockafellar.
\newblock Characterization of the subdifferentials of convex functions.
\newblock {\em Pacific J. Math.}, 17:497--510, 1966.

\bibitem{RW}
L.~C.~G. Rogers and D.~Williams.
\newblock {\em Diffusions, {Markov} processes, and martingales. {Vol}. 2:
  {It{\^o}} calculus.}
\newblock Cambridge: Cambridge University Press, 2nd ed. edition, 2000.

\bibitem{rothaus}
O.~S. Rothaus.
\newblock Analytic inequalities, isoperimetric inequalities and logarithmic
  {Sobolev} inequalities.
\newblock {\em J. Funct. Anal.}, 64:296--313, 1985.

\bibitem{sudakov}
V.~N. Sudakov.
\newblock Typical distributions of linear functionals in finite-dimensional
  spaces of high dimension.
\newblock {\em Dokl. Akad. Nauk SSSR}, 243(6):1402--1405, 1978.
\newblock English translation: Soviet Math. Dokl. 19 (1978), no. 6, 1578–1582
  (1979).

\bibitem{ST}
V.~N. Sudakov and B.~S. Cirel'son.
\newblock Extremal properties of half-spaces for spherically invariant
  measures.
\newblock {\em J. Sov. Math.}, 9:9--18, 1978.

\bibitem{SW1}
G.~Szeg\"o.
\newblock Inequalities for certain eigenvalues of a membrane of given area.
\newblock {\em J. Rational Mech. Anal.}, 3:343--356, 1954.

\bibitem{Talag}
M.~Talagrand.
\newblock Transportation cost for {Gaussian} and other product measures.
\newblock {\em Geom. Funct. Anal.}, 6(3):587--600, 1996.

\bibitem{varo}
N.~Th. Varopoulos.
\newblock Small time {Gaussian} estimates of heat diffusion kernels. {I}: {The}
  semigroup technique.
\newblock {\em Bull. Sci. Math., II. S{\'e}r.}, 113(3):253--277, 1989.

\bibitem{SW2}
H.F. Weinberger.
\newblock An isoperimetric inequality for the $n$-dimensional free membrane
  problem.
\newblock {\em J. Rational Mech. Anal.}, 5:633--636, 1956.

\end{thebibliography}

\end{document}